\documentclass[times]{TRR}
\usepackage{textcomp}
\usepackage[brazil, english]{babel}
\usepackage{array}
\usepackage{amsmath, amssymb, amsthm, bm, amsfonts}
\usepackage[linesnumbered, ruled]{algorithm2e}
\usepackage{caption}
\usepackage{float} 
\usepackage{subcaption}
\usepackage{xcolor}
\usepackage{stfloats}
\usepackage{subfiles}
\usepackage{textcomp}
\usepackage{stfloats}
\usepackage{graphicx}
\usepackage{indentfirst}
\usepackage{pifont}
\usepackage{makecell}
\usepackage{diagbox}
\usepackage{mathtools}
\usepackage{stmaryrd}
\RequirePackage{CJKnumb}
\DeclareMathSizes{10}{10}{7}{5}
\usepackage{color, url, balance, times, helvet, courier}
\usepackage{booktabs, threeparttable, multirow, adjustbox, threeparttable}
\newtheorem{assumption}{Assumption}
\newtheorem{proposition}{Proposition}
\newtheorem{definition}{Definition}
\newtheorem{Property}{Property}
\newtheorem{theorem}{Theorem}

\theoremstyle{remark}
\newtheorem{remark}{Remark}
\usepackage{moreverb,url}
\usepackage[colorlinks,bookmarksopen,bookmarksnumbered,citecolor=blue,urlcolor=blue,linkcolor=blue]{hyperref}
\usepackage[switch]{lineno}

\newcommand\BibTeX{{\rmfamily B\kern-.05em \textsc{i\kern-.025em b}\kern-.08em
T\kern-.1667em\lower.7ex\hbox{E}\kern-.125emX}}

\begin{document}

\runninghead{Weifeng Yang}

\title{An Inertial Block Proximal Linearized Method with Adaptive Momentum for Nonconvex and Nonsmooth Optimization}

\author{Weifeng Yang\affilnum{1}}

\affiliation{\affilnum{1}Yunnan Earthquake Administration, Kunming 650224, China}
\corrauth{Weifeng Yang, ywf841673182@gmail.com}

\begin{abstract}
In this paper, we consider a class of multiblock nonconvex nonsmooth optimization problems, which covers many applications such as the analysis of pre-earthquake anomalies and machine learning. 
To solve this class of problems, we propose the inertial block proximal linearized method with two-phase adaptive momentum (IBPL$^+$-TP). 
Compared to the current methods, our method possesses three main advantages: 
(1) it introduces a two-phase adaptive momentum strategy to effectively update the extrapolation parameters, 
(2) it allows using two different extrapolation points to accelerate the convergence, 
(3) it allows the extrapolation parameters of these two extrapolation points to be independent of and unconstrained by all other parameters.    
While maintaining the above advantages, we prove that our method ensures the monotonic convergence of the objective function of this class of problems, and we also prove that the sequence generated by our method globally converges to a critical point, as well as establish the convergence rate of our method.  
To demonstrate the effectiveness of our method, we apply it to solve two nonconvex and nonsmooth machine learning problems, namely sparse nonnegative matrix factorization with $\ell_0$-constraints and sparse nonnegative CP decomposition with $\ell_0$-constraints. 
The numerical experimental results on solving these problems show that our method outperforms several state-of-the-art methods. 
\end{abstract}


\maketitle

\section{Introduction}
\label{sec:introduction}
In this paper, we consider a class of nonconvex and nonsmooth optimization problems which can be expressed as follows. 
\begin{align}
    \min\limits_{(\left\{{x_{i}} \right\}^{N}_{i=1})}J(\left\{{x_{i}} \right\}^{N}_{i=1})=H(\left\{{x_{i}} \right\}^{N}_{i=1})+\sum_{i=1}^{N}F_{i}(x_{i}),\label{e11}
\end{align}
where $d_{i} \in \mathbb{N}, ~H: \mathbb{D}_{H}  \rightarrow\mathbb{R},~\mathbb{D}_{H} \subseteq \prod_{i=1}^N \mathbb{R}^{d_{i}}, ~F_{i}:\mathbb{D}_{F_{i}} \rightarrow \mathbb{R},~\mathbb{D}_{F_{i}} \subseteq \mathbb{R}^{d_{i}},~ \mathrm{dom} ~ J=\mathbb{D}_{H}\cap (\cap_{i=1}^{N} \mathbb{D}_{F_{i}})$. $F_{i}$ is a proper, lower semicontinuous (possibly nonconvex and nonsmooth) function (e.g.,  $\ell_0$ norm), $H$ is a continuously differentiable (possibly nonconvex) function. 
Many remarkable application problems can be modeled as Eq. (\ref{e11}), e.g., analysis of pre-earthquake abnormal data \cite{zhu2021analysis,fan2022analysis}, regression problems \cite{huang2021prediction,chang2023hybrid}, and tensor decomposition \cite{zhang2025noisy,yang2026graph}, etc.

To solve the problem presented in Eq. (\ref{e11}), the commonly used methods include the Proximal Alternating Minimization (PAM) method \cite{attouch2010proximal,dang2025two} and the Smoothing method \cite{li2024smoothing,chen2012smoothing}. 
The drawback of the PAM method is that it incurs substantial computational cost to obtain the closed-form solutions of the problems, and the drawback of the Smoothing method is that it requires more restrictive assumptions about Eq. (\ref{e11}) and can only obtain the approximate solutions. 
To address these issues, the proximal alternating linear minimization (PALM) method \cite{bolte2014proximal} proposed the block proximal linear operator, i.e., 
\begin{align}
    && x^{k+1}_{j} &\in \operatorname*{\arg\min}\limits_{x} \Big[F_{j}(x)+\frac{1}{2\sigma^{k}_{j}} \|x-x^{k}_{j}\|^{2}\notag\\
    && & +\langle  x-x^{k}_{j}, \nabla_{x_{j}} H(\left\{{x^{k+1}_{i}} \right\}^{j-1}_{i=1},x_{j}^{k},\left\{{x^{k}_{i}} \right\}^{N}_{i=j+1}) \rangle \Big],
    \label{e12}
\end{align}
where $L_{\nabla_{x_{j}} H}$ is the Lipschitz constant of $\nabla_{x_{j}} H(\left\{{x_{i}} \right\}^{N}_{i=1})$ and $\sigma^{k}_{j}\in (0,\frac{1}{L_{\nabla_{x_{j}} H}})$ is the step size. 
By utilizing Eq. (\ref{e12}), the PALM method ensures the convergence and global convergence of Eq. (\ref{e11}).

The method of using inertial terms (extrapolation terms) to accelerate convergence is widely employed in various optimization problems. 
In order for this inertial PALM method to also ensure the convergence and global convergence of Eq. (\ref{e11}), \cite{pock2016inertial} proposed to use an auxiliary function to replace Eq. (\ref{e11}), and to maintain the constraint relationship between the extrapolation parameters and other parameters (i.e.,  the value range of the extrapolation parameter is constrained by the value of other parameters). 
This constraint relationship and this replacement eliminate the impact of the extrapolation terms in the convergence analysis, thereby making this inertial PALM method applicable to the proof framework proposed by \cite{bolte2014proximal}. 
Based on this proof technique, subsequent researchers have also proposed numerous inertial PALM methods. 

However, as outlined above, the methods mentioned above struggle to ensure the monotonic decrease of Eq. (\ref{e11}) during the iteration process since these methods require an auxiliary function to replace Eq. (\ref{e11}), and non-monotonic decrease will weaken the numerical performance \cite{xu2015alternating,yang2023accelerated}. 
Moreover, the methods mentioned above are necessary to maintain the constraint relationship between extrapolation parameters and other parameters, such as step size. This constraint relationship eliminates the impact of the extrapolation terms in the convergence analysis, but it makes the selection and update of the extrapolation parameters extremely complex and difficult even for a single extrapolation point, let alone when using two different extrapolation points. 
These issues obviously diminish the numerical effectiveness of these methods and make it extremely difficult to extend the applicability of these methods to the case where $N>2$. 
Although \cite{yang2023accelerated} proposed a method that ensures the independence of the extrapolation parameter from other parameters, this method is restricted to the case of using a single extrapolation point and cannot be extended to the case of using two different extrapolation points, and this method does not employ a two-phase adaptive momentum strategy to separate the rapid descent phase from the steady refinement phase, thereby reducing its flexibility and numerical performance.

To address the above-mentioned issues, we propose the method named the inertial block proximal linearized method with two-phase adaptive momentum (IBPL$^+$-TP) for solving Eq. (\ref{e11}). IBPL$^+$-TP can be viewed as an inertial version of PALM with the restart step and two different extrapolation points, 
but we propose a modified proof framework, based on the proof framework proposed by \cite{bolte2014proximal}, to demonstrate that IBPL$^+$-TP can still guarantee the monotonic convergence and global convergence properties of Eq. (\ref{e11}) even when the extrapolation parameters of these two different extrapolation points are independent of and unconstrained by any other parameters. 
This relaxed condition and the corresponding scheme of using two different extrapolation points to add inertial force can lead to significantly better numerical performance and substantially reduce the difficulty of updating the extrapolation parameters.  
Furthermore, we develop a two-phase adaptive momentum strategy for the extrapolation parameters of the proposed method, which explicitly separates the rapid descent phase from the steady refinement phase to effectively update the extrapolation parameters, thereby accelerating convergence and enhancing numerical performance. 
Sparse nonnegative matrix factorization (SNMF) and sparse nonnegative CP decomposition (SNCP) are both the powerful feature extraction tools widely used in dimensionality reduction and analysis of multi-dimensional data \cite{liu2025inertial}, 
but SNMF with $\ell_0$ norm constraints and SNCP with $\ell_0$ norm constraints are both NP-hard, nonsmooth and nonconvex problems,  we will apply our proposed method to solve these problems. 
The contributions of this paper are as follows.

\subsection{Contributions}
(1) We propose the inertial block proximal linearized method with two-phase adaptive momentum (IBPL$^+$-TP) for solving Eq. (\ref{e11}). The proposed method provides an algorithm framework that not only employs two different extrapolation points but also allows the extrapolation parameters of these extrapolation points to have a value range that does not depend on the values of any other parameters. 
Additionally, we introduce a two-phase adaptive momentum strategy for the extrapolation parameters of the proposed method, which explicitly separates the rapid descent phase from the steady refinement phase,  thereby effectively updating the extrapolation parameters and accelerating convergence.  
These advantages significantly improve both the flexibility and effectiveness of IBPL$^+$-TP and enhance the numerical performance.

(2) While maintaining the above advantages, we prove that our method ensures the monotonic convergence of Eq. (\ref{e11}).  Additionally, we also prove that the sequence generated by our method is globally convergent to a critical point and establish the convergence rate of our method.

(3) We apply our method to solve the SNMF with $\ell_0$-constraints and the SNCP  with $\ell_0$-constraints problems. The numerical experimental results on solving these problems demonstrate the superior performance and efficiency of our proposed method.  We also investigate both the numerical stability and parameter sensitivity of the proposed method.


\section{Notations and mathematical definitions}
\label{sec2}
In this section, we introduce some symbol definitions and the preliminary knowledge that will be used throughout this paper. We define $\mathbb{R}$ as the set of all real numbers, $\mathbb{N}$ as the set of all natural numbers. 
Additionally, Table \ref{notation} summarizes the notation used in this paper. 

\begin{table}[!ht] 
\centering 
\caption{Summary of frequently-used notations}
\begin{tabular}{p{2.7cm}|p{5cm}} \hline%
Notation & Definition \\ \hline
$\left\{{x_{i}} \right\}^{n}_{i=1}$ & $\{x_{1}, x_{2}. . . . . . , x_{n}\}$ \\
$a\left\{{x_{i}} \right\}^{n}_{i=1} ~(a \in \mathbb{R})$ & $\left\{{ax_{i}} \right\}^{n}_{i=1}$ \\
$\left[ n \right]$  & $\left\{i\right\}_{i=1}^{n}$ \\

$x_{(n)}$  & $\left\{{x_{i}} \right\}^{n}_{i=1}$ \\ 
$x_{(n)}+y_{(n)}$ &	$\left\{{x_{i}}+{y_{i}}\right\}^{n}_{i=1}$ \\ 

$ax_{(n)}~(a\in \mathbb{R})$ &	$\left\{{a}{x_{i}}\right\}^{n}_{i=1}$ \\ 
$\frac{x_{(n)}}{a}~(a\in \mathbb{R})$ &	$\left\{{\frac{x_{i}}{a}}\right\}^{n}_{i=1}$ \\ 

$x_{(n)}\times y_{(n)}$ &	$\left\{{x_{i}}{y_{i}}\right\}^{n}_{i=1}$ \\ 
$\frac{x_{(n)}}{ y_{(n)}}$ &	$\left\{{\frac{x_{i}}{y_{i}}}\right\}^{n}_{i=1}$ \\

$L_{\nabla_{x_{j}} H}$  & the Lipschitz constant of the $\nabla_{x_{j}} H(\left\{{x_{i}} \right\}^{N}_{i=1})$.\\ 
$x^{k}_{i}$  & the $i$-th block of $\left\{{x_{i}} \right\}^{n}_{i=1}$ within the $k$-th outer loop\\ 
$h_{j}(x_{j})$ &$H(\{x_{i}^{k+1}\}_{i=1}^{j-1}, x_{j}, \{x_{i}^{k}\}_{i=j+1}^{N}), $ \\
$\mathcal{C}^{p}_{L}(X)$ & the set of functions satisfying the Lipschitz continuous  property \\
$\langle \cdot , \cdot  \rangle$ & inner product \\
$\|\cdot \|_{p}$ & $l_{p}$ norm \\
$\|\left\{{x_{i}}\right\}^{n}_{i=1}\|$ & $ \sum_{i=1}^{n}\|x_{i}\|$\\ 

$\mathrm{dom} ~ J$ & the domain of function $J$ \\ \hline
\end{tabular}
\label{notation}
\end{table}

\subsection{Symbol definition}
We first introduce some definitions and properties \cite{ciarlet2013linear,armstrong2013basic}. 
\begin{definition} 
Let $X$ be a non-empty set, and $(X,d)$ is a topological space.  $d: X \times X \rightarrow \mathbb{R}$ is called a metric function on $X$ if, for all $x, y, z \in X$, the following conditions hold. 

(i) $d(x, y) \geq 0$, with equality if and only if $x = y$.

(ii) $d(x, y) = d(y, x)$ for all $x, y \in X$.

(iii) $d(x, y) + d(y, z) \geq d(x, z)$.

\noindent Then $(X,d)$ is a metric space, we abbreviate it as $X$. 
\end{definition}

\begin{definition}
$B(x,\epsilon)$ is an open ball which is defined as $B(x,\epsilon):=\left\{y: d(x,y)<\epsilon\right\}$. 
\label{openball}
\end{definition}

\begin{definition}
If $A$ is a set, then define the derived set $A'$ as the set of overall cluster points of $A$, $A^{o}$ as the set of overall inter points of $A$, $\overline{A}$ as the closure of $A$, $A^{c}$ as the complement of $A$.
\label{d1}
\end{definition}
\begin{proposition}
Let $A$ be a closed set, then the derived set $A' \subseteq A$. 
\label{dopen}  
\end{proposition}

\begin{definition}
Let $A$ be a set, if
$(B(x,\epsilon)-\left\{x\right\})\cap A \neq \emptyset ~({\forall}\epsilon>0)$, 
then $x\in A'$. 
\label{dclu}
\end{definition}
\begin{proposition}
A bounded closed set in finite-dimensional space is a compact set. 
\label{dcompact}
\label{p4}
\end{proposition}

\begin{proposition}
In metric space, if 
\begin{center}
 ${\forall}x_{n}, ~x _{0} \in \mathbb{R}^{n}, ~ \lim_{n \rightarrow \infty} x_{n}=x_{0}$,   
\end{center}
then we have 
\begin{center}
$\lim_{x\rightarrow x_{0}}f(x)=f(x_{0}) \Leftrightarrow \lim_{n \rightarrow \infty}f(x_{n})=f(x_{0})$. 
\end{center}
\label{tf4}
\end{proposition}
\begin{definition}
For $S \subseteq \mathbb{R}^{n}$, we define the distance between the set $S$ and the point $x$ as $\operatorname{dist}(x, S): =\operatorname*{inf}\,\left\{\|x-y\|: \,y\in S\right\}. $
\end{definition}

\subsection{Notation and preliminaries for nonconvex analysis}
Next, we introduce some preliminaries for nonconvex analysis \cite{bolte2014proximal,ciarlet2013linear}.  
\begin{definition}
Proper function: a function $g: \mathbb{R}^{n} \rightarrow (-\infty,+\infty]$ is said to be proper if $\mathrm{dom}$ $g \neq \emptyset$, where $\mathrm{dom} ~g =\left\{x \in \mathbb{R}:~ g(x)<\infty \right\}$. 
\label{dproper}
\end{definition}
\begin{definition}
Lower semicontinuous function: if  ${\forall} x_{k} \in \mathrm{dom}~ f$, we have
\begin{center}
$\lim_{k\rightarrow \infty} x_{k}=x, ~ f(x) \leq \operatorname*{lim}\operatorname*{inf}_{k\to\infty}f(x_{k}) ~$,  
\end{center}
then $f$ is called lower semicontinuous at $\mathrm{dom}~ f$. 
\label{lower}
\end{definition}
\begin{definition}
Coercive function: if $f$ is coercive, then $\left\{x | x \in R^{n}, ~ f(x) < a, ~{\forall} a \in \mathbb{R} \right\}$ is bounded and $\inf_{x} f(x) > -\infty$.
\label{dlower}
\end{definition}
\begin{definition}
Let $f$ be a proper lower semicontinuous function, The Fréchet subdifferential of $f$ at $x$, written ${\hat{\partial}} f(x)$, is the set of all vectors u which satisfy  
\begin{center}
$\operatorname*{lim}_{y\neq x,y\to x}\cdot\frac{f(y)-f(x)-\langle u,\ y-x\rangle}{\|y-x\|}\ge0, $
\end{center}
when $ x \notin \mathrm{dom}~ f$, then set ${\hat{\partial}} f(x)=\emptyset$.
\label{d2}
\label{dsubdiff}
\end{definition}
\begin{definition}
The limiting subdifferential  $\partial f(x): =\{u\in\mathbb{R}^{n}: \exists x^{k}\to x,f(x^{k})\to f(x),u^{k}\to u,u^{k}\in\widehat{\partial}f(x^{k})\}.$
\end{definition}
\begin{proposition}
Let $f$ be a proper lower semicontinuous function. If $f$ has a local minimum at $x^{*}$, then $0 \in \partial f(x^{*})$. 
\label{p1}
\end{proposition}
\begin{proposition} 
Let $f$ be a proper lower semicontinuous function, and $g$ be a continuously differential function. Then $\forall x \in \mathrm{dom} ~f$, $\partial(f+g)(x)$ = $\partial f (x) +\nabla g(x)$. 
\label{p2}
\end{proposition}

\begin{definition}
Set\ $f$: $X\rightarrow \mathbb{R}, ~ X \subseteq \mathbb{R}^{n}$, if $f\in \mathcal{C}_{L_i}^P(X)$, then $\forall x_i \in X_i$ ($X_i$ is the $x_i$-section of $X$, $f_i$ is the $x_i$-section of $f$), $\exists U \subset X_i$ is the neighborhood of $x_i$ and $\exists L_{\nabla_{x_{i}} f}>0$, $\forall y_i,z_i \in U$, s.t. 
   \begin{equation}
  \|\nabla_{x_{i}} f_i(z_i)-\nabla_{x_{i}} f_i(y_i)\| \leq L_{\nabla_{x_{i}} f}\|z_i-y_i\|, \notag   
   \end{equation}
where $L_{\nabla_{x_{i}} f}$ is the block-wise local Lipschitz constant of $\nabla_{x_{i}} f$. $\mathcal{C}_{L_i}^P(X)$ is also known as a set of functions satisfying the block-wise local gradient Lipschitz continuity. 
   \label{dlipchitz}
   \label{d3}
\end{definition}

\begin{proposition}
Set\ $f$: $\mathbb{R}^{n}\rightarrow \mathbb{R}$, if  $f \in\mathcal{C}_{L_i}^P(\mathbb{R}^{n})$, and $U \subset X_i$ is the neighborhood of $x_i$, then $\forall y_i\in U$, s.t. 
 $f_i(y_i) \leq f_i(x_i)+\langle \nabla f_i(x_i),y_i-x_i \rangle+\frac{L_{\nabla_{x_{i}} f}}{2}\|y_i-x_i\|^{2}.$
\label{p3}
\end{proposition}

\begin{definition}
Desingularization function: we denote the desingularization function $\phi: [0,\eta)\rightarrow \mathbb{R}_{+}$ which satisfies the following conditions. 

(i) $\phi(0)=0$.

(ii) $\phi$ is continuously differentiable on $[0,\eta)$, and $\phi$ is a concave function from in $(0,\eta)$.

(iii)  $\phi (s) '>0 ~ (\forall s \in (0,\eta))$. 
\label{des}
\end{definition}

The Kurdyka-Łojasiewicz (KŁ) property \cite{attouch2013convergence,xu2013block}, described below, serves as a key tool for global convergence analysis and establishing convergence rate.  
\begin{definition}
Let $f$ be a proper lower semicontinuous function, $f$ is said to have the KŁ property on $\bar{u}\in d o m(\partial f) $, if there exists $\eta\in\left(0, +\infty\right]$, and $U$ is the neighborhood of $\bar{u}$, any $u$ in $U$ that satisfies the condition $f({\bar{u}}) < f(u) < f({\bar{u}})+\eta$ has the following inequalities established.  
\begin{center}
$\phi^{\prime}(f(u)-f(\bar{u}))d i s t(0,\partial f(u))> 1$,
\end{center}
where $\phi$ is the desingularization function. If $f$ has the KŁ property at each point of $\mathrm{dom}$ $\partial f$, then $f$ is a  KŁ function. 
\label{def_KL}
\end{definition}
\begin{proposition}
(Uniformized KŁ property) Let $\Omega$ be a compact set and $f$:  $\mathbb{R}^{n} \rightarrow \mathbb{R} \cup \left\{\infty\right\}$ be a proper and lower semicontinuous function. Assume that $f$ is constant on $\Omega$ and is a KŁ function. Then, there exist $\epsilon>0,~
\eta>0$ such that ${\forall}x \in \Omega$ and all $x$ in the intersection $\left\{x \in \mathbb{R}^{n} : dist(x,\Omega) < \epsilon \right\} \cap \left\{f (\bar{x})< f (x)< f (\bar{x}) + \eta \right\}$ one has 
\begin{center}
$\phi^{\prime}(f(x)-f(\bar{x}))d i s t(0,\partial f(x))> 1$.       
\end{center}
\label{def_uniKL}
\end{proposition}

\vspace{2pt}

\section{The proposed IBPL$^+$ and IBPL$^+$-TP methods}
\label{sec3}
In this section, we present the technical details of our proposed IBPL$^+$ and IBPL$^+$-TP methods. 

First, we make the following standard assumptions about Eq. (\ref{e11}) throughout this paper.

\begin{assumption}
(i) Define $\mathbb{D}_{H} \subseteq \prod_{i=1}^N \mathbb{R}^{d_{i}}$, $H: \mathbb{D}_{H} \rightarrow\mathbb{R} $, $H \in \mathcal{C}_{L_i}^P(\mathbb{D}_{H})$ (see Definition \ref{dlipchitz}). 

(ii) Define $\mathbb{D}_{F_{i}} \subseteq \mathbb{R}^{d_{i}}$, $F_{i}: \mathbb{D}_{F_{i}} \rightarrow \mathbb{R}$ is a lower semicontinuous function (see Definition \ref{dlower}). 

(iii) $J$ is a Kurdyka-Łojasiewicz (KŁ) function (see Definition \ref{def_KL}), and $J$ is also a proper coercive function (see Definition \ref{dproper}) with a lower bound. 
\label{assump1}
\end{assumption}

Then, let $S=\mathrm{dom} ~ J$,  for Eq. (\ref{e11}), we define the proximal operator with two different extrapolation points in our proposed method as follows.  
\begin{align}
\hspace{-3.1mm} &&x^{k+1}_{i} &\in prox_{\sigma^{k}_{i} F_{i}} (z^{k}_{i}-\sigma^{k}_{i} \nabla_{x_{i}}  H(\{x^{k+1}_{j}\}_{j=1}^{i-1},y^k_{i},\{x^{k}_{j}\}_{j=i+1}^{N})) \notag \\
\hspace{-3.1mm} && &= \operatorname*{\arg\min} \limits_{x \in S} \Big[F_{i}(x)+\frac{1}{2\sigma^{k}_{i}}\|x -z^{k}_{i}\|^{2} \notag \\
\hspace{-3.1mm} && &+\langle \nabla_{x_{i}} H(\{x^{k+1}_{j}\}_{j=1}^{i-1},y^k_{i},\{x^{k}_{j}\}_{j=i+1}^{N}), x-z^{k}_{i} \rangle \Big].  \label{iprox} 
\end{align}
Next, we first present our method with adaptive momentum: IBPL$^+$ (Algorithm \ref{IBPL$^+$}), which incorporates an intuitive and practical adaptive momentum strategy to update the extrapolation parameters, as follows. 

\begin{footnotesize}
 \begin{algorithm}[!h]
\caption{IBPL$^+$: Inertial block proximal linearized method with adaptive momentum}\label{IBPL$^+$}
    \SetAlgoLined
     \LinesNumbered
        \KwIn{$\{{x^{1}_{i}} \}^{N}_{i=1}=\{{x^{0}_{i}} \}^{N}_{i=1} \in \mathrm{dom} \ J$, $\rho_{1}>0$, 
        $t_{1}, t_{2}\in[1, \infty)$, $\alpha_{max},\beta_{max}\in[0,1)$, $\alpha^{1}=\{\alpha_{i}^{1}\in [0, 1)\}^{N}_{i=1}$, $ \beta^{1}=\{\beta_{i}^{1}\in [0, 1)\}^{N}_{i=1}$}
        
        
        \For {$k = 1,2,3,\cdots$ }{
        \hspace{-1.1mm}$\{{z^{k}_{i}} \}^{N}_{i=1}$=$\{{x^{k}_{i}} \}^{N}_{i=1}$+$\alpha^{k}\times (\{{x^{k}_{i}} \}^{N}_{i=1}-\{{x^{k-1}_{i}} \}^{N}_{i=1})$, \\

        \hspace{-1.1mm}$\{{y^{k}_{i}} \}^{N}_{i=1}$=$\{{x^{k}_{i}} \}^{N}_{i=1}$+$\beta^{k}\times (\{{x^{k}_{i}} \}^{N}_{i=1}-\{{x^{k-1}_{i}} \}^{N}_{i=1})$. \\

        \For {$i = 1$ to N}{

        \hspace{-1.5mm}$\gamma^{k}_{i}\in (1,\infty), ~\sigma_{i}^{k} =\frac{1}{\gamma^{k}_{i} (L_{\nabla_{x^{k}_{i}} H})}$, 
         
        \hspace{-1.5mm} $h_{i}(y^{k}_{i})=H(\{x^{k+1}_{j}\}_{j=1}^{i-1},y^k_{i},\{x^{k}_{j}\}_{j=i+1}^{N})$, 
         
        \hspace{-1.5mm} $x^{k+1}_{i} \in prox_{\sigma_{i}^{k} F_{i}} (z^{k}_{i}-\sigma_{i}^{k} \nabla_{x_{i}} h_{i}(y^{k}_{i})).$

        }

       \eIf{\begin{equation}\hspace{-2.6mm} J(x^{k+1}_{(N)}) \leq J(x^{k}_{(N)})-\rho_{1}(\|x^{k+1}_{(N)}-y^{k}_{(N)}\|^2+\|x^{k+1}_{(N)}-z^{k}_{(N)}\|^2)\notag\end{equation}} 
        {\hspace{-2.1mm} $\alpha^{k+1}=\min(t_{1}\alpha^{k},\alpha_{max})$, 
        
        \hspace{-1.5mm}$\beta^{k+1}=\min(t_{2}\beta^{k},\beta_{max})$. }
        {\hspace{-2.1mm} $z^{k}_{(N)}=y^{k}_{(N)}=x^{k}_{(N)}$, $\alpha^{k+1}=\frac{\alpha^{k}}{t_{1}},~\beta^{k+1}=\frac{\beta^{k}}{t_{2}}$, 
        
        \hspace{-2.1mm} and re-update $x^{k+1}_{(N)}$ through step 4 to 8.

        }

        }
        \textbf{Return} $\left\{{x^{k+1}_{i}} \right\}^{N}_{i=1}$.
\end{algorithm}
\end{footnotesize}

\begin{remark}
(\romannumeral1) Since the convergence proof of our method only requires $\alpha^{k}<1$ and $\beta^{k}<1$, IBPL$^+$ still converges even if the extrapolation parameters are updated via other strategies.  

(\romannumeral2) When the extrapolation parameters are not adaptively updated, we refer to IBPL$^+$ as IBPL. 

(\romannumeral3) When $\alpha_{i}^{k}\equiv\beta_{i}^{k}\equiv 0$, IBPL$^+$ degenerates into PALM. 
\end{remark}

\begin{table*}[!h]
\centering
\renewcommand\arraystretch{3}
\belowrulesep=0pt\aboverulesep=0pt
\caption{The extrapolation points used in the methods and the constraint relationships between extrapolation parameters and other parameters in the methods. To maintain notational consistency with Algorithm \ref{IBPL$^+$} and Algorithm \ref{IBPLtp}, the extrapolation parameters of the extrapolation points in the methods using two different extrapolation points are also denoted as $\alpha^k_{i}$ and $\beta^k_{i}$, respectively, and the extrapolation parameter of the extrapolation point in the methods using a single extrapolation point is denoted as $\beta^k_{i}$. } 
\centering
\begin{tabular*}{0.85\linewidth}{c|c|c} 
\toprule     
Method & Extrapolation points   & Constraint relationships     \\ 
\midrule   
iPALM \cite{pock2016inertial}&      $z^{k}_{i}$ and $y^{k}_{i}$     & $\frac{1}{\sigma^{k}_{i}}=\frac{1+2\beta_{i}^{k}}{1-2\alpha_{i}^{k}} L_{\nabla_{x^{k}_{i}} H}$ \\ \hline

IBPG \cite{le2020inertial} &   $z^{k}_{i}$ and $y^{k}_{i}$   & $\frac{v(1-v)(\gamma_{i}^{k}-1)^2 L_{\nabla_{x^{k-1}_{i}} H}}{(\beta_{i}^{k}+\gamma_{i}^{k}\alpha_{i}^{k})^2L_{\nabla_{x^{k}_{i}} H}}>1$, where $v \in (0,1),\gamma_{i}^{k}>1$   \\ \hline

BPL \cite{xu2017globally}&   $y^{k}_{i}$  &$\beta_{i}^{k} < \frac{\gamma_{i}^{k}-1}{2(\gamma_{i}^{k}+1)}\sqrt{\frac{L_{\nabla_{x^{k-1}_{i}} H}}{L_{\nabla_{x^{k}_{i}} H}}}$ \\ \hline

TITAN \cite{phan2023inertial} &    $y^{k}_{i}$   & $\beta_{i}^{k} < \frac{\gamma_{i}^{k}-1}{2\gamma_{i}^{k}}\sqrt{\frac{L_{\nabla_{x^{k-1}_{i}} H}}{L_{\nabla_{x^{k}_{i}} H}}}$ \\ \hline

ABPL \cite{yang2023accelerated}& $y^{k}_{i}$  &  $\beta_{i}^{k}<1$ \\ \hline

PGels \cite{yang2024proximal} &  $y^{k}_{i}$  &  $\beta_{i}^k<\sqrt{\frac{\delta\mu^k_{i}(\mu^k_{i}-L_{\nabla_{x^{k}_{i}} H})}{4(\mu^k_{i}+L_{\nabla_{x^{k}_{i}} H})^2}}$, where $\mu^k_{i}\geq \frac{L_{\nabla_{x^{k}_{i}} H}+2c}{1-\delta},~\delta \in (0,1),~c>0$. 
\\ \hline 
\end{tabular*}
\label{paraselect}
\end{table*}


Furthermore, since we can prove in the next section that our proposed method can ensure the monotonic convergence for arbitrary values of the extrapolation parameters while the global convergence of our proposed method only requires the extrapolation parameters to be bounded above by a constant $C$ ($C<1$) after a finite number of iterations, we can leverage this theoretical insight to design a two-phase adaptive strategy for the proposed method. 
Specifically, this two-phase adaptive strategy allows the extrapolation parameters to take arbitrarily large values (even exceeding $1$) in the initial rapid descent phase to accelerate the convergence. Once the iterates stabilize, we impose the upper bounds $\alpha_{max},\beta_{max}\in[0,1)$ on the extrapolation parameters to ensure global convergence. By incorporating this two-phase adaptive strategy into Algorithm \ref{IBPL$^+$}, we propose the IBPL$^+$-TP algorithm (i.e., Algorithm \ref{IBPLtp}). This design can explicitly separate the rapid descent phase from the steady refinement phase while tailoring the adaptive rule to each phase, thereby enhancing the numerical performance.

\begin{footnotesize}

\begin{algorithm}[!h]
\caption{IBPL$^+$-TP: Inertial block proximal linearized method with two-phase adaptive momentum}\label{IBPLtp}
    \SetAlgoLined
     \LinesNumbered
        \KwIn{$\{{x^{1}_{i}} \}^{N}_{i=1}=\{{x^{0}_{i}} \}^{N}_{i=1} \in \mathrm{dom} \ J$, $\rho_{1}>0$, 
        $t_{1}, t_{2}\in[1, \infty)$, $\alpha_{rapid},\beta_{rapid}\in [0,\infty)$, $\alpha_{max},\beta_{max}\in[0,1)$, $\alpha^{1}=\{\alpha_{i}^{1}\in [0, 1)\}^{N}_{i=1}$, $ \beta^{1}=\{\beta_{i}^{1}\in [0, 1)\}^{N}_{i=1}$, $T\in [0, \infty)$. $\overline{\alpha}=\alpha_{rapid},~\overline{\beta}=\beta_{rapid}$.}
        

        \For {$k = 1,2,3,\cdots$ }{
        \hspace{-1.1mm}$\{{z^{k}_{i}} \}^{N}_{i=1}$=$\{{x^{k}_{i}} \}^{N}_{i=1}$+$\alpha^{k}\times (\{{x^{k}_{i}} \}^{N}_{i=1}-\{{x^{k-1}_{i}} \}^{N}_{i=1})$, \\

        \hspace{-1.1mm}$\{{y^{k}_{i}} \}^{N}_{i=1}$=$\{{x^{k}_{i}} \}^{N}_{i=1}$+$\beta^{k}\times (\{{x^{k}_{i}} \}^{N}_{i=1}-\{{x^{k-1}_{i}} \}^{N}_{i=1})$. \\

        \For {$i = 1$ to N}{

        \hspace{-1.5mm}$\gamma^{k}_{i}\in (1,\infty), ~\sigma_{i}^{k} =\frac{1}{\gamma^{k}_{i} (L_{\nabla_{x^{k}_{i}} H})}$, 
         
        \hspace{-1.5mm} $h_{i}(y^{k}_{i})=H(\{x^{k+1}_{j}\}_{j=1}^{i-1},y^k_{i},\{x^{k}_{j}\}_{j=i+1}^{N})$, 
         
        \hspace{-1.5mm} $x^{k+1}_{i} \in prox_{\sigma_{i}^{k} F_{i}} (z^{k}_{i}-\sigma_{i}^{k} \nabla_{x_{i}} h_{i}(y^{k}_{i})).$

        }

       \eIf{\begin{align}\hspace{-2.6mm} && J(x^{k+1}_{(N)}) &\leq J(x^{k}_{(N)}) \notag \\
       && &-\rho_{1}(\|x^{k+1}_{(N)}-y^{k}_{(N)}\|^2+\|x^{k+1}_{(N)} 
      -z^{k}_{(N)}\|^2)\notag\end{align}} 
        {\hspace{-2.1mm} $\alpha^{k+1}=\min(t_{1}\alpha^{k},\overline{\alpha})$, 
        
        \hspace{-1.5mm}$\beta^{k+1}=\min(t_{2}\beta^{k},\overline{\beta})$. }
        {\hspace{-2.1mm} $z^{k}_{(N)}=y^{k}_{(N)}=x^{k}_{(N)}$, $\alpha^{k+1}=\frac{\alpha^{k}}{t_{1}},~\beta^{k+1}=\frac{\beta^{k}}{t_{2}}$, 
        
        \hspace{-2.1mm} and re-update $x^{k+1}_{(N)}$ through step 4 to 8.

        }

        }

        \If{$\frac{|J(x^{k+1}_{(N)})-J(x^{k}_{(N)})|}{|J(x^{0}_{(N)})|}< T$} 
        {$\overline{\alpha}=\alpha_{max}$, $\overline{\beta}=\beta_{max}$. }

        \textbf{Return} $\left\{{x^{k+1}_{i}} \right\}^{N}_{i=1}$.

\end{algorithm}
\end{footnotesize}

\begin{remark}
(\romannumeral1) When $T=\infty$, IBPL$^+$-TP degenerates into IBPL$^+$.



\label{remark1}
\end{remark}

As outlined in the Introduction, IBPL$^+$ and IBPL$^+$-TP can be viewed as an inertial version of PALM with the restart step and two different extrapolation points, but in IBPL$^+$ and IBPL$^+$-TP, the extrapolation parameters $\alpha^{k}_{i}$ and $\beta^{k}_{i}$ of two extrapolation points $z^{k}_{i} $ and $y^{k}_{i}$ have a value range that does not depend on the values of any other parameters, such as the step size factor $\gamma^{k}_{j}$ and $L_{\nabla_{x^{k}_{j}} H}$. This independence simplifies the difficulty of selecting and updating the extrapolation parameters of these two extrapolation points, thereby improving the convergence performance of both IBPL$^+$ and IBPL$^+$-TP. Additionally, in IBPL$^+$, $\alpha^k_{i}$ and $\beta^k_{i}$ can be adaptively updated, while in IBPL$^+$-TP, adaptive rules can be tailored for different phases to update $\alpha^k_{i}$ and $\beta^k_{i}$ more efficiently. 
We also list parameter constraint relationships of existing inertial PALM methods and the number of extrapolation points of existing inertial PALM methods in Table \ref{paraselect}. 
These methods serve as baseline methods and will be compared with our proposed methods in the experiment section later.

Notably, although ABPL \cite{yang2023accelerated} can use a single extrapolation point with independent extrapolation parameters, it cannot be extended to the algorithmic framework with multiple extrapolation points due to its algorithmic structure. 
In contrast,  both Algorithm \ref{IBPL$^+$} and Algorithm \ref{IBPLtp} can be extended to the case of using two different extrapolation points while ensuring that the extrapolation parameters of these extrapolation points are independent of and unconstrained by any other parameters. Furthermore, compared to ABPL, which cannot separate the rapid descent phase from the steady refinement phase, Algorithm \ref{IBPLtp} can employ a two-phase adaptive strategy that separates these two phases to effectively update the extrapolation parameters. These advantages significantly improve the flexibility and effectiveness of the proposed methods compared to ABPL, and will be verified in the experiment section later.

\section{Convergence analysis}
\label{sec4}
In this section, we demonstrate that Algorithm \ref{IBPL$^+$} and Algorithm \ref{IBPLtp} can ensure the monotonic and global convergence of Eq. (\ref{e11}), as well as establish its convergence rate. Since the value range of extrapolation parameters of the two extrapolation points in Algorithm \ref{IBPL$^+$} and Algorithm \ref{IBPLtp} does not depend on the values of any other parameters, we propose a modified proof framework as shown below, thereby making it applicable to our proposed method.

\begin{Property}   
The main steps to prove the monotonic and global convergence properties of Eq. (\ref{e11}) through our proposed method are as follows. 

(\romannumeral1) Monotonic decrease property, i.e., there exists a positive constant $\rho$ such that 
$\rho(\|x^{k+1}_{(N)}-y^{k}_{(N)}\|+\|x^{k+1}_{(N)}-z^{k}_{(N)}\|)^{2}\leq J(x^{k}_{(N)})-J(x^{k+1}_{(N)})$.

(\romannumeral2) Subgradient lower bound property, i.e., there exists a positive constant $\rho_{b}$ such that 
$\rho_{b}(\|x^{k+1}_{(N)}-y^{k}_{(N)}\|+\|x^{k+1}_{(N)}-z^{k}_{(N)}\|)\geq \|p_{x^{k+1}}\|$,     where $p_{x^{k+1}} \in \partial J(x^{k+1}_{(N)})$. 

(\romannumeral3) Using the Kurdyka–Łojasiewicz (KŁ) property, the generated sequence $\left\{x^{k}_{(N)}\right\}_{k \in \mathbb{N}}$ is a Cauchy sequence. 
\label{based}
\end{Property}  

Notably, from Remark \ref{remark1}, we know that the Algorithm \ref{IBPLtp} degenerates into Algorithm \ref{IBPL$^+$} when $T=\infty$. Consequently, we only need to prove that the Algorithm \ref{IBPLtp} satisfies the Property \ref{based} in this section.

\subsection{Monotonic decrease of the objective function}
We prove Property \ref{based}(\romannumeral1).  
\begin{theorem} 
Let $\left\{x^{k}_{(N)}\right\}_{k \in \mathbb{N}}$ be the sequence generated by Algorithm \ref{IBPLtp}. \\ 
(\romannumeral1) 
$J(x^{k+1}_{(N)})$ is nonincreasing and in particular 
\begin{equation}
\begin{aligned}
    J(x^{k+1}_{(N)})\leq
    J(x^{k}_{(N)}) -\rho(\|x^{k+1}_{(N)}-y^{k}_{(N)}\|+\|x^{k+1}_{(N)}-z^{k}_{(N)}\|)^{2}, \\
\end{aligned}
\label{e33}
\end{equation}
where $\rho$ is a positive constant.  \\
(\romannumeral2) We have 
\begin{center}
$\lim_{k \to \infty} \|x^{k+1}_{(N)}-y^{k}_{(N)}\|=\lim_{k \to \infty} \|x^{k+1}_{(N)}-z^{k}_{(N)}\|=0$. 
\end{center}
\label{t1}
\end{theorem}

\begin{proof}
(\romannumeral1) 
We first define 
\begin{center}
    $h_{j}(x_{j})=H(\{x_{i}^{k+1}\}_{i=1}^{j-1}, x_{j}, \{x_{i}^{k}\}_{i=j+1}^{N})$. 
\end{center}
If $J(x^{k+1}_{(N)}) \leq J(x^{k}_{(N)})-\rho_{1}(\|x^{k+1}_{(N)}-y^{k}_{(N)}\|^2+\|x^{k+1}_{(N)}-z^{k}_{(N)}\|^2)$, from $(\|x\|+\|y\|)^2\leq 2(\|x\|^2+\|y\|^2)$, we have

\begin{align}
    J(x^{k+1}_{(N)}) \leq J(x^{k}_{(N)})-\frac{\rho_{1}}{2}(\|x^{k+1}_{(N)}-y^{k}_{(N)}\|+\|x^{k+1}_{(N)}-z^{k}_{(N)}\|)^2.
    \label{c1end1}
\end{align}

If not ,since Proposition \ref{p3}, we have 
\begin{align}
&&H( x_{1}^{k+1}, \{x_{j}^{k}\}_{j=2}^{N})&\leq H(x^{k}_{(N)}) \notag\\ 
&& &+\langle \nabla_{x_{1}} H(x^{k}_{(N)}), x_{1}^{k+1}-x^{k}_{1}\rangle \notag\\ 
&&\ &+\frac{L_{\nabla_{x^{k}_{1}} H}}{2}\|x^{k+1}_{1}-x^{k}_{1}\|^{2}. 
\label{e331}
\end{align}
From Eq. (\ref{iprox}), we obtain 
\begin{align}
 && F_{1}(x^{k}_{1}) &\geq \langle \nabla_{x_{1}} H(x^{k}_{(N)}), x_{1}^{k+1}-x^{k}_{1} \rangle+F_{1}(x^{k+1}_{1}) \notag \\
 &&\  &+\frac{\sigma^{k}_{1}}{2}\|x^{k+1}_{1}-x^{k}_{1}\|^{2}. 
 \label{e332}
\end{align}
Then sum of Eq. (\ref{e331}) and Eq. (\ref{e332}), we have 
\begin{align}
&& H(x^{k}_{(N)})+F_{1}(y_{1}^{k}) &\geq F_{1}(x_{1}^{k+1})+\rho_{t}\|x^{k+1}_{1}-x^{k}_{1}\|^{2} \notag \\
&& &+H( x_{1}^{k+1}, \{x_{j}^{k}\}_{j=2}^{N}), \notag 
\end{align}
where $\rho_{t}=\frac{1}{2\sigma^{k}_{1}}-\frac{L_{\nabla_{x^{k}_{1}} H}}{2}$.  Assuming  Theorem \ref{t1}(\romannumeral1) holds when $i=n$, i.e., 
\begin{align}
&& &H(\{x_{j}^{k+1}\}_{j=1}^{n},\{x_{j}^{k}\}_{j=n+1}^{N})+\sum_{j=1}^n F_{j}(x_{j}^{k+1}) \notag \\
&& &\leq H(x^{k}_{(N)})+\sum_{j=1}^n F_{j}(x_{j}^{k}) \notag \\
&& &-\rho_{t}\|\{x_{j}^{k+1}\}_{j=1}^{n}-\{x_{j}^{k}\}_{j=1}^{n}\|^{2}, \label{e333}
\end{align}
where $\rho_{t}=\frac{1}{2}\min(\left\{\frac{1}{\sigma^{k}_{j}}-L_{\nabla_{x^{k}_{j}} H}
\right\}_{j=1}^{n})$. 

Since Eq. (\ref{iprox}), we also have  
\begin{align}
&& F_{n+1}(x^{k}_{n+1}) &\geq F_{n+1}(x^{k+1}_{n+1}) +\frac{\sigma^{k}_{n+1}}{2}\|x^{k+1}_{n+1}-x^{k}_{n+1}\|^{2} \notag \\ 
&&   &+\langle \nabla_{x_{n+1}} h_{n+1}(x^{k}_{n+1}), x^{k+1}_{n+1}-x^{k}_{n+1} \rangle. \label{e334}
\end{align}
From Proposition \ref{p3}, we infer 
\begin{align}
&& H(x^{k+1}_{(n+1)}, \{x_{j}^{k}\}_{j=n+2}^{N}) &\leq H(x^{k+1}_{(n)}, \{x_{j}^{k}\}_{j=n+1}^{N}) \notag \\
&& &+\frac{L_{\nabla_{x^{k}_{n+1}} H}}{2}\|x^{k+1}_{n+1}-x^{k}_{n+1}\|^{2}  \notag \\
&& &+ \langle \nabla_{x_{n+1}} h_{n+1}(x^{k}_{n+1}), x^{k+1}_{n+1}-x^{k}_{n+1}\rangle .\label{e335}
\end{align}
Thus, sum of Eq. (\ref{e333}), Eq. (\ref{e334}) and Eq. (\ref{e335}), when $n+1=N$, we have  
\begin{align}
&& &J(\{{x^{k+1}_{j}} \}^{N}_{j=1}) \leq J(\{{x^{k}_{j}} \}^{N}_{j=1})-\rho_{t}\|x^{k+1}_{(N)}-x^{k}_{(N)}\|^{2},
\label{c1end2}
\end{align}
where $\rho_{t}=\frac{1}{2}\min(\left\{\frac{1}{\sigma^{k}_{j}}-L_{\nabla_{x^{k}_{j}} H}
\right\}_{j=1}^{N})$. 

From Eq. (\ref{c1end1}) and Eq. (\ref{c1end2}), let $\rho=\frac{1}{4}\min(\rho_{t},\rho_{1})$, then the Theorem \ref{t1}(\romannumeral1) is true. 

(ii) From Eq. (\ref{e33}), we have
\begin{small}
\begin{align}
\rho(\|x^{k+1}_{(N)}-y^{k}_{(N)}\|+\|x^{k+1}_{(N)}-z^{k}_{(N)}\|)^2\leq J({x^{k}_{(N)}})-J({x^{k+1}_{(N)}}). \notag
\end{align}
\end{small}
Sum of both side, since $J$ has a lower bound, thus we have  
\begin{small}
\begin{align}
&& \hspace{-5mm} \rho \sum_{k=1}^{\infty} (\|x^{k+1}_{(N)}-y^{k}_{(N)}\|+\|x^{k+1}_{(N)}-z^{k}_{(N)}\|)^2&=\sum_{k=1}^{\infty}(J(x_{(N)}^{k})-J(x_{(N)}^{k+1})) \notag \\
&& &=J(x_{(N)}^{1})-\inf J \notag \\
&& &<\infty.\notag
\end{align}
\end{small}
It follows that
\begin{center}
$\lim_{k \to \infty} (\|x^{k+1}_{(N)}-y^{k}_{(N)}\|+\|x^{k+1}_{(N)}-z^{k}_{(N)}\|)^{2}=0$,  
\end{center}
which implies 
\begin{center}
$\lim_{k \to \infty} \|x^{k+1}_{(N)}-y^{k}_{(N)}\|=\lim_{k \to \infty} \|x^{k+1}_{(N)}-z^{k}_{(N)}\|=0$.
\end{center}
\end{proof}

\subsection{Subgradient lower bound for the iterates gap}
We prove Property \ref{based}(\romannumeral2) and the derived set of $\left\{x^{k}_{(N)}\right\}_{k \in \mathbb{N}}$ is the critical point set and the compact set.

\begin{theorem}
(i) In Algorithm \ref{IBPLtp}, if we define
\begin{align}
p_{x_{j}^{k+1}}=\nabla_{x_{j}}h_{j} (x^{k+1}_{j})-\nabla_{x_{j}}h_{j} (y^{k}_{j})+\frac{1}{\sigma^{k}_{j}
}(z^{k}_{j}-x^{k+1}_{j}),  \label{PXJ}
\end{align}
where $h_{j}(x_{j})=H(\{x_{i}^{k+1}\}_{i=1}^{j-1}, x_{j}, \{x_{i}^{k}\}_{i=j+1}^{N})$, then we have 
\begin{equation}
\begin{aligned}
&& p_{x_{j}^{k+1}} &\in \partial_{x_{j}}J(\{x_{i}^{k+1}\}_{i=1}^{j-1},x_{j}^{k+1}, \{x_{i}^{k}\}_{i=j+1}^{N}),  
\end{aligned}
\label{e34}
\end{equation}

(ii) We also have
\begin{equation}
\begin{aligned}
       \|\{p_{x^{k+1}_{i}} \}^{N}_{i=1}\| \leq \rho_{b}(\|{x^{k+1}_{(N)}}-{z^{k}_{(N)}}\|+\|{x^{k+1}_{(N)}}-{y^{k}_{(N)}}\|),
\end{aligned}
\label{e35}
\end{equation}    
where $\rho_{b}$ is a positive constant.   
\label{t2}
\end{theorem}
\begin{proof}
(i) By Proposition \ref{p1}, Proposition \ref{p2} and Eq. (\ref{iprox}), ${\forall} j \in  [N]$, it follows that 
\begin{align}
&& 0 &\in  \nabla_{x_{j}}h_{j}(y^{k}_{j})-\frac{1}{\sigma^{k}_{j}} (z^{k}_{j}-x^{k+1}_{j})  \notag \\
&& & +\partial_{x_{j}}(\sum_{i=1}^{j-1} F_{i}(x^{k+1}_{i})+\sum_{i=j+1}^N F_{i}(x^{k}_{i})+F_{j}(x^{k+1}_{j})). \notag
\end{align}
Thus we infer 
\begin{align}
&& &\nabla_{x_{j}}h_{j} (x^{k+1}_{j}) -\nabla_{x_{j}}h_{j} (y^{k}_{j})+ \frac{1}{\sigma^{k}_{j}} (z^{k}_{j}-x^{k+1}_{j}) \notag \\ 
&&\ &\in \nabla_{x_{j}}h_{j} (x^{k+1}_{j})+\partial_{x_{j}}(\sum_{i=1}^{j} F_{i}(x^{k+1}_{i})+\sum_{i=j+1}^N F_{i}(x^{k}_{i})) \notag\\
&&\ &=  \partial_{x_{j}}J(\{x_{i}^{k+1}\}_{i=1}^{j-1},x_{j}^{k+1}, \{x_{i}^k\}_{i=j+1}^{N}), \notag
\end{align}
which means the Theorem \ref{t2}(\romannumeral1) is true.

(ii) From Definition \ref{d3} and Assumption \ref{assump1}, ${\forall} j \in  [N]$, we have 
\begin{align}
&& \|p_{x_{j}^{k+1}}\| &=\|\nabla_{x_{j}}h_{j}(x_{j}^{k+1})-\nabla_{x_{j}}h_{j}(y_{j}^{k})+ \frac{1}{\sigma^{k}_{j}} (z^{k}_{j}-x^{k+1}_{j}) \| \notag \\
&& &\leq \|\nabla_{x_{j}}h_{j}(x_{j}^{k+1})-\nabla_{x_{j}}h_{j}(y_{j}^{k})\| \notag \\
&& &+ \|\frac{1}{\sigma^{k}_{j}} (z^{k}_{j}-x^{k+1}_{j}) \|  \notag \\
&& &\leq L_{\nabla_{x^{k}_{j}} H}\|x^{k+1}_{j}-y^{k}_{j}\|+\frac{1}{\sigma^{k}_{j}} \|x^{k+1}_{j}-z^{k}_{j}\| \notag \\
&& &\leq(L_{\nabla_{x^{k}_{j}} H}+\frac{1}{\sigma^{k}_{j}}) (\|x^{k+1}_{j}-y^{k}_{j}\|+\|x^{k+1}_{j}-z^{k}_{j}\|).  
\end{align}
Since the arbitrariness of $j$, we infer    
\begin{align}
\|\left\{p_{x^{k+1}_{j}} \right\}_{j=1}^N\| \leq \rho_{b}(\|{x^{k+1}_{(N)}}-{z^{k}_{(N)}}\|+\|{x^{k+1}_{(N)}}-{y^{k}_{(N)}}\|),
\label{c32}
\end{align}
where $\rho_{b}=max(\left\{\frac{1}{\sigma^{k}_{j}}+L_{\nabla_{x^{k}_{j}} H}\right\}_{j=1}^N)$. Therefore, Eq. (\ref{e35}) is true. 
\end{proof}

\begin{theorem}
Let $A$ denote the set of the sequence generated by the Algorithm \ref{IBPLtp},  then the derived set $A'$ of $A$ must be a compact set. \label{t3}
\end{theorem}
\begin{proof}
From Definition \ref{d1}, Proposition \ref{dopen} and Assumption \ref{assump1}, we know that
\begin{center}
$A' \subseteq \overline{A} \subseteq  B(x,\delta) ~({\forall} x \in \overline{A}, {\exists} \delta <\infty)$, 
\end{center}
where $B(x,\delta)$ is an open ball which defined as 
\begin{center}
$B(x,\delta)=\left\{y: d(x,y)<\delta\right\}$, 
\end{center}
where $d$ is a metric function defined on $X$, so $A'$ is a bounded set. Next, we just need to prove $ \overline{A'} \subseteq A'$. Take $x\in \overline{A'}$, we know that $B(x,\frac{\delta}{2})\cap A' \neq \emptyset$, take $y \in B(x,\frac{\delta}{2})\cap A'$, if $y=x$, then the Theorem \ref{t3} holds. If not, take $z \in (B(y,r)-\left\{y\right\}) \cap A$, where $r<d(x,y)$, then $z\neq x$ and $d(x,z)\leq d(x,y)+d(y,z)<\delta$, thus we have 
\begin{center}
$(B(x,\delta)-\left\{x\right\}) \cap A\neq \emptyset~({\forall} \delta>0)$.
\end{center}
Then from Definition \ref{dclu} and Proposition \ref{dcompact}, the Theorem \ref{t3} is true.  
\end{proof}

Define $x'$ is the derived set of $\{x^{k}_{(N)}\}_{k \in \mathbb{N}}$. Next, we prove the following theorem. 
\begin{theorem}
(i) Let $\{x^{k}_{(N)}\}_{k \in \mathbb{N}}$ be a sequence generated by Algorithm \ref{IBPLtp}, then $J$ is constant on $x'$. 

(ii) $x' \subseteq crit~ J$, where $crit~ J$ is the critical point set of $J$, i.e.,  
\begin{center}
$0\in \partial J^{*}$ $(J^{*}=\lim_{k \rightarrow \infty} J(x^{k}_{(N)}))$.      
\end{center}

\begin{proof}
(i) ${\forall}\overline{x} \in x'$, there exists a subsequence $x_{(N)}^{k_{j}}$  such that 
\begin{center}
  $\lim_{j \to \infty} x_{(N)}^{k_{j}}= \overline{x}$.  
\end{center}
Let 
\begin{align}
    && &F(x_{(N)}^{k_{j}})=\sum_{i=1}^{N}F_{i}(x_{i}^{k_{j}}). \label{ec200}
\end{align}
Since $F_{i}$ is lower semicontinuous, from Definition \ref{lower} and Eq. (\ref{ec200}), we obtain that 
\begin{align}
 && & \lim_{j \to \infty} \inf F(x_{(N)}^{k_{j}}) \geq F(\overline{x}). \label{t41}
\end{align}
Choosing $k = k_{j} -1$, from Eq. (\ref{iprox}), we infer 
\begin{align}
&&\lim_{j \to \infty} \sup F_{i}(x_{i}^{k_{j}})  &\leq \lim_{j \to \infty} \sup (F_{i}(\overline{x_{i}})+\frac{1}{2\sigma_{i}^{k}}\|\overline{x_{i}}-z^{k}_{i}\|_{F}^{2} \notag \\
&& &+\langle \nabla_{x_{i}} h_{i}(y_{i}^k), \overline{x_{i}}-z^{k}_{i}\rangle),  \label{t421}
\end{align}
where $h_{i}(y_{i}^k)=H(\{x^{k+1}_{n}\}_{n=1}^{i-1}, y_{i}^{k},\{x^k_{n}\}_{n=i+1}^{N})$, $\lim_{j \to \infty} x_{i}^{k_{j}}= \overline{x_{i}}$. Since $\lim_{j \to \infty} x_{(N)}^{k_{j}}= \overline{x}$ and  $\lim_{k \to \infty} \|x^{k+1}_{(N)}-y^{k}_{(N)}\|=\lim_{k \to \infty} \|x^{k+1}_{(N)}-z^{k}_{(N)}\|=0$ (Theorem \ref{t1}(ii)), we can get that 
\begin{align}
&& \lim_{j \to \infty} \|\overline{x_{i}}-z^{k}_{i}\|_{F} &\leq \lim_{j \to \infty} \|\overline{x_{i}}-x^{k_{j}}_{i}\|_{F} \notag \\ 
&& &+\lim_{j \to \infty} \|x^{k_{j}}_{i}-z^{k_{j}-1}_{i}\|_{F} \notag \\
&&& =0. \label{t422}
\end{align}
From Eq. (\ref{t421}) and Eq. (\ref{t422}), we infer 
\begin{align}
&& \lim_{j \to \infty} \sup F_{i}(x_{i}^{k_{j}}) &\leq \lim_{j \to \infty} \sup F_{i}(\overline{x_{i}}), \label{ectemp}
\end{align}
where $\lim_{j \to \infty} x_{i}^{k_{j}}= \overline{x_{i}}$. Therefore, from Eq. (\ref{ec200}) and Eq. (\ref{ectemp}), we have
\begin{align}
&& &\lim_{j \to \infty} \sup F(x_{(N)}^{k_{j}}) \leq \lim_{j \to \infty} \sup F(\overline{x}). \label{t431}
\end{align}
From Eq. (\ref{ec200}), Eq. (\ref{t431}) and Eq. (\ref{t41}), we infer 
\begin{align}
&& &\lim_{j \to \infty}F(x_{(N)}^{k_{j}}) =F(\overline{x}). 
\label{t43}
\end{align}
From Theorem \ref{t1} and Proposition \ref{tf4}, we have  
\begin{align}
  && & \lim_{j \to \infty} H(x_{(N)}^{k_{j}})=H( \overline{x}).
   \label{t432}
\end{align}
Thus from Eq. (\ref{t43}) and  Eq. (\ref{t432}), we infer 
\begin{align}
   &&\hspace{-2mm}\lim_{j \to \infty} H(x_{(N)}^{k_{j}})+\lim_{j \to \infty}F(x_{(N)}^{k_{j}})&=\lim_{j \to \infty} (H(x_{(N)}^{k_{j}})+F(x_{(N)}^{k_{j}}))\notag \\
   && &=\lim_{j \to \infty} J(x_{(N)}^{k_{j}}) \notag \\
   && &=J( \overline{x}). \notag
\end{align}
This means $J$ is constant on $x'$.

(ii) Since Theorem \ref{t2}, we know that 
\begin{align}
&& p_{x_{j}^{k+1}} &\in \partial_{x_{j}}J(\{x_{i}^{k+1}\}_{i=1}^{j-1},x_{j}^{k+1}, \{x_{i}^{k}\}_{i=j+1}^{N}),   \label{t21}
\end{align}
where the definition of $p_{x^{k+1}_{j}}$ is the same as Eq. (\ref{PXJ}). From Theorem \ref{t1}(ii) and Theorem \ref{t2}, we infer
\begin{align}
&& \lim_{k \rightarrow \infty} \|\{p_{x^{k+1}_{i}} \}_{i=1}^{N}\| &\leq \lim_{k \rightarrow \infty} \rho_{b} \|x^{k+1}_{(N)}-y^{k}_{(N)}\| \notag \\
&& &+\lim_{k \rightarrow \infty} \rho_{b} \|x^{k+1}_{(N)}-z^{k}_{(N)}\| \notag \\
&& &=0.  \label{t22}
\end{align}
Since Theorem \ref{t1}(i) and Theorem \ref{t5}(i), we know that 
\begin{align}
 && \lim_{k \rightarrow \infty} J(\{x_{i}^{k+1}\}_{i=1}^{j-1},x_{j}^{k+1}, \{x_{i}^{k}\}_{i=j+1}^{N})&=J^{*}.\label{t23} 
\end{align}
Therefore, from Eq. (\ref{t21}),  Eq. (\ref{t22}) and Eq. (\ref{t23}), we have 
\begin{align}
&& 0&\in \partial_{x_{j}}J(\{x_{i}^{k+1}\}_{i=1}^{j-1},x_{j}^{k+1}, \{x_{i}^{k}\}_{i=j+1}^{N}) \notag \\
&& &= \partial_{x_{j}} J^{*}.  \notag
\end{align}
This means $0 \in  \partial J^{*}$.
\end{proof}

\label{t5}
\end{theorem}

\subsection{Global convergence}
Utilizing the Definition \ref{def_KL}, we prove Property \ref{based}(\romannumeral3) that the sequence generated by Algorithm \ref{IBPLtp} has global convergence (see Theorem \ref{glo}) and establish the convergence rate (see Theorem \ref{rate}). 
\begin{theorem}
The sequence $\{x^{k}_{(N)}\}_{k \in \mathbb{N}}$ generated by Algorithm \ref{IBPLtp} converges  when $\alpha_{\max},\beta_{\max}<1$, i.e, 
\begin{center}
$\lim_{k\rightarrow \infty} \|x^{k+p}_{(N)}-x^{k}_{(N)}\|=0 ~(\forall p\in \mathbb{N}).$    
\end{center}
\label{glo}
\end{theorem}
\begin{proof}
Since Definition \ref{def_KL}, there exists a concave function $\phi$ so that 
\begin{align}
\phi^{'}(J(x^{k}_{(N)})-J( \overline{x})) dist(0, \partial J(x^{k}_{(N)})) \geq 1. \label{KL1}
\end{align}
From $\phi$ is the convex function, we have
\begin{align}
&&&\phi(J(x^{k+1}_{(N)})-J( \overline{x}))-\phi(J(x^{k}_{(N)})-J( \overline{x})) \notag \\
&& &\leq\phi^{'}(J(x^{k}_{(N)})-J( \overline{x}))(J(x^{k+1}_{(N)})-J(x^{k}_{(N)})). 
\label{temp1}
\end{align}
From Theorem \ref{t2}, we infer 
\begin{align}
&& \hspace{-1mm}dist(0, \partial J(x^{k}_{(N)}) ) &\leq \|\left\{p_{x^{k}_{i}} \right\}^{N}_{i=1}\| \notag \\
&& &\leq \rho_{b}(\|x^{k}_{(N)}-y^{k-1}_{(N)}\|+\|x^{k}_{(N)}-z^{k-1}_{(N)}\|). \label{glo1}
\end{align}
Since the Eq. (\ref{KL1}) and Eq. (\ref{glo1}), we have 
\begin{align}
&&  \hspace{-4mm}\phi^{'}(J(x^{k}_{(N)})-J( \overline{x})) &\geq \frac{1}{dist(0, \partial J(x^{k}_{(N)}))} \notag \\
&& &\geq \frac{1}{\rho_{b}(\|x^{k}_{(N)}-y^{k-1}_{(N)}\|+\|x^{k}_{(N)}-z^{k-1}_{(N)}\|)} \label{glo2}. 
\end{align}
Let $G(k)=J(x^{k}_{(N)})-J( \overline{x})$, from  Eq. (\ref{temp1}), Eq. (\ref{glo1}) and Eq. (\ref{glo2}), we have 
  \begin{align}
&& &\hspace{-3mm}\phi(G(k))-\phi(G(k+1)) \notag \\
&& &\geq \phi^{'}(G(k))(G(k)-G(k+1))  \notag \\
&& &\geq \frac{G(k)-G(k+1)}{\rho_{b}(\|x^{k}_{(N)}-y^{k-1}_{(N)}\|+\|x^{k}_{(N)}-z^{k-1}_{(N)}\|)} \notag \\
&& &\geq \frac{\rho(\|x^{k+1}_{(N)}-y^{k}_{(N)}\|+\|x^{k+1}_{(N)}-z^{k}_{(N)}\|)^2}{\rho_{b}(\|x^{k}_{(N)}-y^{k-1}_{(N)}\|+\|x^{k}_{(N)}-z^{k-1}_{(N)}\|)}. \notag 
\end{align}
Define C=$\frac{\rho}{\rho_{b}}$, C is a constant, so we infer 
\begin{small}
  \begin{align}
    && &\hspace{-5mm} (\|x^{k+1}_{(N)}-y^{k}_{(N)}\|+\|x^{k+1}_{(N)}-z^{k}_{(N)}\|)^2  \notag \\
    && &\hspace{-5mm}\leq C (\phi(G(k))-\phi(G(k+1))) (\|x^{k}_{(N)}-y^{k-1}_{(N)}\|+\|x^{k}_{(N)}-z^{k-1}_{(N)}\|). \notag
   \end{align}
   \end{small}
Using the fact that $2ab\leq a^{2}+b^{2}$ 
  \begin{align}
    && &2 (\|x^{k+1}_{(N)}-y^{k}_{(N)}\|+\|x^{k+1}_{(N)}-z^{k}_{(N)}\|) \notag \\
    & &&\leq C (\phi(G(k))-\phi(G(k+1)))  \notag \\
    && &+\|x^{k}_{(N)}-y^{k-1}_{(N)}\|+\|x^{k}_{(N)}-z^{k-1}_{(N)}\|. \notag
   \end{align}
Sum both sides 
\begin{align}
&&  &2 \sum_{k=l+1}^{K} (\|x^{k+1}_{(N)}-y^{k}_{(N)}\|+\|x^{k+1}_{(N)}-z^{k}_{(N)}\|)  \notag \\
&&&\leq \sum_{k=l+1}^{K} (\|x^{k}_{(N)}-y^{k-1}_{(N)}\|+\|x^{k}_{(N)}-z^{k-1}_{(N)}\|)\notag \\
&& &+C(\phi(G(l+1))-\phi(G(K+1))) \notag \\
&& &=C (\phi(G(l+1))-\phi(G(K+1))) \notag \\
&& &+\sum_{k=l+1}^{K}(\|x^{k+1}_{(N)}-y^{k}_{(N)}\|+\|x^{k+1}_{(N)}-z^{k}_{(N)}\|)  \notag \\
&& &+(\|x^{l+1}_{(N)}-y^{l}_{(N)}\|+\|x^{l+1}_{(N)}-z^{l}_{(N)}\|).\label{glotemp0}
\end{align}
From Assumption \ref{assump1} and Eq. (\ref{glotemp0}),  we can get that 
\begin{align}
&& &\lim_{K\rightarrow \infty}\sum_{k=l+1}^{K} (\|x^{k+1}_{(N)}-y^{k}_{(N)}\|+\|x^{k+1}_{(N)}-z^{k}_{(N)}\|)  \notag \\
&& &\leq (\|x^{l+1}_{(N)}-y^{l}_{(N)}\|+\|x^{l+1}_{(N)}-z^{l}_{(N)}\|)+ C\phi(G(l+1))  \notag \\
&& &- \lim_{K\rightarrow \infty} C\phi(G(K+1))) \notag \\
 && & < \infty .\notag
\end{align}
Thus we have 
\begin{align}
&& \sum_{k=l+1}^{\infty} (\|x^{k+1}_{(N)}-y^{k}_{(N)}\|+\|x^{k+1}_{(N)}-z^{k}_{(N)}\|) &< \infty. \label{c7}
\end{align}
Since 
\begin{align}
&& & \|x^{k+1}_{(N)}-x^{k}_{(N)}\| -\alpha_{max} \|x^{k}_{(N)}-x^{k-1}_{(N)}\|  \
  \leq\|x^{k+1}_{(N)}-z^{k}_{(N)}\|, \notag \\
 && &\|x^{k+1}_{(N)}-x^{k}_{(N)}\| -\beta_{max} \|x^{k}_{(N)}-x^{k-1}_{(N)}\|  \
  \leq\|x^{k+1}_{(N)}-y^{k}_{(N)}\| . \label{glotemp1}
\end{align}
From Eq. (\ref{c7}) and Eq. (\ref{glotemp1}), let $l=\inf \{k:\frac{|J(x^{k+1}_{(N)})-J(x^{k}_{(N)})|}{|J(x^{0}_{(N)})|}< T \}$, we know that 
\begin{align}
&& &\sum_{k=l+1}^{\infty} (2\|x^{k+1}_{(N)}-x^{k}_{(N)}\| -(\alpha_{max}+\beta_{max}) \|x^{k}_{(N)}-x^{k-1}_{(N)}\|)  \notag \\
  && &\leq \sum_{k=l+1}^{\infty}(\|x^{k+1}_{(N)}-y^{k}_{(N)}\|+\|x^{k+1}_{(N)}-z^{k}_{(N)}\|) \notag \\ &&&< \infty. \notag
\end{align}
Therefore, define $c_{max}=\frac{\alpha_{max}+\beta_{max}}{2}$, we have 
\begin{align}
&& &\sum_{k=l+1}^{\infty} (\|x^{k+1}_{(N)}-x^{k}_{(N)}\| -c_{max} \|x^{k}_{(N)}-x^{k-1}_{(N)}\|)\notag \\
&& &\leq \sum_{k=l+1}^{\infty} (\|x^{k+1}_{(N)}-y^{k}_{(N)}\|+\|x^{k+1}_{(N)}-z^{k}_{(N)}\|)\notag \\
&& & < \infty, \label{glotemp2}
\end{align}
and 
\begin{align}
&& &\sum_{k=l+1}^{\infty} \|x^{k+1}_{(N)}-x^{k}_{(N)}\| -\sum_{k=l+1}^{\infty}c_{max} \|x^{k}_{(N)}-x^{k-1}_{(N)}\| \notag \\
&& &=\sum_{k=l+1}^{\infty} \|x^{k+1}_{(N)}-x^{k}_{(N)}\| -\sum_{k=l+1}^{\infty}c_{max} \|x^{k+1}_{(N)}-x^{k}_{(N)}\| \notag \\
&& &-c_{max} \|x^{l+1}_{(N)}-x^{l}_{(N)}\| \notag \\
&& &=\sum_{k=l+1}^{\infty} (\|x^{k+1}_{(N)}-x^{k}_{(N)}\| -c_{max} \|x^{k+1}_{(N)}-x^{k}_{(N)}\|) \notag \\
&& &-c_{max} \|x^{l+1}_{(N)}-x^{l}_{(N)}\| \notag \\
&& & =\sum_{k=l+1}^{\infty} (1-c_{max})\|x^{k+1}_{(N)}-x^{k}_{(N)}\|-c_{max} \|x^{l+1}_{(N)}-x^{l}_{(N)}\|. \label{glotemp3}
\end{align}
Let $s^{k}=\|x^{k+1}_{(N)}-x^{k}_{(N)}\|$, from Eq. (\ref{glotemp2}) and Eq. (\ref{glotemp3}), we have   
\begin{align}
&& \sum_{k=l+1}^{\infty} (1-c_{max})s^{k}-c_{max} s^{l} &=\sum_{k=l+1}^{\infty}( s^{k} -c_{max} s^{k-1}) \notag \\
&& &<\infty. \notag
\end{align}
Then we infer
\begin{align}
\sum_{k=l+1}^{\infty} (1-c_{max})s^{k}<\infty. \label{glotemp4} 
\end{align}
From  $ (1-c_{max})$ is a positive constant and Eq. (\ref{glotemp4}), we infer 
\begin{align}
\sum_{k=l+1}^{\infty}\|x^{k+1}_{(N)}-x^{k}_{(N)}\|=\sum_{k=l+1}^{\infty} s^{k}<\infty. \notag 
\end{align}
Then we have 
\begin{align}
\sum_{k=0}^{\infty}\|x^{k+1}_{(N)}-x^{k}_{(N)}\|=\sum_{k=0}^{\infty} s^{k}<\infty. \notag
\end{align}
This shows that 
\begin{align}
&& \lim_{K\rightarrow \infty} \|x^{K+p}_{(N)}-x^{K}_{(N)}\|& = \lim_{K\rightarrow \infty}\|\sum_{k=K}^{K+p-1}(x^{k+1}_{(N)}-x^{k}_{(N)})\|  \notag\\
&& &\leq \lim_{K\rightarrow \infty}\sum_{k=K}^{K+p-1}\|x^{k+1}_{(N)}-x^{k}_{(N)} \| \notag \\
&& &=\lim_{K\rightarrow \infty}\sum_{k=K}^{\infty} s^{k} \notag\\
&& &= 0.  \notag
\end{align}
This means the Theorem \ref{glo} is true.  
\end{proof}

\begin{theorem} 
(Convergence rate)  Let $\{x^{k}_{(N)}\}_{k \in \mathbb{N}}$ be the sequence generated by Algorithm \ref{IBPLtp}, the desingularizing function has the form of $~\phi(t) = \frac{C}{\theta}t^{\theta}$, with $\theta\in (0, 1], ~ C > 0$. Let $ r^{k} = |J(x^{k}_{(N)})-J^{*}|$, $J^{*}=\lim_{k \rightarrow \infty} J(x^{k}_{(N)})$. Then the following assertions hold.  \\
(\romannumeral1)  If $\theta =1$, the Algorithm \ref{IBPLtp} terminates in finite steps. \\
(\romannumeral2)  If $\theta \in [\frac{1}{2}, 1)$, then there exist an integer $k_{2}$ such that
\begin{center}
$r^{k}\leq (\frac{d_{1}C^{2}}{1+d_{1}C^{2}})^{k-k_{2}} r^{k_{2}},~{\forall k > k_{2}}$. 
\end{center}
(\romannumeral3)  If $\theta \in (0, \frac{1}{2})$, then there exist an integer $k_{3}$  such that
\begin{center}
$r^{k}\leq [\frac{C}{(k-k_{3})d_{2}(1-2\theta)}]^{\frac{1}{1-2\theta}},~{\forall k > k_{3}}$.
\end{center}
$d_{1}$ and $d_{2}$ are defined as 
\begin{align}
&& d_{1}&=(\frac{\rho_{b}^2}{\rho}), \notag \\
&& d_{2}&=\min\left\{\frac{1}{2d_{1}C}, \frac{C}{1-2\theta}(2^{\frac{2\theta-1}{2\theta-2}})r_{k_{3}}^{2\theta-1}\right\}\notag, 
\end{align}
where the definition of $\rho$ is the same as Eq. (\ref{e33}), the definition of $\rho_{b}$ is the same as Eq. (\ref{e35}). 
\label{rate}
\end{theorem}
\begin{proof}
We assume that $\forall k \in \mathbb{N}$, $J(x^{k}_{(N)}) \neq J^*$. From Proposition \ref{def_uniKL}, there exists a desingular function $\phi$ such that
\begin{center}
    $\phi^{\prime}(J(x^{k}_{(N)})-J^*)dist(0,\partial J(x^{k}_{(N)}))\geq 1$.    
\end{center}
According to the above equation, from Theorem \ref{t1} and Theorem \ref{t2}, we can obtain 
\begin{align}
   && 1  &\leq (\phi^{\prime}(J(x^{k}_{(N)})-J^*)dist(0,\partial J(x^{k}_{(N)})))^2 \notag \\ 
    &&&\leq (\phi^{\prime}(r_{k}))^2\rho_{b}^2(\|x^{k}_{(N)}-y^{k-1}_{(N)}\|+\|x^{k}_{(N)}-z^{k-1}_{(N)}\|)^2 \notag \\ 
    &&&\leq (\phi^{\prime}(r_{k}))^2\rho_{b}^2 \frac{J(x^{k-1}_{(N)})-J(x^{k}_{(N)})}{\rho} \notag \\ 
    &&&=d_{1} (\phi^{\prime}(r_{k}))^2(r_{k-1}-r_{k}).
    \label{eq_rate1}
\end{align}
Since $\phi(t)=\frac{C}{\theta}t^{\theta}$, we have $\phi^{\prime}(t)=Ct^{\theta-1}$. From $\phi^{\prime}(t)=Ct^{\theta-1}$ and Eq. (\ref{eq_rate1}), we infer 
\begin{align}
   1\leq d_{1}C^2 r_{k}^{2\theta-2}(r_{k-1}-r_{k}). 
    \label{eq_rate2}
\end{align}
From Theorem \ref{glo} and Eq. (\ref{eq_rate2}), we can follow the same technique of the proof of Theorem 2 in \cite{li2017convergence} to discuss the cases $\theta=1$, $\theta\in [\frac{1}{2},1)$ and $\theta\in (0,\frac{1}{2})$, thereby establishing the result. 
\end{proof}

\section{Numerical experiments} 
\label{Experi}
In this section, we apply our method to both the sparse nonnegative matrix factorization with $\ell_0$-constraints and the sparse nonnegative CP decomposition with $\ell_0$-constraints problems. All the experiments are implemented in this configuration: 
Intel(R) Core(TM) i7-10850H CPU @ 2.70GHz, RAM	16.0 GB, Matlab 2021b. The fundamental tensor computation is based on Tensor Toolbox 3.5 \cite{bader2006algorithm}.  The code is available at  \url{https://github.com/Weifeng-Yang/IBPL}.

\begin{table*}[!ht]
\belowrulesep=0pt\aboverulesep=0pt
\centering
\caption{\centering Average results by methods on lp\_ship12l and Franz2 datasets. The best performance is highlighted in bold.}
\label{Tabmat}
\begin{tabular*}{0.999\linewidth}{p{1.6cm}| c c c|c c c} 
\hline      
\diagbox[width=5.74em]{\multirow{2}{*}{Method}}{Dataset} & \multicolumn{3}{c|}{lp\_ship12l}  & \multicolumn{3}{c}{Franz2}     \\ \cline{2-7} 
& Obj & Rel& Ranking & Obj & Rel& Ranking  \\
\midrule   
 
    iPALM & 116969.243$\pm$5718.409 & 3.797$\pm$0.091& 0 & 903794.645$\pm$17406.221 & 9.168$\pm$0.089 & 0 \\
    BPL   & 7895.702$\pm$44.705 & 0.987$\pm$0.003 & 0 & 11718.541$\pm$102.456 & 1.044$\pm$0.005 & 0\\
    IBPG  & 8254.258$\pm$83.169 & 1.009$\pm$0.005 & 0 & 15641.389$\pm$360.038 & 1.206$\pm$0.014 & 0\\
    TITAN & 8122.686$\pm$79.22 & 1.001$\pm$0.005 & 0 & 14332.043$\pm$411.796 & 1.154$\pm$0.017 & 0\\
    ABPL  & 2674.483$\pm$23.419 & 0.574$\pm$0.003 & 0 & 10045.369$\pm$32.421 & 0.967$\pm$0.002 & 0\\
    PGels & 7868.767$\pm$32.693 & 0.985$\pm$0.002 & 0 & 11746.153$\pm$106.965 & 1.045$\pm$0.005 & 0\\
    APGL  & 2331.628$\pm$22.931 & 0.536$\pm$0.003 & 0 & 9797.648$\pm$51.815 & 0.955$\pm$0.003 & 0\\
    IBPL  & 2014.515$\pm$16.306 & 0.498$\pm$0.002 & 0 & 9265.226$\pm$18.497 & 0.928$\pm$0.001 & 0\\

    IBPL$^+$& 1921.459$\pm$13.057 & 0.487$\pm$0.002 & 2  & 9072.122$\pm$10.522 & 0.919$\pm$0.001 & 1 \\  


 IBPL$^+$-TP &
 \textbf{1701.968} $\pmb{\pm}$ \textbf{14.471} &
 \textbf{0.458} $\pmb{\pm}$ \textbf{0.002} &
 \textbf{18} &
\textbf{8409.709} $\pmb{\pm}$ \textbf{11.229} &
 \textbf{0.884} $\pmb{\pm}$ \textbf{0.001} &
\textbf{19} \\
\hline 
\end{tabular*}
\subcaption{\centering Average results by methods on lp\_ship12l and Franz2 datasets with $r=300$. }

\belowrulesep=0pt\aboverulesep=0pt
\centering
\begin{tabular*}{1.008\linewidth}{p{1.6cm}| c c c|c c c} 
\hline      
\diagbox[width=5.74em]{\multirow{2}{*}{Method}}{Dataset} & \multicolumn{3}{c|}{lp\_ship12l}  & \multicolumn{3}{c}{Franz2}     \\ \cline{2-7} 
& Obj & Rel& Ranking & Obj & Rel& Ranking  \\
\midrule   
 
iPALM & 274671.909$\pm$5261.825 & 5.820$\pm$0.056 &0 &  2294144.347$\pm$48785.535 & 14.606$\pm$0.155 & 0  \\
    BPL   & 8239.175$\pm$60.148 & 1.008$\pm$0.004&0 & 14213.246$\pm$514.229 & 1.150$\pm$0.02 &  0\\
    IBPG  & 8777.91$\pm$173.528 & 1.040$\pm$0.010 &0& 24350.958$\pm$1800.05 & 1.504$\pm$0.054 &  0\\
    TITAN & 8569.892$\pm$162.056 & 1.028$\pm$0.010 &0 & 20296.587$\pm$714.018 & 1.374$\pm$0.024 & 0 \\
    ABPL  & 2987.942$\pm$56.18 & 0.607$\pm$0.006 &0& 10280.06$\pm$29.515 & 0.978$\pm$0.001 & 0 \\
    PGels & 8216.564$\pm$48.52 & 1.007$\pm$0.003 &0& 14244.869$\pm$228.494 & 1.151$\pm$0.009 & 0 \\
    APGL  & 2431.996$\pm$21.917 & 0.548$\pm$0.002 &0 & 10162.652$\pm$49.785 & 0.972$\pm$0.002 & 0 \\
    IBPL  & 1997.605$\pm$46.944 & 0.496$\pm$0.006 &1& 9386.052$\pm$25.501 & 0.934$\pm$0.001 &  1\\

     IBPL$^+$ & 1775.245$\pm$15.603 & 0.468$\pm$0.002 & 2  & 9202.255$\pm$8.378 & 0.925$\pm$0.003 & 3 \\ 
 

IBPL$^+$-TP &
\textbf{1507.042} $\pmb{\pm}$ \textbf{18.845}&
\textbf{0.431} $\pmb{\pm}$ \textbf{0.003}&
\textbf{17} &
\textbf{8476.699} $\pmb{\pm}$ \textbf{13.740} &
\textbf{0.888} $\pmb{\pm}$ \textbf{0.003} &
\textbf{16} \\
 \hline

\end{tabular*}
\subcaption{\centering Average results by methods on lp\_ship12l and Franz2 datasets with $r=400$. }

\belowrulesep=0pt\aboverulesep=0pt
\centering
\begin{tabular*}{1.008\linewidth}{p{1.6cm}| c c c|c c c} 
\hline      
\diagbox[width=5.74em]{\multirow{2}{*}{Method}}{Dataset} & \multicolumn{3}{c|}{lp\_ship12l}  & \multicolumn{3}{c}{Franz2}     \\ \cline{2-7} 
& Obj & Rel& Ranking & Obj & Rel& Ranking  \\
\midrule   
 
iPALM & 507559.417$\pm$9574.294 & 7.912$\pm$0.074 &0 & 4305136.325$\pm$88174.097 & 20.009$\pm$0.206 &  0\\
    BPL   & 8641.317$\pm$68.161 & 1.032$\pm$0.004 &0 & 18099.618$\pm$520.759 & 1.297$\pm$0.019 & 0 \\
    IBPG  & 9743.509$\pm$282.52 & 1.096$\pm$0.016 &0 & 41373.688$\pm$2043.449 & 1.961$\pm$0.048 & 0 \\
    TITAN & 9150.735$\pm$238.632 & 1.062$\pm$0.014 &0 & 32044.524$\pm$2323.143 & 1.725$\pm$0.064 & 0 \\
    ABPL  & 3162.031$\pm$59.733 & 0.624$\pm$0.006 &0 & 10400.602$\pm$15.113 & 0.984$\pm$0.001 &  0\\
    PGels & 8627.905$\pm$55.228 & 1.032$\pm$0.003 &0 & 18601.769$\pm$367.712 & 1.315$\pm$0.013 &  0\\
    APGL  & 2521.159$\pm$22.137 & 0.558$\pm$0.002 &0 & 10342.989$\pm$24.04 & 0.981$\pm$0.001 &  0\\
    IBPL  & 2021.258$\pm$30.533 & 0.499$\pm$0.004 &1 & 9448.546$\pm$45.143 & 0.937$\pm$0.002 &  0\\

 IBPL$^+$ & 1688.248$\pm$21.81 & 0.456$\pm$0.003 & 3  & 9174.862$\pm$24.834 & 0.924$\pm$0.001 & 2 \\ 


IBPL$^+$-TP &
\textbf{1474.000} $\pmb{\pm}$ \textbf{29.00} &
\textbf{0.426} $\pmb{\pm}$ \textbf{0.004} &
\textbf{16} &
\textbf{8554.791} $\pmb{\pm}$ \textbf{27.879} &
\textbf{0.892} $\pmb{\pm}$ \textbf{0.002} &
\textbf{18} \\

\hline 

\end{tabular*}
\caption{\centering Average results by methods on lp\_ship12l and Franz2 datasets with $r=500$.}
\end{table*}

\begin{table*}[!t]
\belowrulesep=0pt\aboverulesep=0pt
\centering
\caption{\centering  Average results by methods on BASEHOCK dataset. The best performance is highlighted in bold.}
\label{Tabmat2}
\belowrulesep=0pt\aboverulesep=0pt
\centering
\begin{tabular*}{1.01\linewidth}{p{1.6cm}| c c c|c c c} 
\hline      
\diagbox[width=5.74em]{\multirow{2}{*}{Method}}{Rank} & \multicolumn{3}{c|}{r=300}  & \multicolumn{3}{c}{r=400}     \\ \cline{2-7}
& Obj & Rel& Ranking & Obj & Rel& Ranking  \\
\midrule

    iPALM & 5.573e+05$\pm$7.391e+03 & 1.305$\pm$0.009 & 0 & 9.454e+05$\pm$1.676e+04 & 1.699$\pm$0.015 & 0 \\
    
    BPL   & 1.679e+05$\pm$2.364e+03 & 0.716$\pm$0.005 & 0 & 1.957e+05$\pm$4.175e+03 & 0.773$\pm$0.008 & 0 \\
    
    IBPG  & 2.061e+05$\pm$3.116e+03 & 0.793$\pm$0.006 & 0 & 2.300e+05$\pm$4.224e+03 & 0.838$\pm$0.008 & 0  \\
    
    TITAN & 2.015e+05$\pm$4.225e+03 & 0.784$\pm$0.008 & 0 & 2.191e+05$\pm$3.393e+03 & 0.818$\pm$0.006  & 0 \\
    
    ABPL  & 1.295e+05$\pm$1.576e+03 & 0.629$\pm$0.004 & 0 & 1.421e+05$\pm$1.948e+03 & 0.659$\pm$0.005 & 0 \\
    
    PGels & 1.678e+05$\pm$2.033e+03 & 0.716$\pm$0.004 & 0 & 1.945e+05$\pm$3.217e+03 & 0.771$\pm$0.006 & 0 \\
    
    APGL  & 8.919e+04$\pm$1.253e+03 & 0.522$\pm$0.004 & 0 & 9.859e+04$\pm$1.546e+03 & 0.549$\pm$0.004  & 0 \\
    
    IBPL  & 7.743e+04$\pm$1.706e+03 & 0.486$\pm$0.005 & 2 & 8.491e+04$\pm$2.144e+03 & 0.509$\pm$0.006 & 0 \\
    
    IBPL$^+$ & 7.714e+04$\pm$1.077e+03 & 0.485$\pm$0.003 & 2 & 8.104e+04$\pm$1.269e+03 & 0.497$\pm$0.004  & 2\\
    
    IBPL$^+$-TP & \textbf{6.682e+04} $\pmb{\pm}$ \textbf{1.144e+03} & \textbf{0.452} $\pmb{\pm}$ \textbf{0.003} & \textbf{16} & \textbf{7.335e+04} $\pmb{\pm}$ \textbf{3.719e+03} & \textbf{0.473} $\pmb{\pm}$ \textbf{0.011} & \textbf{18} \\

\hline 
\end{tabular*}

\vspace{5pt}

\centering
\begin{tabular*}{0.5625\linewidth}{p{1.6cm}| c c c} 
\hline      
\diagbox[width=5.74em]{\multirow{2}{*}{Method}}{Rank} & \multicolumn{3}{c}{r=500}      \\ \cline{2-4}
& Obj & Rel& Ranking \\
\midrule   
 
    iPALM  & 1.547e+06$\pm$4.955e+04 & 2.173$\pm$0.035 &   0\\
  
    BPL   & 2.132e+05$\pm$2.707e+03 & 0.807$\pm$0.005 &  0\\
    
    IBPG  & 2.465e+05$\pm$6.540e+03 & 0.868$\pm$0.011 & 0 \\
    
    TITAN  & 2.384e+05$\pm$5.668e+03 & 0.853$\pm$0.010 &  0\\
    
    ABPL    & 1.500e+05$\pm$2.123e+03 & 0.677$\pm$0.005 &   0\\
    
    PGels  & 2.124e+05$\pm$2.019e+03 & 0.805$\pm$0.004 &  0 \\
    
    APGL   & 1.055e+05$\pm$1.376e+03 & 0.567$\pm$0.004 &   0\\
    
    IBPL  & 8.960e+04$\pm$1.908e+03 & 0.523$\pm$0.006 &   0\\
    
    IBPL$^+$  & 8.566e+04$\pm$2.225e+03 & 0.511$\pm$0.007 &    1\\
    
    IBPL$^+$-TP  & \textbf{7.564e+04} $\pmb{\pm}$ \textbf{4.160e+03} & \textbf{0.481} $\pmb{\pm}$ \textbf{0.008} &  \textbf{19} \\
 
\hline 
\end{tabular*}

\end{table*}

\begin{figure*}[!t]
    \centering 
    \includegraphics[width=0.33\linewidth]{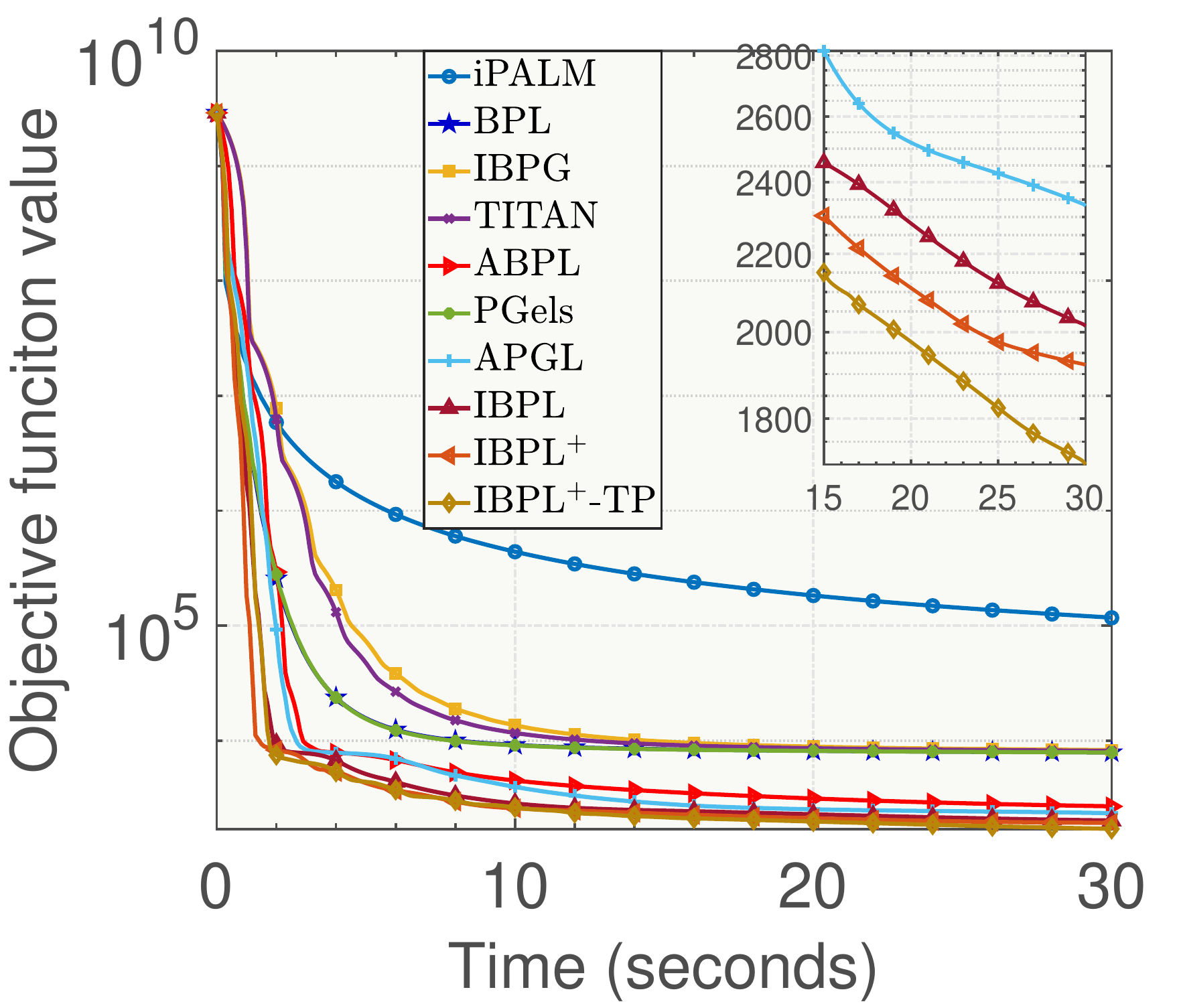}
    \includegraphics[width=0.33\linewidth]{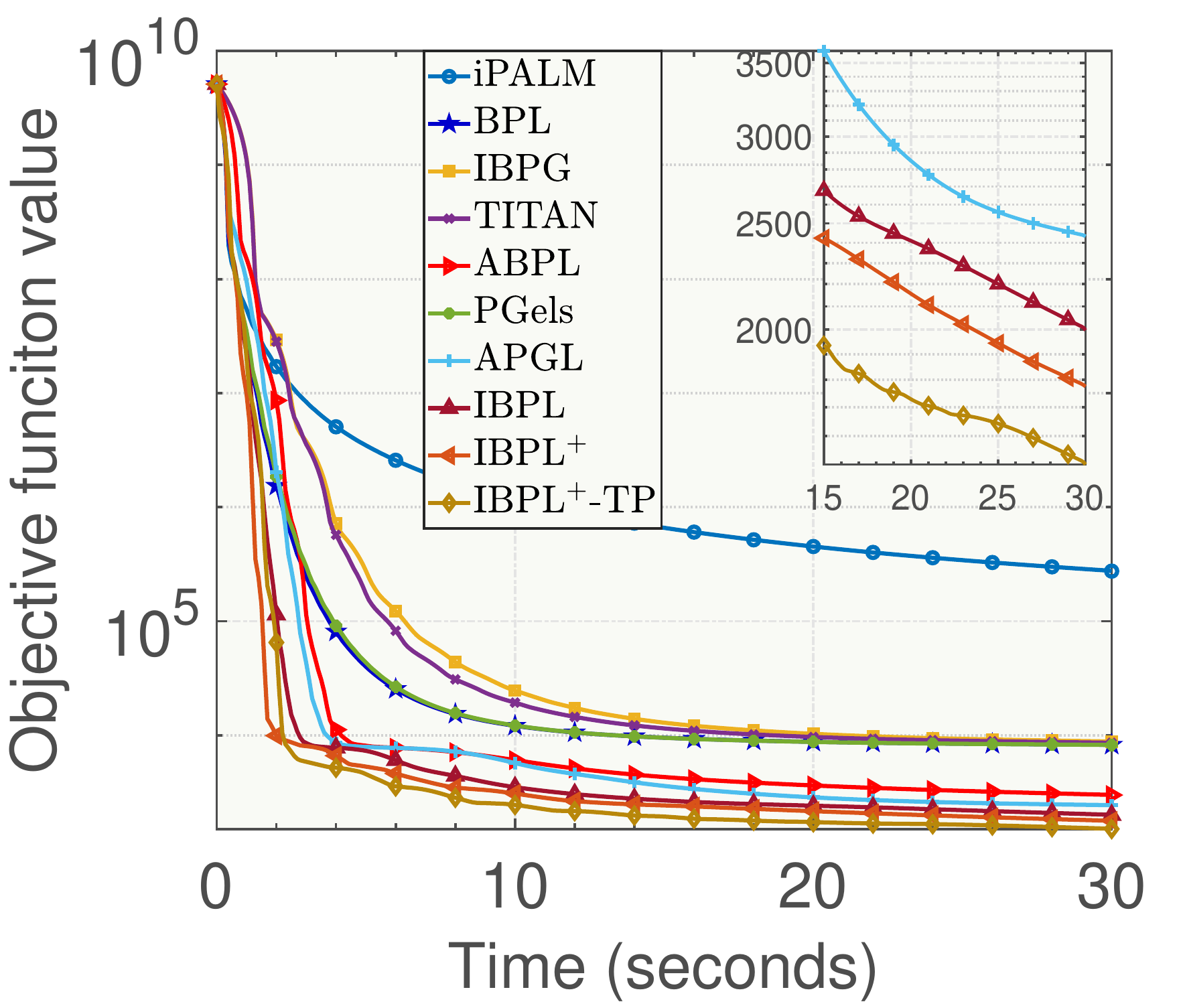}
    \includegraphics[width=0.33\linewidth]{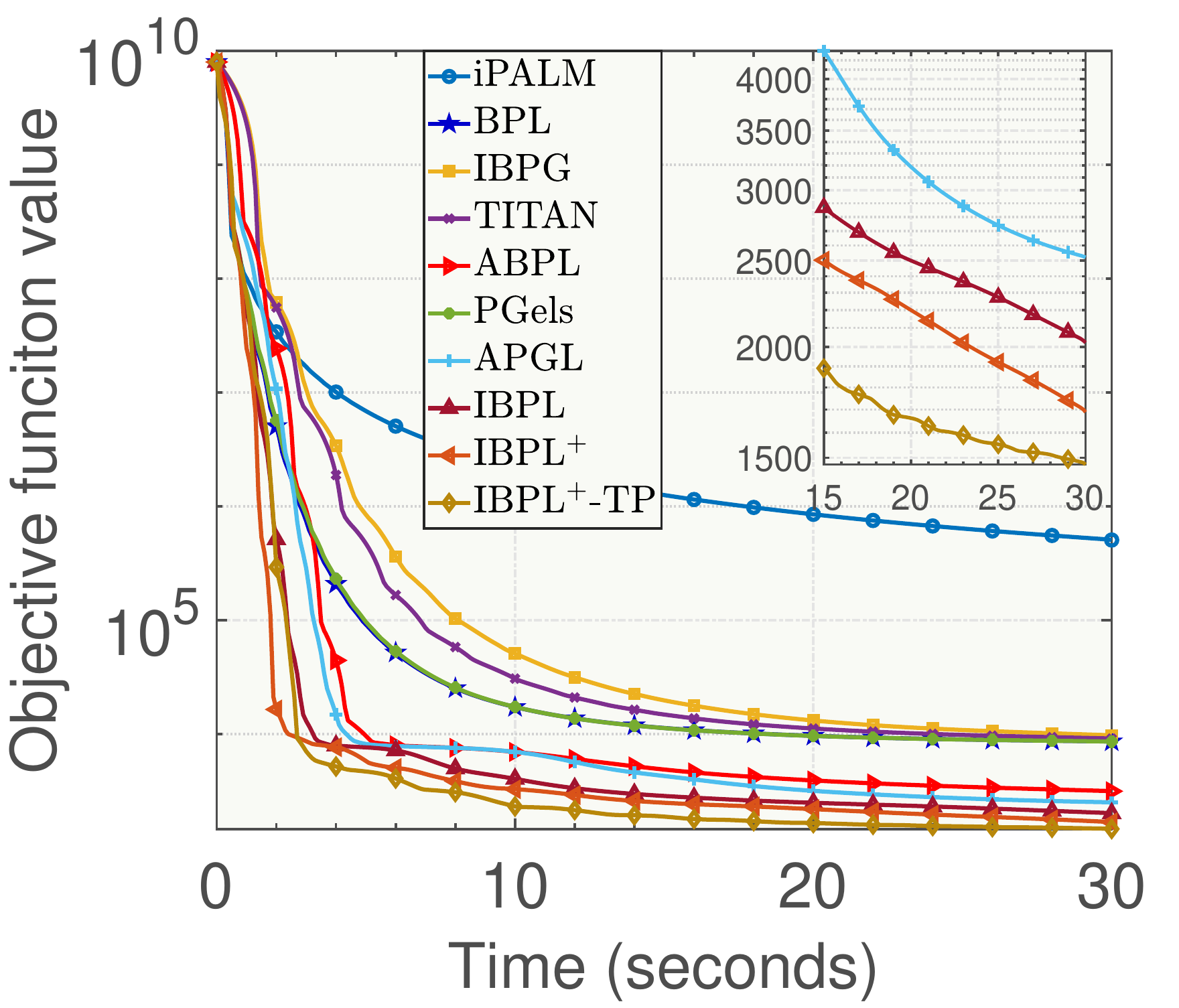}
    \caption{ 
    Average convergence behavior of all methods on lp\_ship12l dataset with  $r=300~$(left), $r=400~$(middle) and $r=500~$(right). 
    A partially enlarged view of several curves with the fastest descending speed and the best effect is also provided. 
    }\label{figmat1}
\end{figure*}

\begin{figure*}[!ht]
    \centering 
    \includegraphics[width=0.33\linewidth]{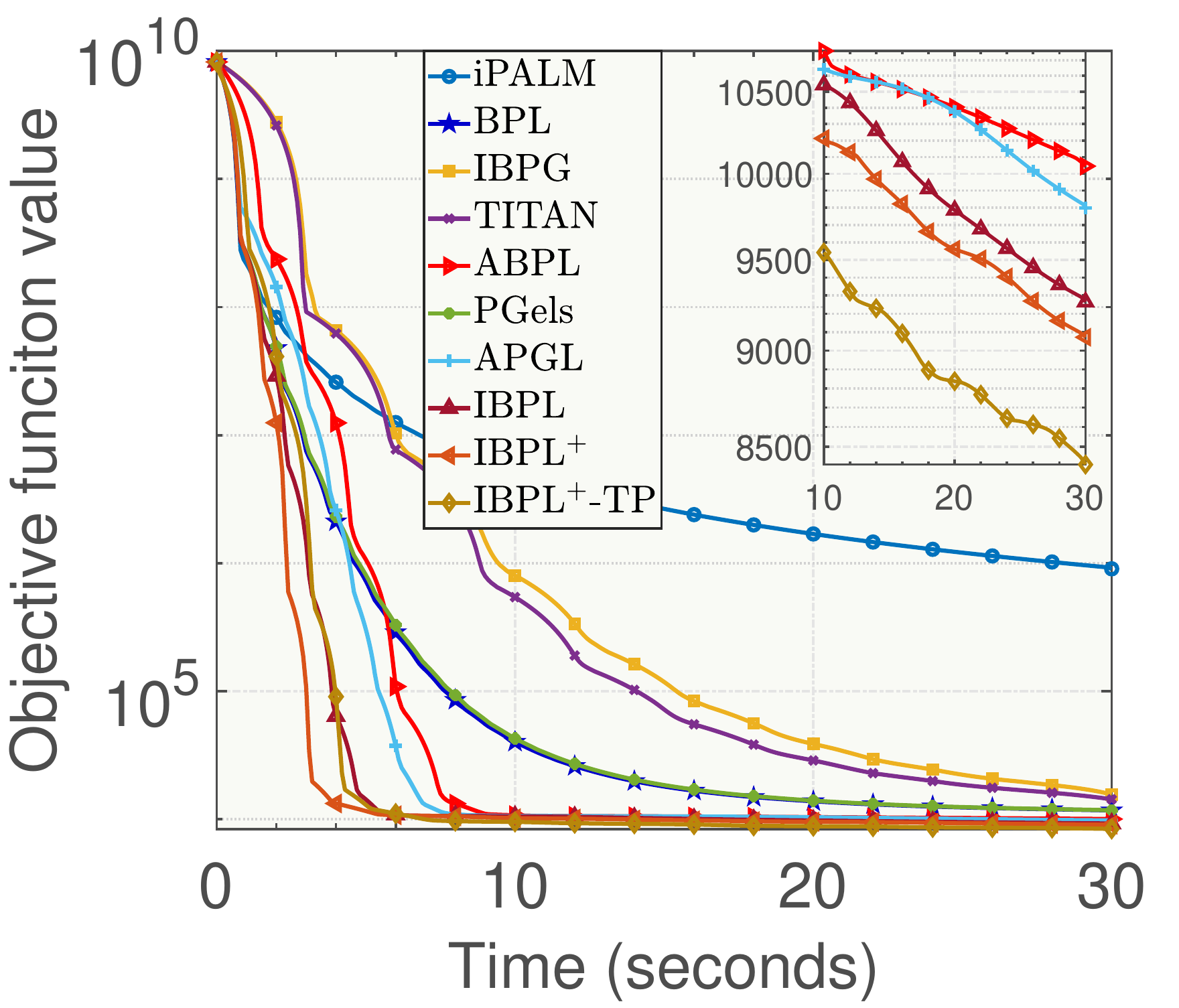}
    \includegraphics[width=0.33\linewidth]{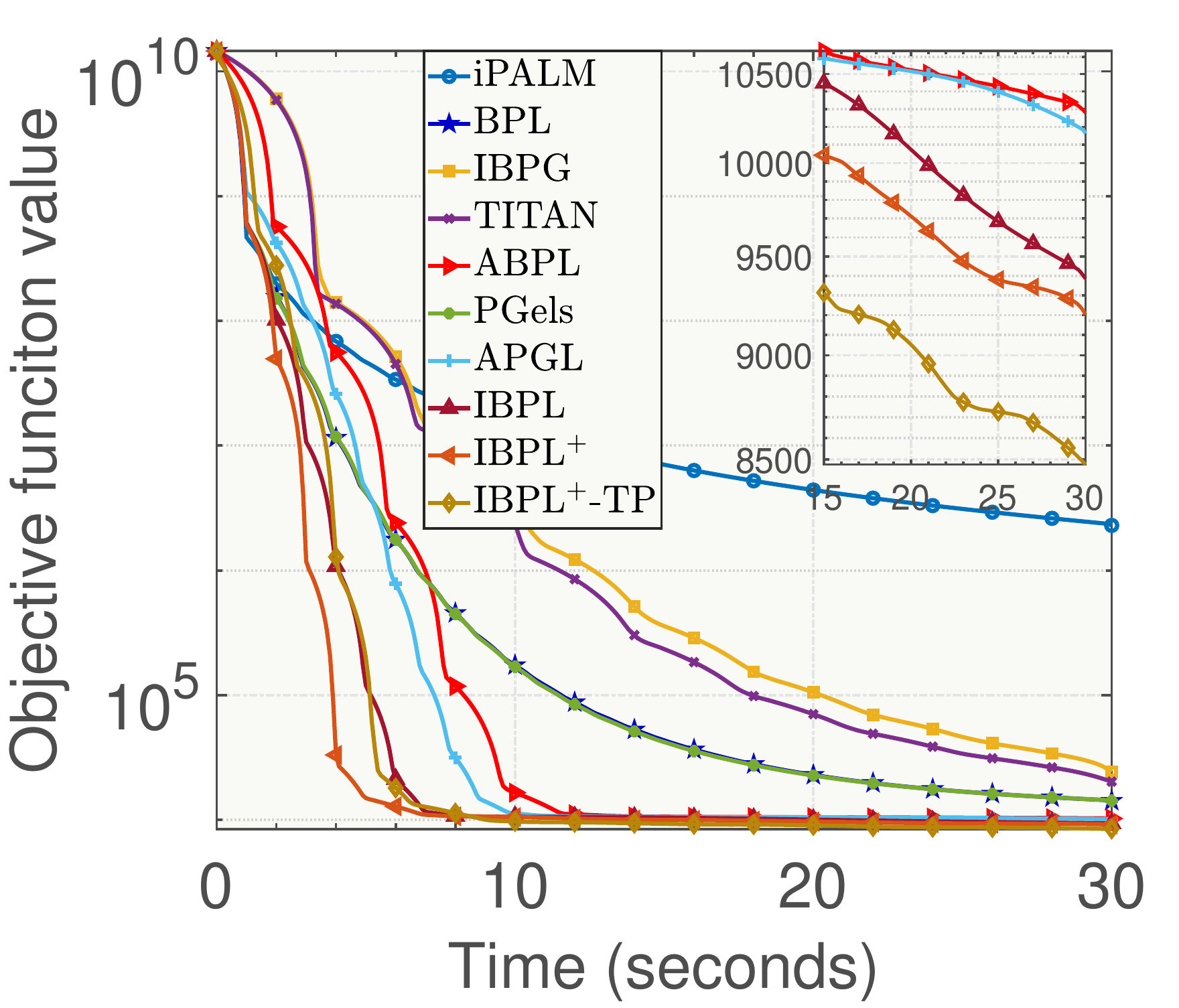}
    \includegraphics[width=0.33\linewidth]{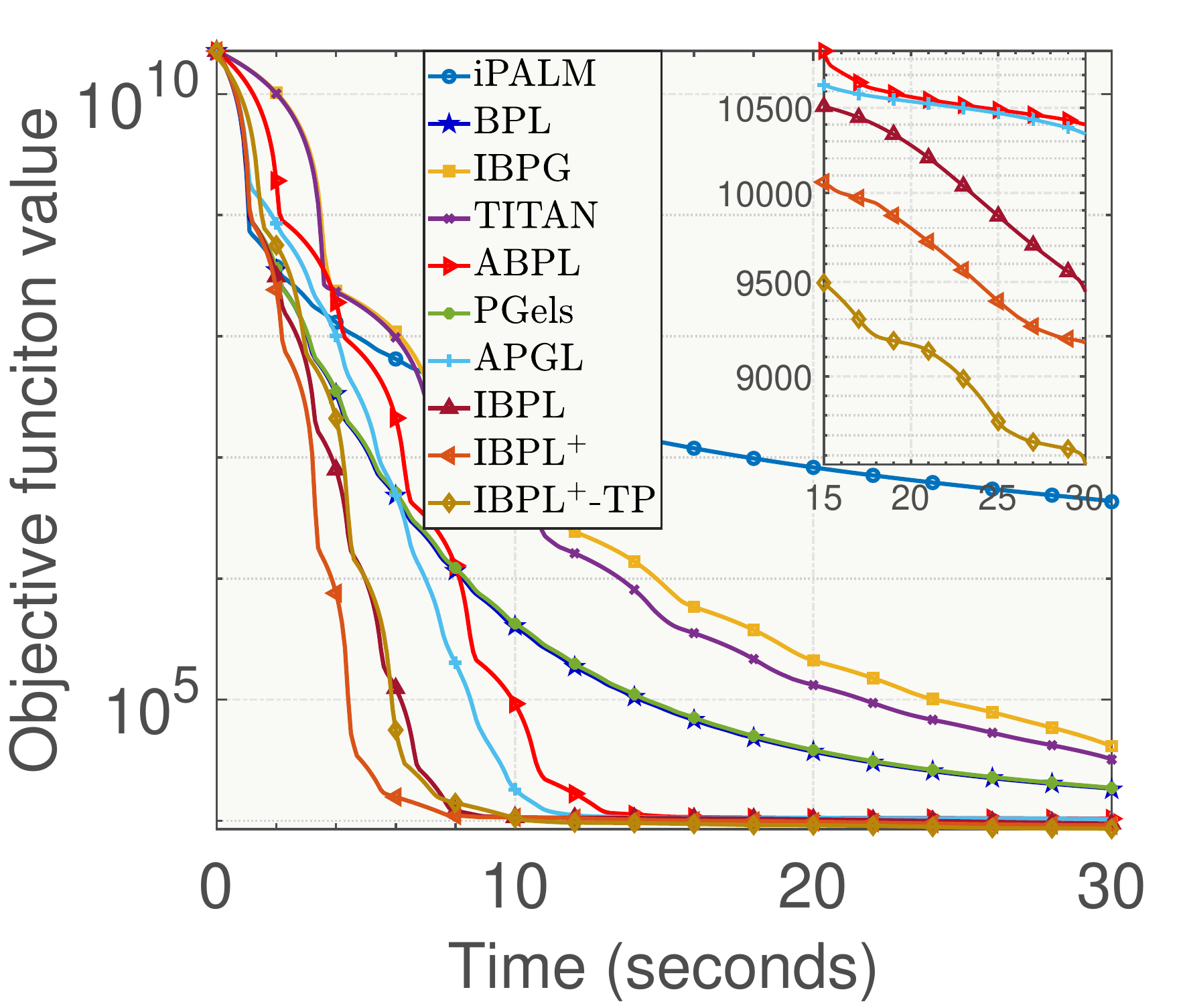}
    \caption{
     Average convergence behavior of all methods on Franz2 dataset with  $r=300~$(left), $r=400~$(middle) and $r=500~$(right). A partially enlarged view of several curves with the fastest descending speed and the best effect is also provided. 
    }\label{figmat2}
\end{figure*}

\begin{figure*}[!ht]
    \centering 
    \includegraphics[width=0.33\linewidth]{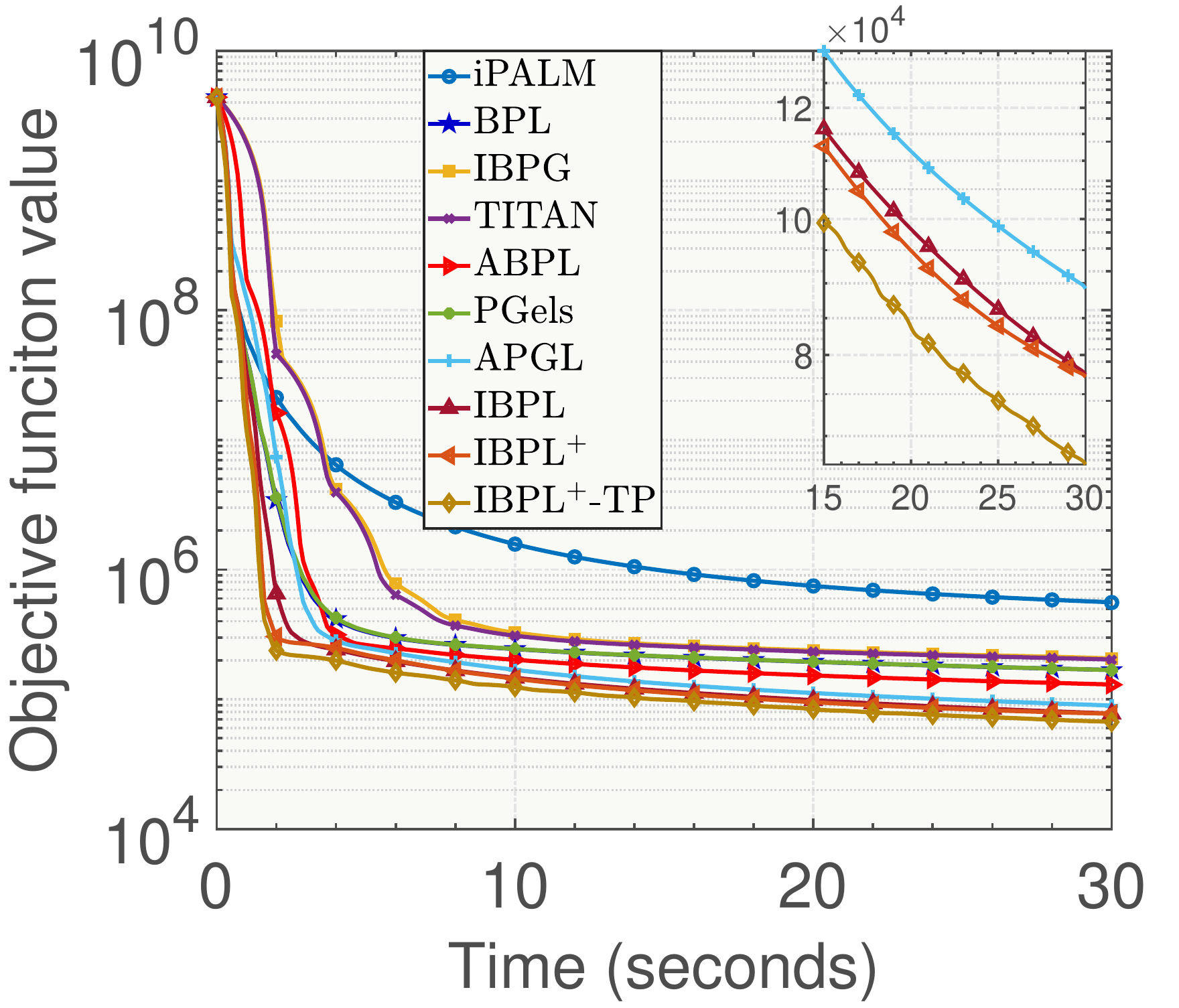}
    \includegraphics[width=0.33\linewidth]{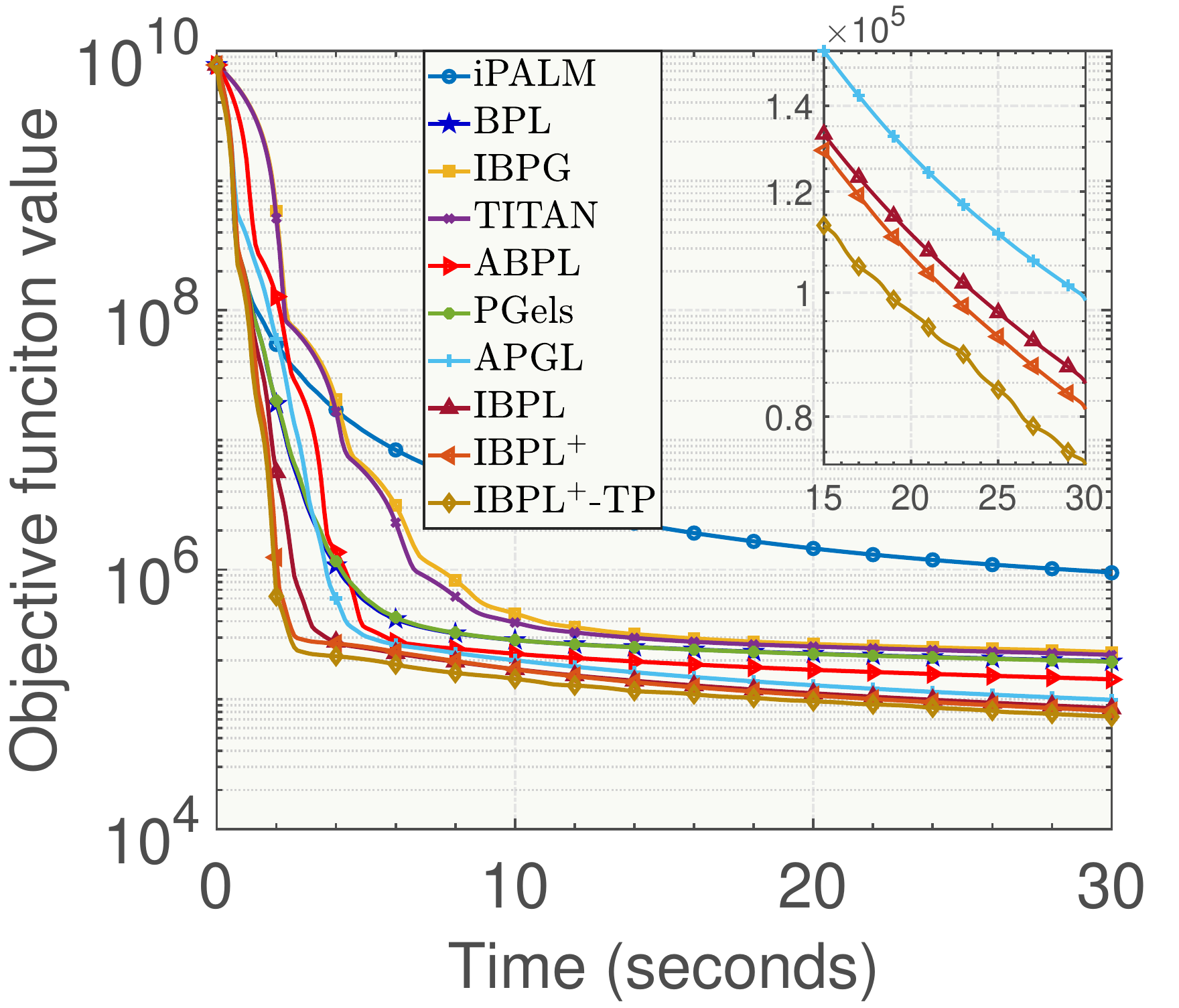}
    \includegraphics[width=0.33\linewidth]{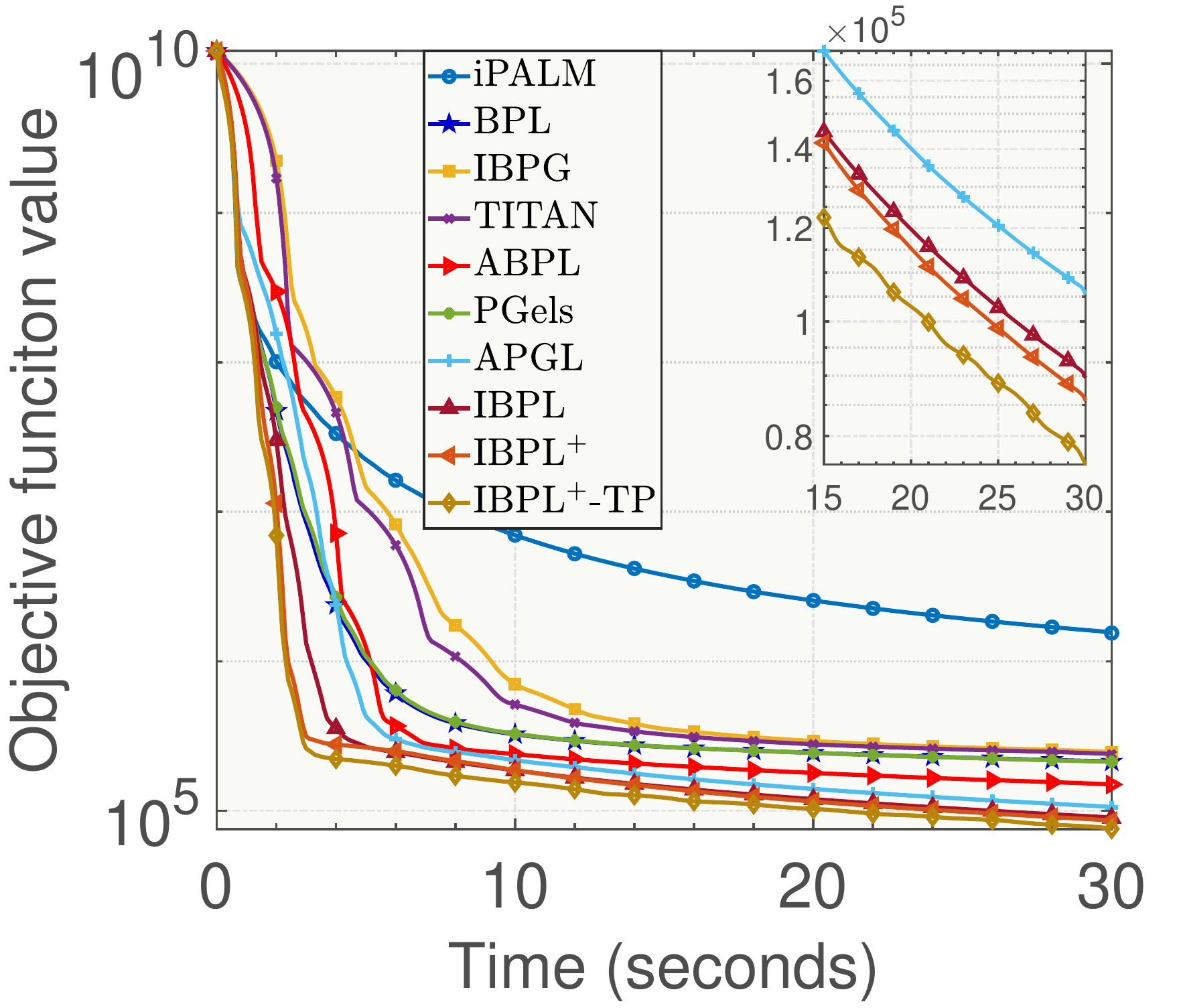}
    \caption{ 
     Average convergence behavior of all methods on BASEHOCK dataset with  $r=300~$(left), $r=400~$(middle) and $r=500~$(right). A partially enlarged view of several curves with the fastest descending speed and the best effect is also provided. 
    }\label{figmat3}
\end{figure*}

\subsection{Proposed Methods and Baseline Methods}
\label{subsection: asup}  
We first propose a transformation form of the proposed method to test the effectiveness of adaptive momentum in the proposed method.  

\begin{itemize} 
    \item [1)]IBPL: it is a version of IBPL$^+$ with non-adaptive momentum, which means that the extrapolation parameters $\beta^{k}_{i}$ are set as $t_{1}=1,~t_{k+1}=\frac{1+\sqrt{1+4t_{k}^{2}}}{2}, ~ \beta^{k}_{i}=\frac{t_{k}-1}{t_{k+1}}$ and the extrapolation parameters $\alpha^{k}_{i}$ are set as $\alpha^{k}_{i}=\min(1.03\beta^{k}_{i},\alpha_{max}) $.  

\end{itemize}

Furthermore, to demonstrate the effectiveness of our method, we compare our method and its transformation form with the following methods for solving Eq. (\ref{e11}).   
\begin{itemize} 
    
    \item [1)] iPALM  \cite{pock2016inertial}:  The inertial Proximal Alternating Linearized Minimization (iPALM) method. This method cannot guarantee that Eq. (\ref{e11}) decreases monotonically. 
    
    \item [2)] BPL  \cite{xu2017globally}:  Randomized/deterministic block prox-linear (BPL) method. 

    \item [3)] IBPG  \cite{le2020inertial}:  Inertial block proximal gradient (IBPG) method. This method cannot guarantee that Eq. (\ref{e11}) decreases monotonically.  

    \item [4)] TITAN  \cite{phan2023inertial}: inerTIal block majorizaTion minimizAtioN (TITAN) method.  

    \item [5)] ABPL  \cite{yang2023accelerated}:  Randomized/deterministic accelerated block proximal linear method with adaptive momentum. This method can only ensure the independence of the extrapolation parameters when only a single extrapolation point is used.

    \item [6)] PGels \cite{yang2024proximal}: Proximal gradient method with extrapolation and line search.

    \item [7)] APGL \cite{yang2025sparse}: Alternating Proximal Gradient Linearized method.  This method provides a more flexible version of the ABPL method that simplifies the two restart steps required in ABPL to just a single restart step.

\end{itemize}
All parameters involved in the compared methods are set as suggested in the reference papers. 
Notably, although there are several ADMM algorithms \cite{wang2019global,bai2025inexact,li2025proximal} that can also solve a class of multiblock nonconvex and nonsmooth problems,  the forms of the problems solved by these ADMM algorithms are completely different from Eq. (\ref{e11}).  
Similarly, while there are also some nonsmooth optimization algorithms \cite{jiang2025inexact,lyaqini2024non,xiao2024adam} for solving nonsmooth optimization problems, most of them can only solve convex problems and the forms of the problems solved by these nonsmooth optimization algorithms are also completely different from Eq. (\ref{e11}). 
Therefore, these ADMM algorithms and nonsmooth optimization algorithms are out of the scope of this paper.

\subsection{Sparse nonnegative matrix factorization with $\ell_0$-constraints ($\ell_0$-SNMF)} 
\subsubsection{$\ell_0$-SNMF model}
The NMF (nonnegative matrix factorization) model is an important tool for extracting hierarchical features, and the sparsity of the solution of NMF can force the model to extract only significant features while filtering out redundant information in the data \cite{li2025superpixel,yang2024low}.  
The sparsity is commonly achieved by using norm regularization terms, and the $\ell_{0}$ norm is the most intuitive representation of sparsity \cite{weston2003use,shi2022cardinality}. However, optimizing the Sparse NMF with $\ell_0$ norm constraints ($\ell_0$-SNMF) is known to be NP-hard, nonconvex and nonsmooth. 
To address this challenge, we will design a practical convergent scheme for solving $\ell_0$-SNMF based on our method and apply this scheme to solve the problem.

Mathematically, $\ell_0$-SNMF can be presented as follows. 
\begin{equation}
\begin{aligned}
&&  ~\min \limits_{U,V} &\frac{1}{2}\|X-UV\|_{F}^{2},\\
&& s. t. ~  U,V &\geq 0,~\|U\|_{0} \leq s_{1},~\|V\|_{0} \leq s_{2},   
\end{aligned}
\label{e41}
\end{equation}
where $X \in \mathbb{R}^{m\times n}$, $U\in\mathbb{R}^{m\times r}$ and $V\in\mathbb{R}^{r\times n}$.

\subsubsection{Solving $\ell_0$-SNMF using IBPL$^+$-TP}
If we write Eq. (\ref{e41}) in the form of Eq. (\ref{e11}), then we have 
\begin{align}
&&H(U,V)&=\frac{1}{2}\|X-UV\|_{F}^{2}, \notag \\
&& F_{1}(U)=\delta_{1}(U),~ & F_{2}(V)=\delta_{2}(V),\label{smNMFEq}
\end{align}
where
\begin{equation}
\delta_{i}(x)=\left\{
\begin{aligned}
0&, \quad x\geq 0,~\|x\|_{0} \leq s_{i}, \\
\infty&,\quad else.
\end{aligned}
\right. 
\label{delta}
\end{equation}
It is easy to see that $\nabla_{U} H(U,V)$ and $\nabla_{V} H(U,V)$ are 
\begin{align}
&& \nabla_{U} H(U,V)&=UVV^{T}-XV^{T}, \notag \\ 
&& \nabla_{V} H(U,V)&=U^{T}UV-U^{T}X.\notag
\end{align}
Therefore, the Lipschitz constant of $\nabla_{U} H(U,V)$ and $\nabla_{V} H(U,V)$ are 
\begin{align}
&& L(U)&=\|VV^T\|_{F},~L(V)=\|UU^T\|_{F}. \notag
\end{align}

Additionally, Frobenius norm satisfies Definition \ref{def_KL} \cite{bolte2014proximal}, Eq. (\ref{delta}) is a proper and lower semicontinuous KŁ function \cite{attouch2013convergence,bolte2014proximal}. Thus, Eq. (\ref{smNMFEq}) satisfies Assumption \ref{assump1}, and it is easy to verify that the proximal operator (Eq. (\ref{iprox})) for solving the $\ell_0$-SNMF can be expressed as follows. 
\begin{align} 
 && U^{k+1} &\in prox_{\sigma_{1}^{k} F_{1}}(B^{k}_{1}) \notag \\
 && &\in \operatorname*{\arg\min} \limits_{A} \left\{\frac{1}{2\sigma_{1}^{k}}\|A-B^{k}_{1}\|_{F}^{2}+F_{1}(A) \right\} \notag \\  
&& &\in  \operatorname*{\arg\min} \limits_{A}
 \left\{\|A-B^{k}_{1}\|_{F}^{2}: A\geq 0,~ \|A\|_{0}\leq s_{1}\right\}, \notag  \\
  && V^{k+1} &\in prox_{\sigma_{2}^{k} F_{2}}(B^{k}_{2}) \notag \\
 && &\in \operatorname*{\arg\min} \limits_{A} \left\{\frac{1}{2\sigma_{2}^{k}}\|A-B^{k}_{2}\|_{F}^{2}+F_{2}(A) \right\} \notag \\  
&& &\in  \operatorname*{\arg\min} \limits_{A}
 \left\{\|A-B^{k}_{2}\|_{F}^{2}: A\geq 0,~ \|A\|_{0}\leq s_{2}\right\}, \notag  
\end{align}
where 
\begin{center}
$B^{k}_{i}=z^{k}_{i}-\sigma_{i}^{k} \nabla_{x_{i}} H(\{x^{k+1}_{j}\}_{j=1}^{i-1},y^k_{i},\{x^{k}_{j}\}_{j=i+1}^{N})$, 
$ \frac{1}{\sigma_{1}^{k}}=\gamma^{k}_{1}L(U^k),~\frac{1}{\sigma_{2}^{k}}=\gamma^{k}_{2}L(V^k), ~ \gamma^{k}_{i}>1$.    
\end{center}

\subsubsection{Numerical results}
\label{NMFNum}
In this experiment, we test the methods on the datasets from  SuiteSparse\footnote{\url{https://sparse.tamu.edu/}} \cite{10.1145/2049662.2049663}:  lp\_ship12l ($\mathbb{R}^{1153\times5533}$) and Franz2 ($\mathbb{R}^{4032\times4480}$), which are widely used sets of sparse matrix benchmarks, and we also test the methods on the BASEHOCK\footnote{\url{https://jundongl.github.io/scikit-feature/datasets.html}} dataset ($\mathbb{R}^{1993\times4862}$) \cite{lang1995learning}, which is a benchmark text dataset in machine learning and data mining.

For parameter settings, we set the initial parameters as  
$t_{2}=1.1, ~ \beta^{1}_{i}=0.6,~ \alpha^{k}_{i}=\min(1.03\beta^{k}_{i},\alpha_{max}), ~ \alpha_{max}=\beta_{max}=0.9999, ~ \gamma^{k}_{i}=1.01,~\rho_{1}=10^{-5}$, $T=10^{-3}$, 
and the number of non-zero elements in each matrix cannot exceed $30\%$ of the total number of elements. The initial matrix is generated by the uniform distribution. For evaluation metrics, we define Obj as the objective function value, relative error (Rel) as Rel$=\frac{\|X-UV\|_{F}}{\|X\|_{F}}$. $Ranking$ is defined as the number of times that the corresponding method obtains the lowest relative error among all methods across all independent runs. We adopt these metrics to quantitatively describe the numerical performance of all methods in this experiment.

We define the maximum running time as $t_{max}$ (s) and set $t_{max}=30$ in this experiment. 
Table \ref{Tabmat} and Table \ref{Tabmat2} report the average results over 20 independent runs with $r$ varying among $\{300,400,500\}$ on these datasets. Figure \ref{figmat1}, Figure \ref{figmat2} and Figure \ref{figmat3} also record the evolution of average objective function value over 20 independent runs with respect to time on these datasets.  
From Table \ref{Tabmat}, Table \ref{Tabmat2},  Figure \ref{figmat1}, Figure \ref{figmat2} and Figure \ref{figmat3}, we observe that  
\begin{itemize}

    \item [1)] \textit{IBPL outperforms all other methods except IBPL$^+$ and IBPL$^+$-TP in all cases}: 
    This shows that using two different extrapolation points and the independence of extrapolation parameters of these extrapolation points can effectively enhance the numerical performance. 

    \item [2)] \textit{IBPL$^+$ outperforms all other methods except IBPL$^+$-TP in all cases}: 
    This demonstrates that when using two different extrapolation points with independent extrapolation parameters, incorporating an adaptive momentum update strategy for the extrapolation parameters into the proposed method can further accelerate convergence and improve numerical performance. 

    \item [3)] \textit{IBPL$^+$-TP remarkably outperforms all other methods in all cases}: 
    This not only reaffirms the benefits of using two different extrapolation points with independent extrapolation parameters and the adaptive momentum update strategy for the extrapolation parameters, but also shows that introducing a two-phase adaptive momentum update strategy, which explicitly separates the rapid descent phase from the steady refinement phase to update the extrapolation parameters effectively, provides an additional boost in convergence speed and numerical robustness, making IBPL$^+$-TP the most effective method overall.
     
\end{itemize}



\begin{table*}[!t]
\belowrulesep=0pt\aboverulesep=0pt
\centering
\caption{\centering Average results by methods on BreastMNIST and PosteriorEmission datasets. The best performance is highlighted in bold.}
\label{Tabten}
\begin{tabular*}{1.010\linewidth}{p{1.6cm}| c c c|c c c} 
\hline      
\diagbox[width=5.74em]{\multirow{2}{*}{Method}}{Dataset} & \multicolumn{3}{c|}{BreastMNIST}  & \multicolumn{3}{c}{PosteriorEmission}     \\ \cline{2-7} 
& Obj & Rel& Ranking & Obj & Rel& Ranking  \\
\midrule   
 
    iPALM & 3.923e+10$\pm$3.990e+09 & 0.442$\pm$0.023 & 0 & 1.910e+08$\pm$5.426e+06 & 0.715$\pm$0.010 &  0\\
    BPL   & 1.909e+10$\pm$9.550e+08 & 0.309$\pm$0.008 & 0& 1.587e+08$\pm$6.439e+06 & 0.652$\pm$0.013 &  0\\
    IBPG  & 1.751e+10$\pm$8.503e+08 & 0.296$\pm$0.007 & 0& 1.526e+08$\pm$3.680e+06 & 0.639$\pm$0.008 & 0 \\
    TITAN & 1.756e+10$\pm$1.042e+09 & 0.296$\pm$0.009 & 0& 1.498e+08$\pm$2.779e+06 & 0.633$\pm$0.006 &  0\\
    ABPL  & 1.826e+10$\pm$1.018e+09 & 0.302$\pm$0.008 & 0& 1.559e+08$\pm$4.424e+06 & 0.646$\pm$0.009 &  0\\
    PGels & 1.906e+10$\pm$5.450e+08 & 0.308$\pm$0.004 & 0& 1.596e+08$\pm$5.878e+06 & 0.654$\pm$0.012 &  0\\
    APGL  & 1.566e+10$\pm$2.307e+08 & 0.280$\pm$0.002 & 0& 1.484e+08$\pm$2.265e+06 & 0.630$\pm$0.005 &  0\\
    IBPL  & 1.550e+10$\pm$2.404e+08 & 0.278$\pm$0.002 & 0& 1.463e+08$\pm$2.443e+06 & 0.626$\pm$0.005 &  1\\

 IBPL$^+$ & 1.492e+10$\pm$4.234e+08 & 0.273$\pm$0.004 & 2  & 1.447e+08$\pm$1.192e+06 & 0.623$\pm$0.003 & 3 \\  


 IBPL$^+$-TP &
\textbf{1.340e+10} $\pmb{\pm}$ \textbf{2.074e+08} &
 \textbf{0.259} $\pmb{\pm}$ \textbf{0.002} &
\textbf{18} &
 \textbf{1.384e+10} $\pmb{\pm}$ \textbf{1.287e+08} &
 \textbf{0.609} $\pmb{\pm}$ \textbf{0.003} &
 \textbf{16} \\
\hline 
\end{tabular*}
\subcaption{\centering Average results by methods on BreastMNIST and PosteriorEmission datasets with $R=50$.}

\belowrulesep=0pt\aboverulesep=0pt
\centering
\begin{tabular*}{1.010\linewidth}{p{1.6cm}| c c c|c c c} 
\hline      
\diagbox[width=5.74em]{\multirow{2}{*}{Method}}{Dataset} & \multicolumn{3}{c|}{BreastMNIST}  & \multicolumn{3}{c}{PosteriorEmission}     \\ \cline{2-7}
& Obj & Rel& Ranking & Obj & Rel& Ranking  \\
\midrule   
 
    iPALM & 3.786e+10$\pm$4.324e+09 & 0.434$\pm$0.024& 0 & 1.936e+08$\pm$4.403e+06 & 0.720$\pm$0.008 &  0\\
    BPL   & 1.891e+10$\pm$1.450e+09 & 0.307$\pm$0.012& 0 & 1.600e+08$\pm$8.465e+06 & 0.654$\pm$0.017 &  0\\
    IBPG  & 1.764e+10$\pm$1.074e+09 & 0.297$\pm$0.009& 0 & 1.531e+08$\pm$4.174e+06 & 0.640$\pm$0.009 &  0\\
    TITAN & 1.716e+10$\pm$1.310e+09 & 0.293$\pm$0.011& 0 & 1.524e+08$\pm$5.495e+06 & 0.639$\pm$0.011 &  0\\
    ABPL  & 1.820e+10$\pm$9.756e+08 & 0.301$\pm$0.008& 0 & 1.588e+08$\pm$6.088e+06 & 0.652$\pm$0.013 &  0\\
    PGels & 1.881e+10$\pm$4.553e+08 & 0.306$\pm$0.004& 0 & 1.635e+08$\pm$7.625e+06 & 0.662$\pm$0.015 &  0\\
    APGL  & 1.500e+10$\pm$2.233e+08 & 0.274$\pm$0.002& 0 & 1.489e+08$\pm$4.060e+06 & 0.631$\pm$0.009 &  0\\
    IBPL  & 1.482e+10$\pm$2.294e+08 & 0.272$\pm$0.002& 0 & 1.444e+08$\pm$2.602e+06 & 0.622$\pm$0.006 & 1 \\

 IBPL$^+$ & 1.417e+10$\pm$2.924e+08 & 0.266$\pm$0.003 & 1  & 1.424e+08$\pm$2.244e+06 & 0.617$\pm$0.005 & 2 \\ 
 

 IBPL$^+$-TP &
 \textbf{1.298e+10} $\pmb{\pm}$ \textbf{1.717e+08} &
\textbf{0.255} $\pmb{\pm}$ \textbf{0.002} &
\textbf{19} &
\textbf{1.369e+08} $\pmb{\pm}$ \textbf{1.654e+06} &
\textbf{0.606} $\pmb{\pm}$ \textbf{0.004}&
\textbf{18} \\

 \hline

\end{tabular*}
\subcaption{\centering Average results by methods on BreastMNIST and PosteriorEmission datasets with $R=60$.}

\belowrulesep=0pt\aboverulesep=0pt
\centering
\begin{tabular*}{1.010\linewidth}{p{1.6cm}| c c c|c c c} 
\hline      
\diagbox[width=5.74em]{\multirow{2}{*}{Method}}{Dataset} & \multicolumn{3}{c|}{BreastMNIST}  & \multicolumn{3}{c}{PosteriorEmission}     \\ \cline{2-7}
& Obj & Rel& Ranking & Obj & Rel& Ranking  \\
\midrule   
 
    iPALM & 3.684e+10$\pm$3.170e+09 & 0.428$\pm$0.018& 0 & 1.938e+08$\pm$3.015e+06 & 0.720 $\pm$0.006 &  0\\
    BPL   & 1.822e+10$\pm$8.670e+08 & 0.302$\pm$0.007& 0 & 1.604e+08$\pm$5.153e+06 & 0.655$\pm$0.010 &  0\\
    IBPG  & 1.741e+10$\pm$1.267e+09 & 0.295$\pm$0.011& 0 & 1.515e+08$\pm$4.722e+06 & 0.637$\pm$0.010 &  0\\
    TITAN & 1.644e+10$\pm$1.195e+09 & 0.286$\pm$0.010& 0 & 1.510e+08$\pm$4.097e+06 & 0.636$\pm$0.009 & 0 \\
    ABPL  & 1.836e+10$\pm$8.292e+08 & 0.303$\pm$0.007& 0 & 1.603e+08$\pm$7.561e+06 & 0.655$\pm$0.015 &  0\\
    PGels & 1.846e+10$\pm$5.708e+08 & 0.304$\pm$0.005& 0 & 1.636e+08$\pm$6.250e+06 & 0.662$\pm$0.013 &  0\\
    APGL  & 1.452e+10$\pm$2.031e+08 & 0.269$\pm$0.002& 0 & 1.480e+08$\pm$2.779e+06 & 0.629$\pm$0.006 &  0\\
    IBPL  & 1.431e+10$\pm$2.079e+08 & 0.267$\pm$0.002& 1 & 1.449e+08$\pm$3.340e+06 & 0.623$\pm$0.007 &  0\\

    IBPL$^+$ & 1.382e+10$\pm$3.737e+08 & 0.263$\pm$0.004 & 2  & 1.423e+08$\pm$1.852e+06 & 0.617$\pm$0.004 & 1 \\ 


 IBPL$^+$-TP&
\textbf{1.270e+10} $\pmb{\pm}$ \textbf{1.928e+08} &
\textbf{0.252} $\pmb{\pm}$ \textbf{0.002} &
\textbf{17} &
\textbf{1.358e+08} $\pmb{\pm}$ \textbf{1.763e+06} &
\textbf{0.603} $\pmb{\pm}$ \textbf{0.004} &
\textbf{19}\\
 
\hline 

\end{tabular*}
\subcaption{\centering Average results by methods on BreastMNIST and PosteriorEmission datasets with $R=70$.}
\end{table*}

\begin{table*}[!t]
\belowrulesep=0pt\aboverulesep=0pt
\centering
\caption{\centering Average results by methods on microPNW dataset. The best performance is highlighted in bold.}
\label{Tabten2}
\belowrulesep=0pt\aboverulesep=0pt
\centering
\begin{tabular*}{0.889\linewidth}{p{1.6cm}| c c c|c c c} 
\hline      
\diagbox[width=5.74em]{\multirow{2}{*}{Method}}{Rank} & \multicolumn{3}{c|}{R=50}  & \multicolumn{3}{c}{R=60}     \\ \cline{2-7}
& Obj & Rel& Ranking & Obj & Rel& Ranking  \\
\midrule   
 
    iPALM & 327.703$\pm$21.194 & 0.032$\pm$0.001 &0 & 341.241$\pm$25.401 & 0.033$\pm$0.001 &0 
     \\
  
    BPL   & 181.355$\pm$13.133 & 0.024$\pm$0.001 &0 & 188.587$\pm$23.319 & 0.024$\pm$0.001 &0
     \\
    
    IBPG  & 164.771$\pm$12.327 & 0.023$\pm$0.001 &0 & 157.803$\pm$14.614 & 0.022$\pm$0.001 &0
     \\
    
    TITAN & 158.653$\pm$12.225 & 0.022$\pm$0.001 &0 & 157.464$\pm$15.139 & 0.022$\pm$0.001 &0
     \\
    
    ABPL  & 108.637$\pm$16.714 & 0.018$\pm$0.001 &0 & 93.848$\pm$10.788 & 0.017$\pm$0.001 &0
      \\
    
    PGels & 188.777$\pm$23.372 & 0.024$\pm$0.001 &0 & 199.243$\pm$23.889 & 0.025$\pm$0.001 &0
      \\
    
    APGL  & 90.267$\pm$8.961 & 0.017$\pm$0.001 &2 & 81.055$\pm$11.975 & 0.016$\pm$0.001 &1
     \\
    
    IBPL  & 84.102$\pm$6.359 & 0.016$\pm$0.001 &4 & 76.162$\pm$7.323 & 0.015$\pm$0.001 &3
      \\
    
    IBPL$^+$ & 82.568$\pm$4.699 & 0.016$\pm$0.001 &5 & 71.855$\pm$7.684 & 0.015$\pm$0.001 &5
      \\
    
    IBPL$^+$-TP & \textbf{78.305} $\pmb{\pm}$ \textbf{5.124} & \textbf{0.015} $\pmb{\pm}$ \textbf{0.001} &\textbf{9} & \textbf{66.853} $\pmb{\pm}$ \textbf{6.723} & \textbf{0.014} $\pmb{\pm}$ \textbf{0.001} &\textbf{11}  \\

\hline 
\end{tabular*}

\vspace{5pt}

\centering
\begin{tabular*}{0.5025\linewidth}{p{1.6cm}| c c c} 
\hline      
\diagbox[width=5.74em]{\multirow{2}{*}{Method}}{Rank} & \multicolumn{3}{c}{R=70}      \\ \cline{2-4}
& Obj & Rel& Ranking \\
\midrule   
 
    iPALM  & 360.060$\pm$35.019 & 0.034$\pm$0.002 &  0\\
  
    BPL   & 183.856$\pm$26.585 & 0.024$\pm$0.002 &  0\\
    
    IBPG  & 161.444$\pm$16.288 & 0.023$\pm$0.001 & 0 \\
    
    TITAN & 155.195$\pm$14.071 & 0.022$\pm$0.001 &  0\\
    
    ABPL   & 89.397$\pm$13.593 & 0.017$\pm$0.001 &  0\\
    
    PGels  & 200.730$\pm$25.827 & 0.025$\pm$0.002 & 0 \\
    
    APGL   & 67.675$\pm$6.814 & 0.015$\pm$0.001 &  2\\
    
    IBPL  & 66.607$\pm$6.623 & 0.014$\pm$0.001 &  3\\
    
    IBPL$^+$  & 64.320$\pm$7.179 & 0.014$\pm$0.001 &  6\\
    
    IBPL$^+$-TP  & \textbf{58.633} $\pmb{\pm}$ \textbf{6.101} & \textbf{0.014} $\pmb{\pm}$  \textbf{0.001} & \textbf{9} \\
 
\hline 
\end{tabular*}

\end{table*}

\vspace{5pt}

\begin{figure*}[!ht]
    \centering 
    \includegraphics[width=0.33\linewidth]{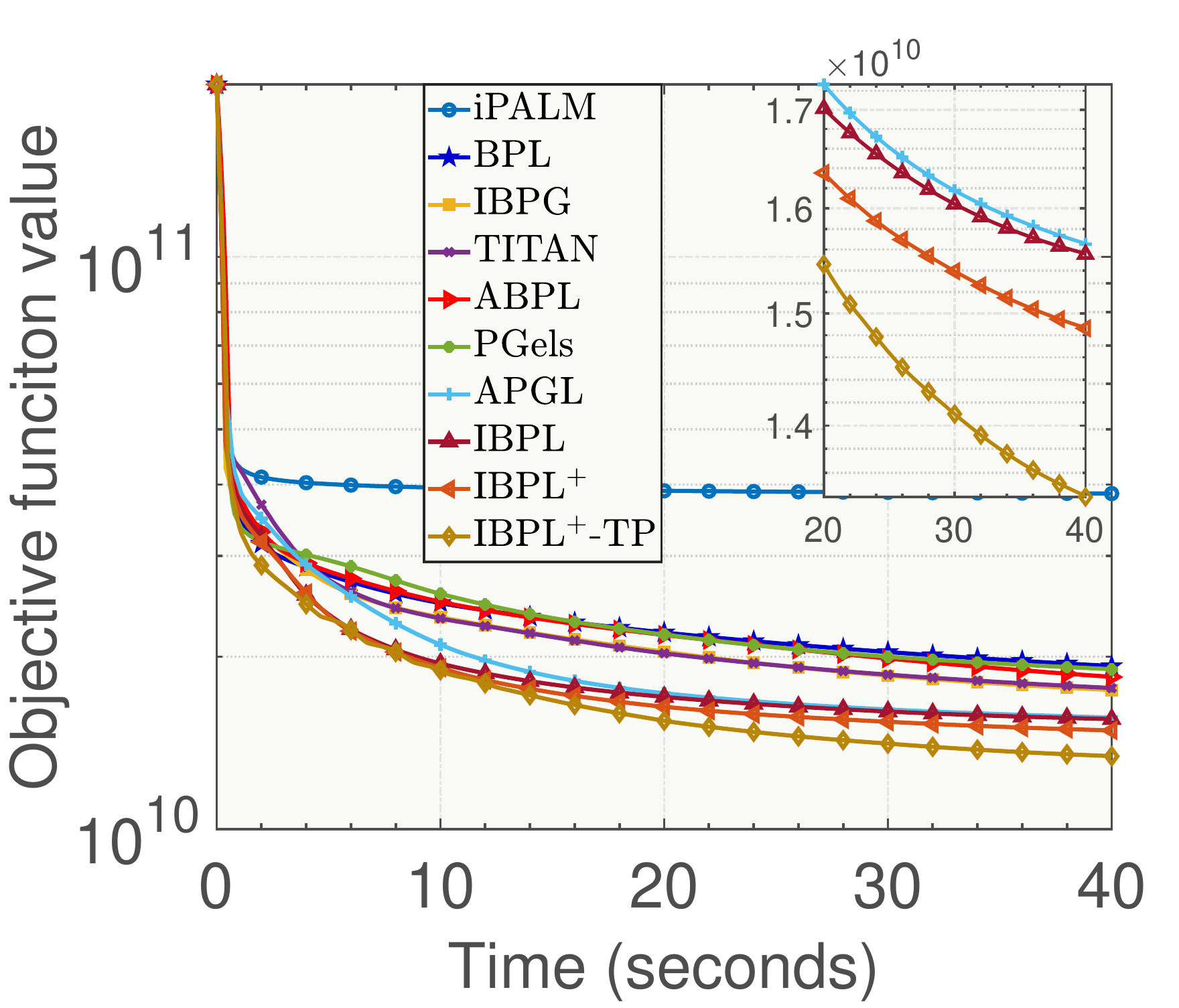}
    \includegraphics[width=0.33\linewidth]{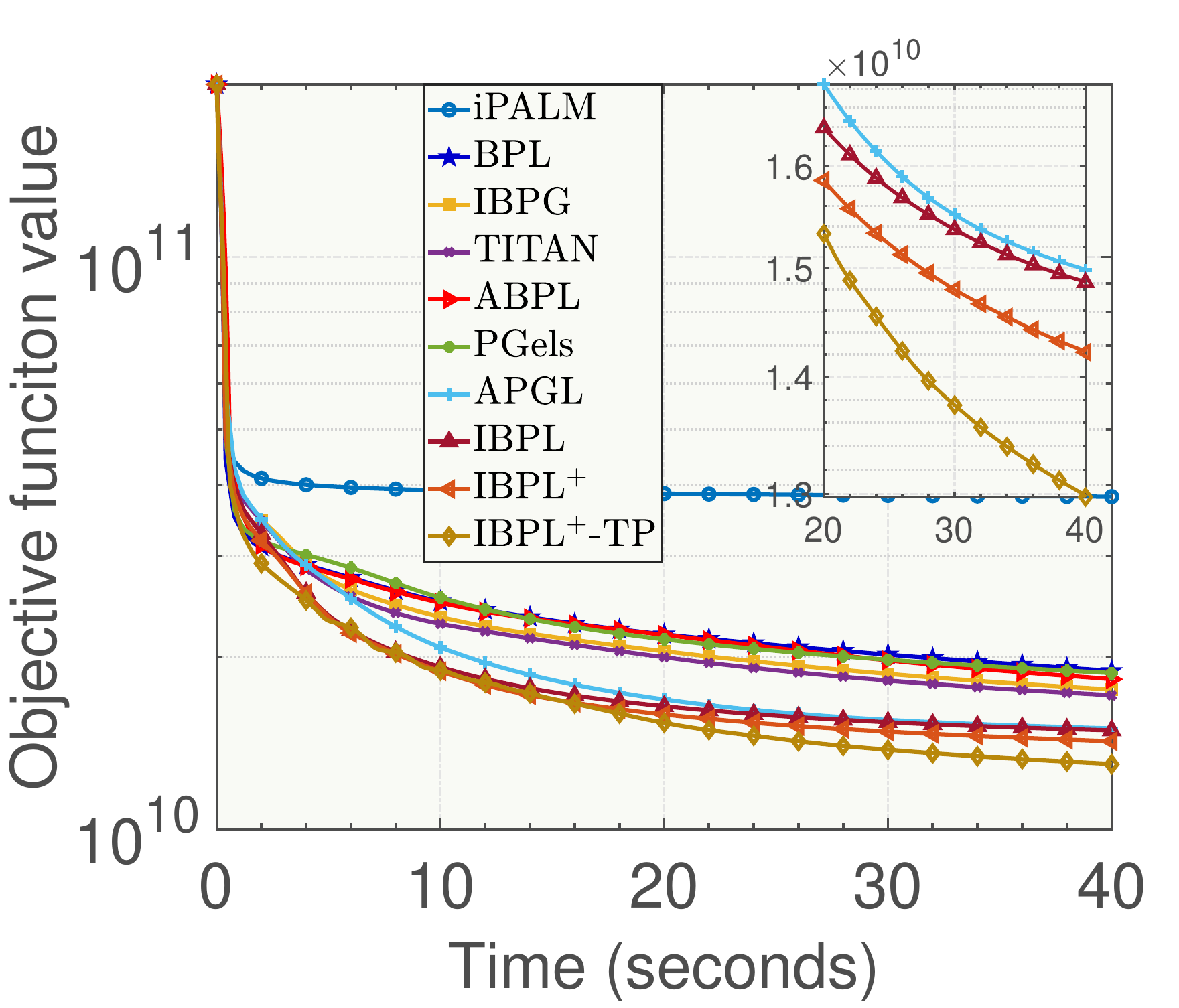}
    \includegraphics[width=0.33\linewidth]{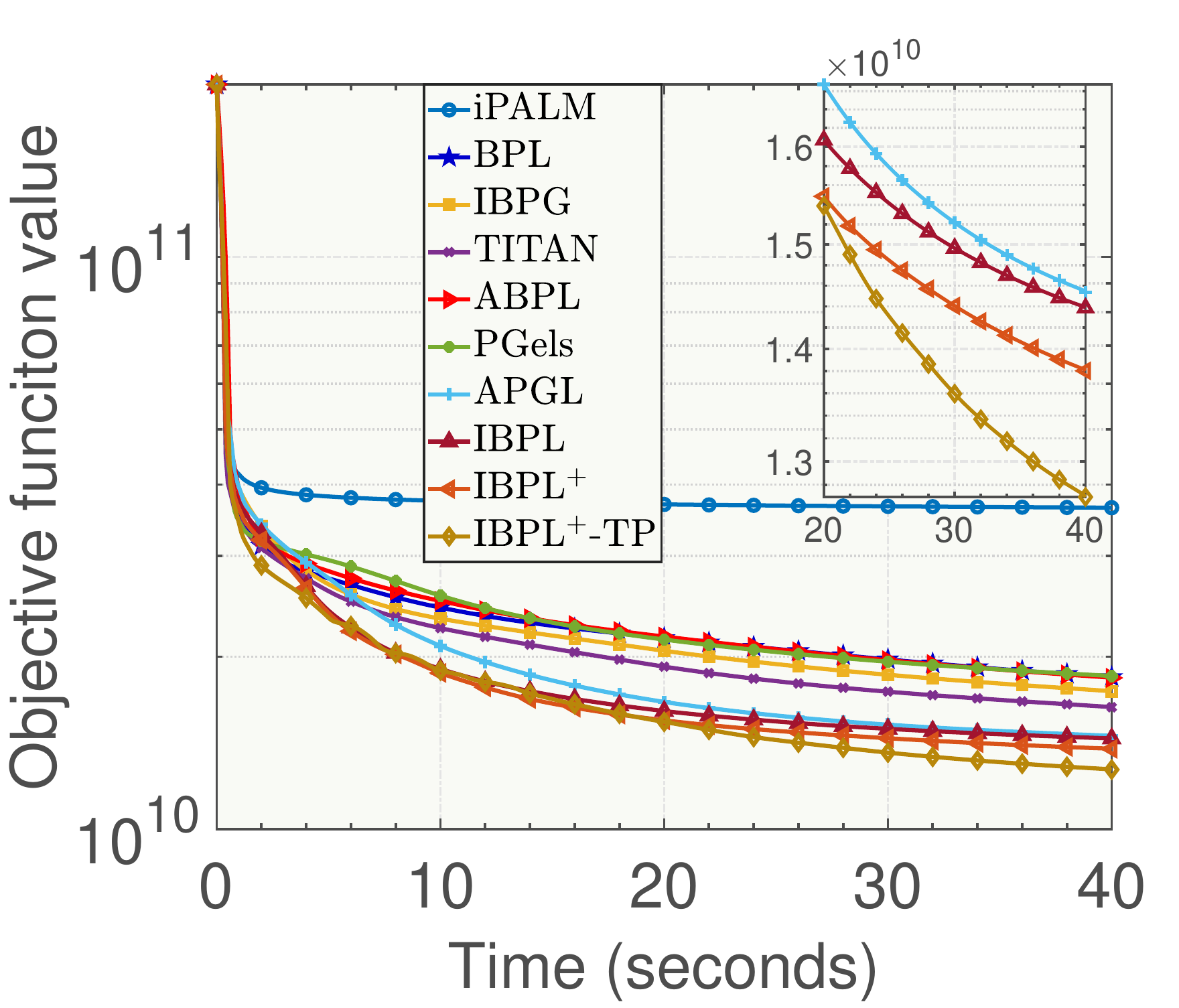}
    \caption{
    Average convergence behavior of all methods on BreastMNIST dataset with $R=50~$(left), $R=60~$(middle) and $R=70~$(right). A partially enlarged view of several curves with the fastest descending speed and the best effect is also provided. 
    }\label{figten1}
\end{figure*}

\begin{figure*}[!ht]
    \centering 
    \includegraphics[width=0.33\linewidth]{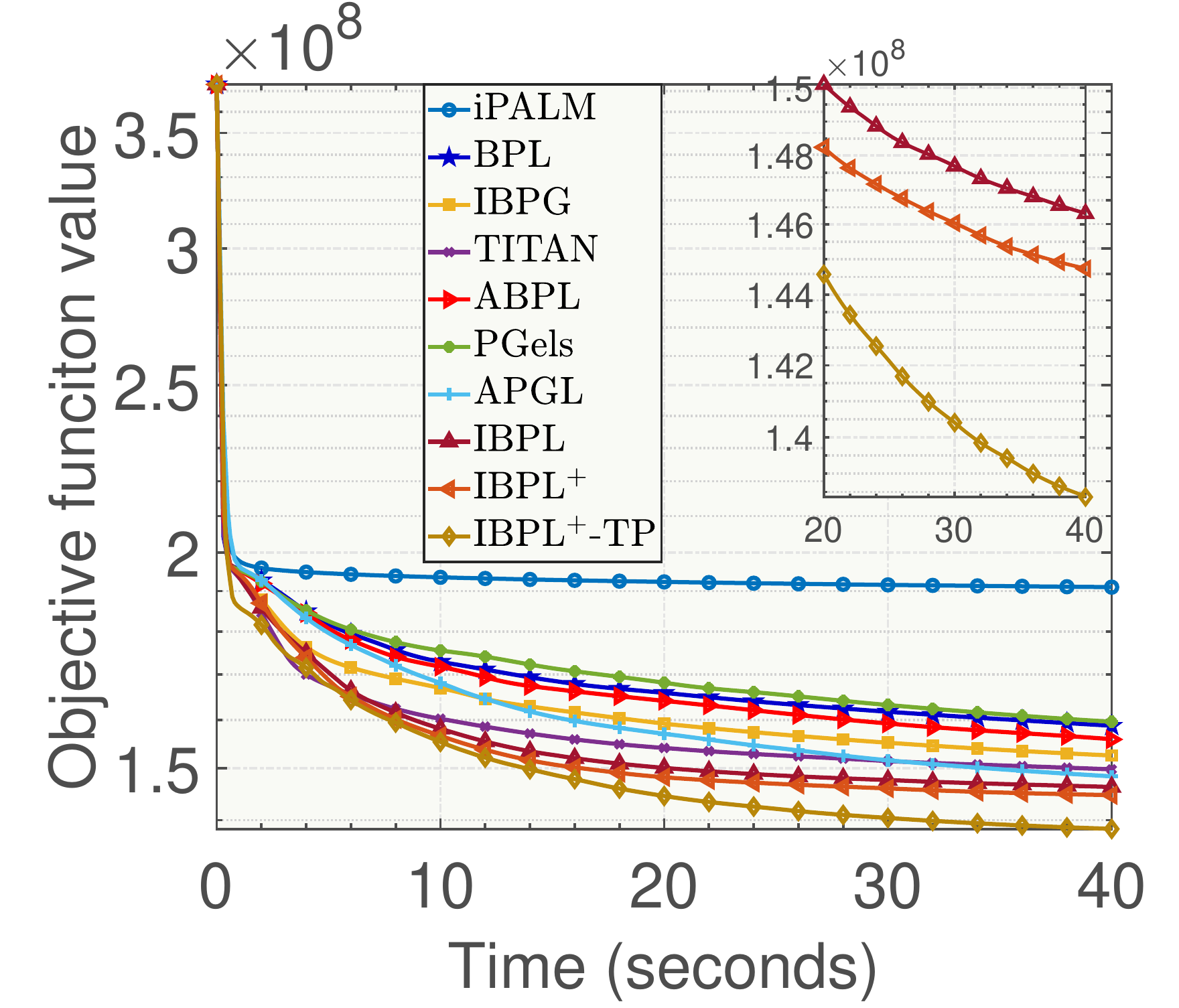}
    \includegraphics[width=0.33\linewidth]{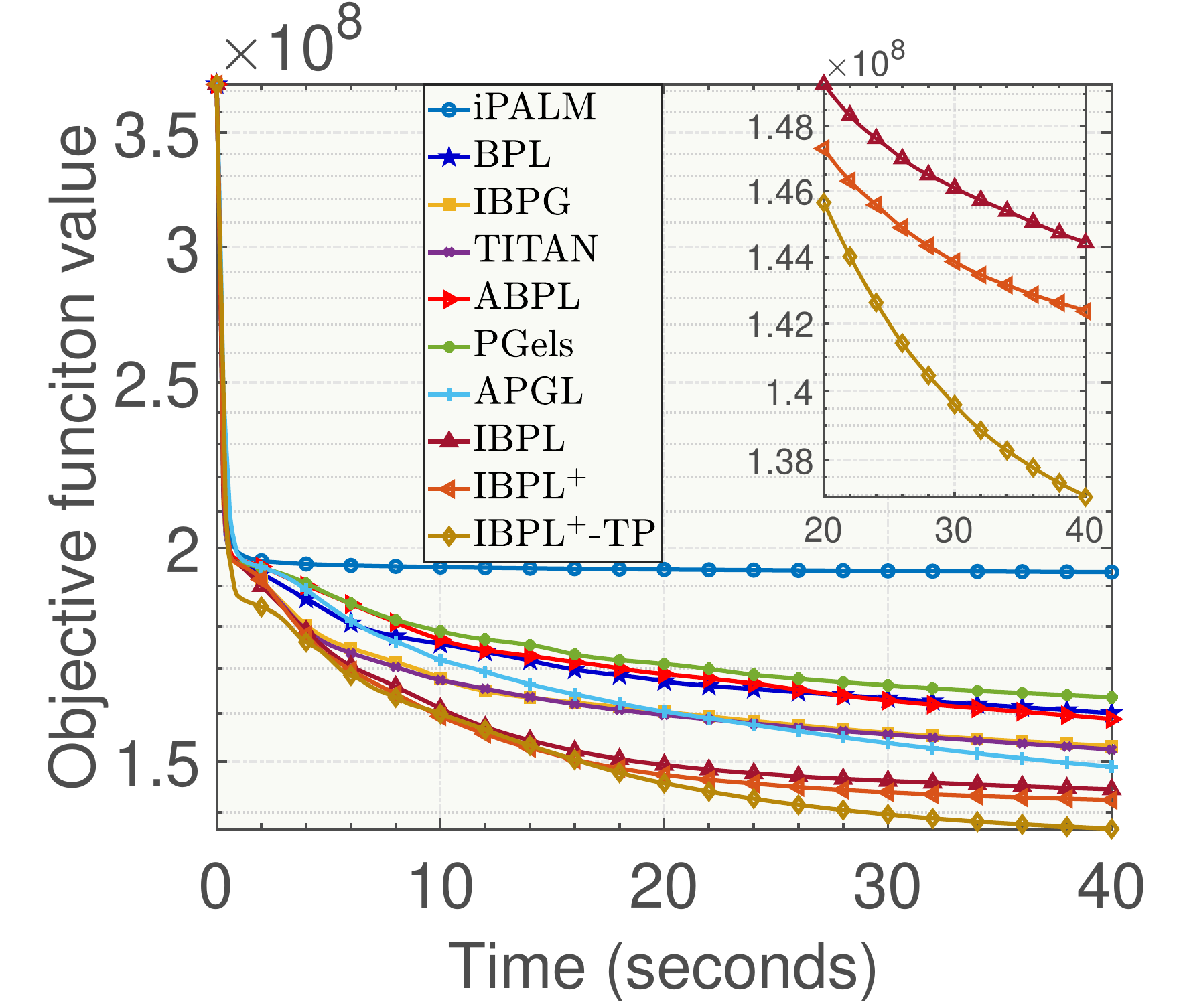}
    \includegraphics[width=0.33\linewidth]{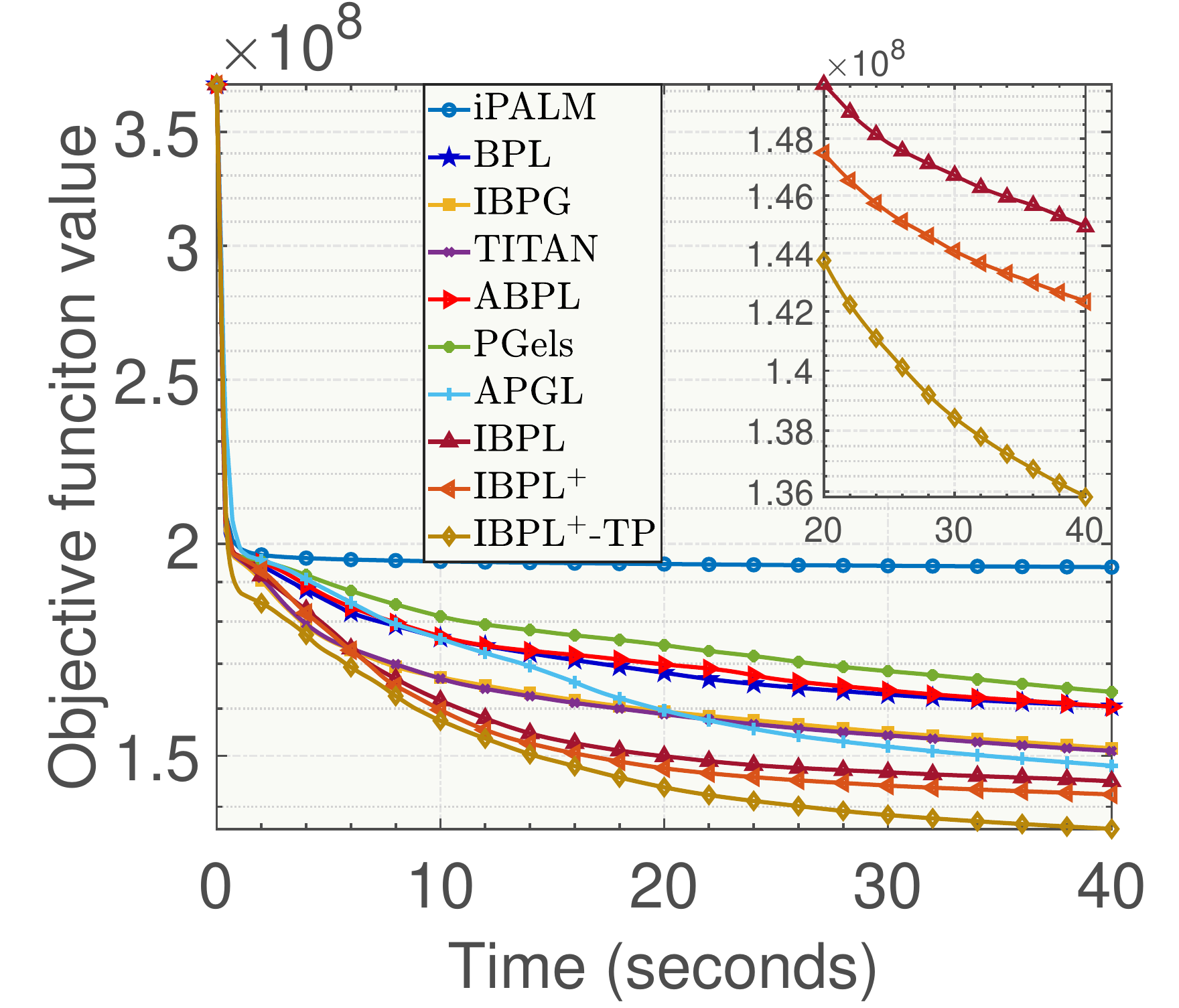}
    \caption{
     Average convergence behavior of all methods on PosteriorEmssion dataset with  $R=50~$(left), $R=60~$(middle) and $R=70~$(right). A partially enlarged view of several curves with the fastest descending speed and the best effect is also provided. 
    }\label{figten2}
\end{figure*}

\begin{figure*}[!ht]
    \centering 
    \includegraphics[width=0.33\linewidth]{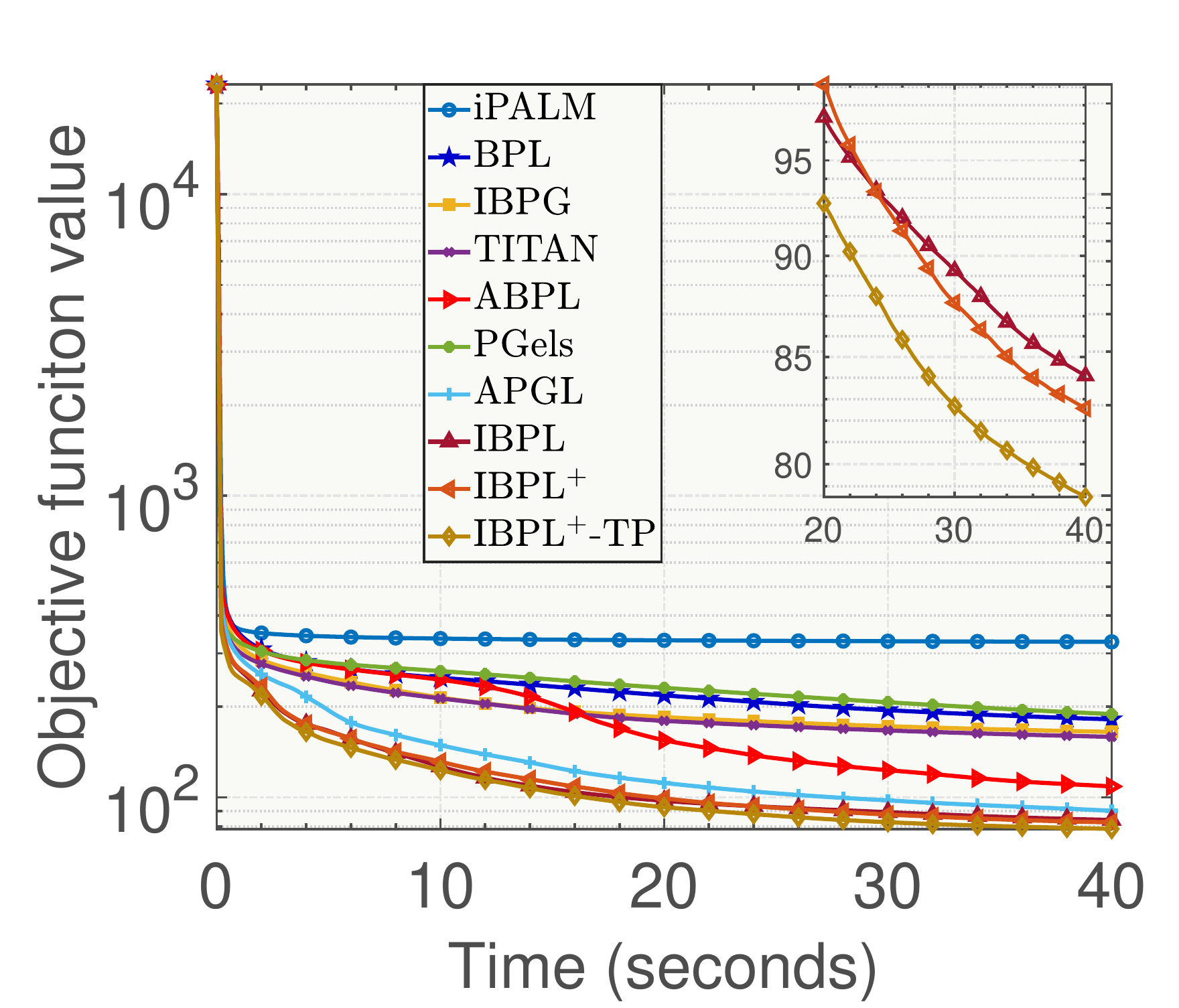}
    \includegraphics[width=0.33\linewidth]{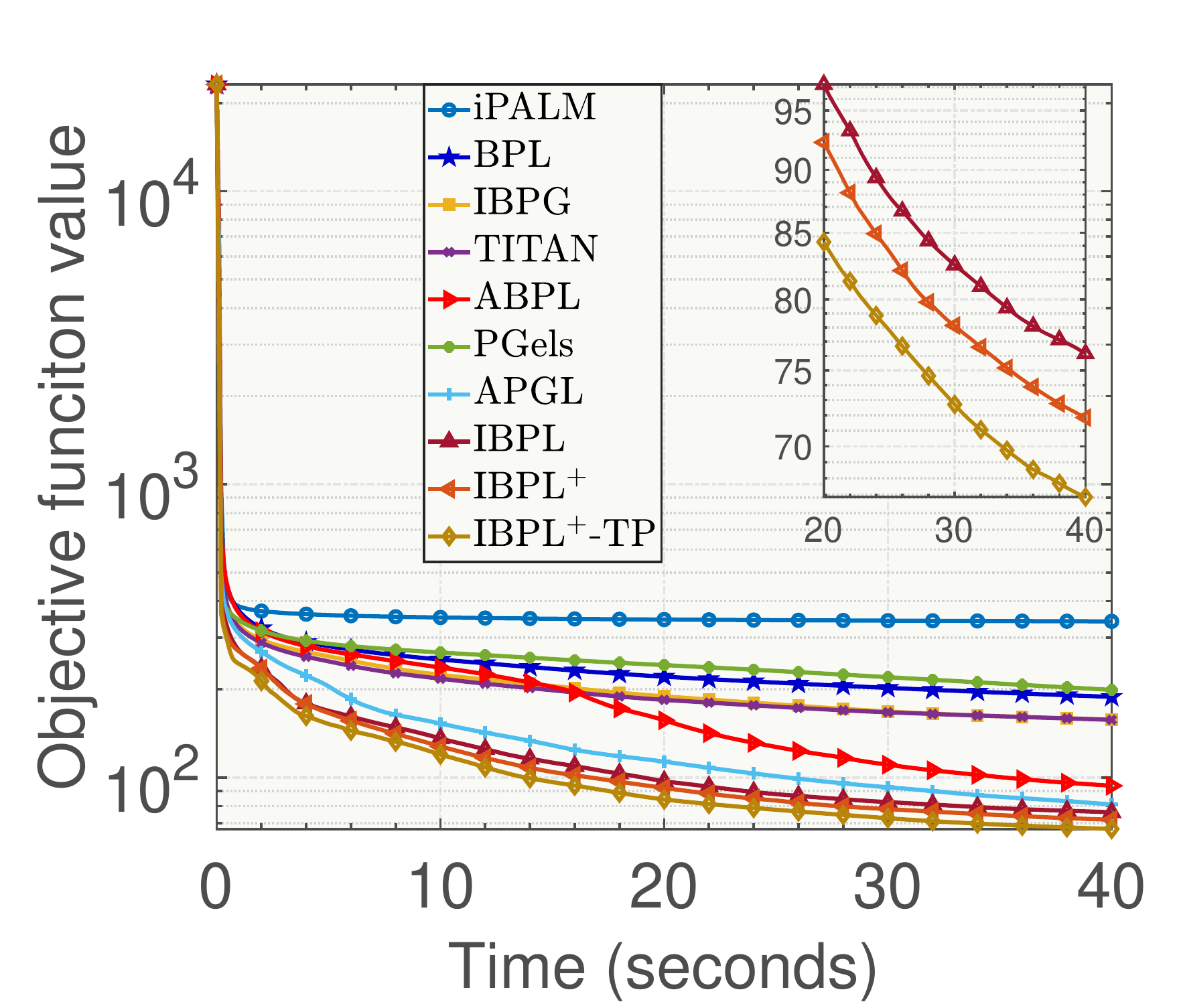}
    \includegraphics[width=0.33\linewidth]{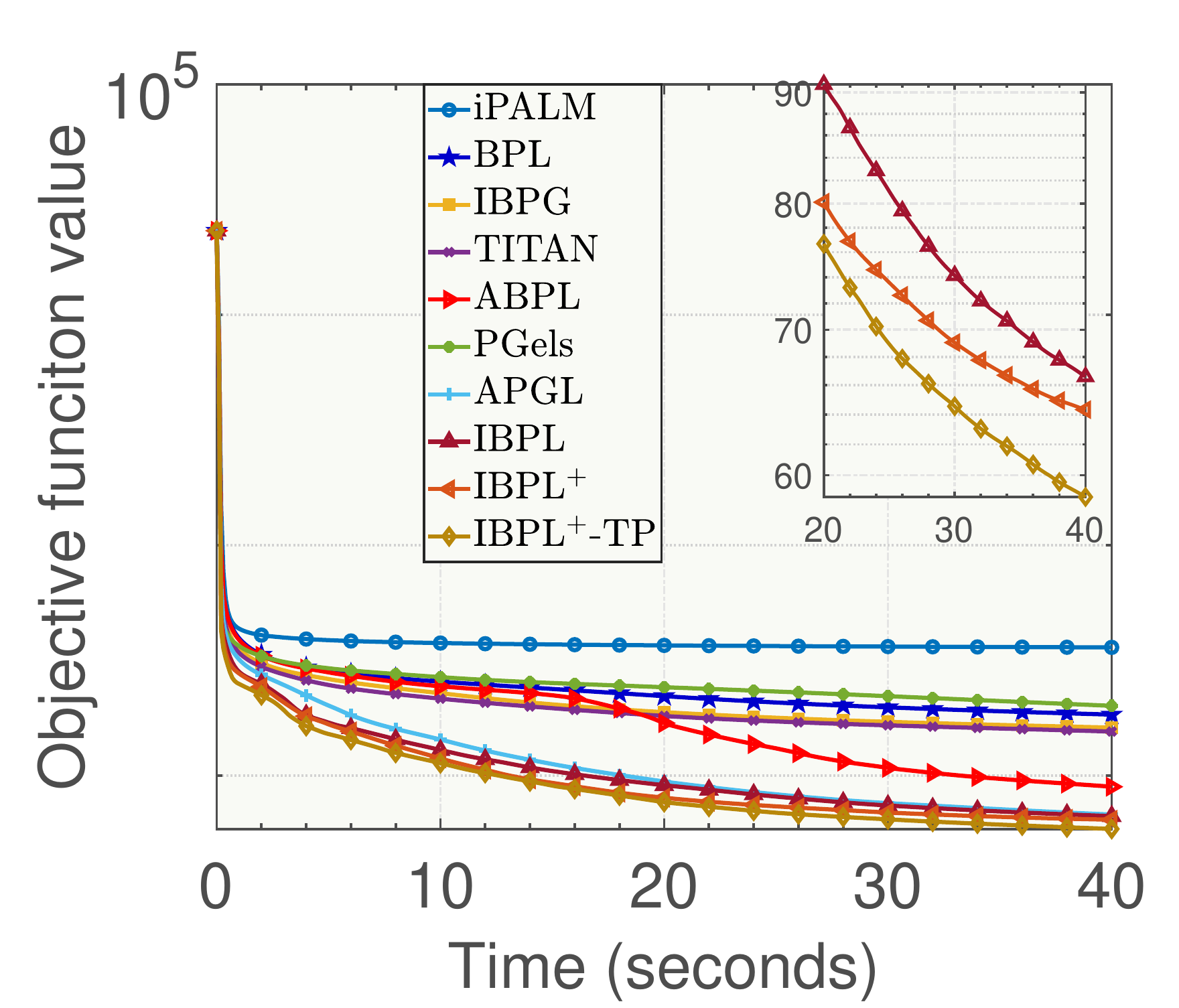}
    \caption{ 
     Average convergence behavior of all methods on microPNW dataset with  $R=50~$(left), $R=60~$(middle) and $R=70~$(right). A partially enlarged view of several curves with the fastest descending speed and the best effect is also provided. 
    }\label{figten3}
\end{figure*}

\subsection{Sparse nonnegative CP decomposition with $\ell_0$-constraints ($\ell_0$-SNCP)}
\subsubsection{$\ell_0$-SNCP model}
The NCP (nonnegative CP decomposition) is an important tool for analyzing nonnegative multi-way data \cite{liu2025inertial,chen2022unsupervised,mohaoui2024cp}. 
Similarly, since the $\ell_{0}$ norm is the most intuitive representation of sparsity, the $\ell_0$ norm is better suited for feature selection in NCP. 
However, optimizing the Sparse NCP with $\ell_0$ norm constraints ($\ell_0$-SNCP) is known to be NP-hard, nonsmooth and nonconvex.  
To address this challenge, we will also design a practical convergent scheme for solving $\ell_0$-SNCP based on our method and apply this scheme to solve the problem.

Mathematically, the $\ell_0$-SNCP model can be presented as follows. 
\begin{align}
&& &\min\limits_{\left\{A_{i}\right\}_{i=1}^{N}}\frac{1}{2}\|\mathcal{X}-{\llbracket {\left\{A_{i}\right\}_{i=1}^{N}}  
 \rrbracket }\|_{F}^{2},  \notag \\
&& &s. t.~ A_{i} \geq 0, ~ \|A_{i}\|_{0} \leq s_{i}, 
\label{e43}
\end{align}
where $\mathcal{X} \in \prod_{i=1}^{N} \mathbb{R}^{d_{i}}$ and $A_{i}\in\mathbb{R}^{d_{i}\times R}$, $ \llbracket \ \rrbracket $ represents Kruskal operator, $\odot$ represents the Khatri-Rao product.

\subsubsection{Solving $\ell_0$-SNCP using IBPL$^+$-TP}
Similarly, if we write Eq. (\ref{e43}) in the form of Eq. (\ref{e11}), then we have 
\begin{align}
    && H(\left\{{A_{i}} \right\}^{N}_{i=1})&=\frac{1}{2}\|\mathcal{X}-{\llbracket {\left\{A_{i}\right\}_{i=1}^{N}} \rrbracket }\|_{F}^{2}, ~F_{i}(A_{i})&=\delta(A_{i}). 
    \label{SNCPEq}
\end{align}
The definition of $\delta(A_{i})$ is the same as Eq. (\ref{delta}). 

Let $X^{(n)} \in  \mathbb{R}^{d_{n} \times \prod_{i=1, i \neq n}^{N} d_{i}}$ represent the mode-n unfolding of the original tensor $\mathcal{X}$, the mode-n unfolding of the   $ \llbracket {\left\{A_{i}\right\}_{i=1}^{N}} \rrbracket $ can be written as $A_{n}B_{n}^{T}$, where  $B_{n}= \prod_{j=1,j\neq n}^{N}\odot A_{j}$. Therefore, in the BCD frame, Eq. (\ref{e43}) can be decomposed into the sub-problems as follows. 
\begin{center}
     $\min\limits_{A_{i}}J_{i}(A_{i})=\frac{1}{2}\|X^{(i)}-A_{i}(B_{i})^{T}\|_{F}^{2}$, 
     \\
$ s. t.~ A_{i} \geq 0, ~ \|A_{i}\|_{0} \leq s_{i}$, 
\end{center}
where  $B_{i}= \prod_{j=1,~j\neq i}^{N}\odot A_{j}$, so the gradient of $H(\left\{{A_{i}} \right\}^{N}_{i=1})$  is 
\begin{align}
&& \nabla_{A_{i}} H(\left\{{A_{i}} \right\}^{N}_{i=1})&=\frac{\partial}{\partial A_{i}}J_{i}(A_{i}) \notag\\
&& &=A_{i}(B_{i})^{T}(B_{i})-X^{(i)}B_{i}. \label{sncpgrad}
 \end{align}
By  Eq. (\ref{sncpgrad}), it is easy to see that the Lipschitz constant of $\nabla_{A_{i}} H(\left\{{A_{i}} \right\}^{N}_{i=1})$ is  
\begin{center}
$L(A_{i})=\|(B_{i})^{T}B_{i}\|_{F}$,      
\end{center}
where  $B_{i}= \prod_{j=1,~j\neq i}^{N}\odot A_{j}$. 
As a result, Eq. (\ref{SNCPEq}) also satisfies Assumption \ref{assump1}, and it is easy to verify that the proximal operator (Eq. (\ref{iprox})) for solving the $\ell_0$-SNCP problem can be expressed as follows. 
\begin{align}
\notag
\end{align}
\begin{align}
 && A_{i}^{k+1} &\in prox_{\sigma_{i}^{k} F_{i}}
 (U^{k}_{i}) \notag \\
 && &=\operatorname*{\arg\min} \limits_{A}\left\{\frac{1}{2\sigma_{i}^{k}}\|A-U^{k}_{i}\|_{F}^{2}+F_{i}(A) \right\}  \notag \\ 
 && &= \operatorname*{\arg\min} \limits_{A}\left\{\|A-U^{k}_{i}\|_{F}^{2}: A\geq 0,~\|A\|_{0}\le s_{i}\right\},  \notag
\end{align}
where 
\begin{center}
$U^{k}_{i}=z^{k}_{i}-\sigma_{i}^{k} \nabla_{A_{i}} H(\{x^{k+1}_{j}\}_{j=1}^{i-1},y^k_{i},\{x^{k}_{j}\}_{j=i+1}^{N})$, 
$\frac{1}{\sigma_{i}^{k}}=\gamma^{k}_{i}L(A^{k}_{i}), ~ \gamma^{k}_{i}>1$. 
\end{center}

\subsubsection{Numerical results}
We test the methods on the BreastMNIST\footnote{\url{https://github.com/MedMNIST/MedMNIST}}  \cite{yang2021medmnist}, PosteriorEmission\footnote{\url{https://zenodo.org/record/7646462}} \cite{tichy_ondrej_2023_7646462}, and microPNW\footnote{\url{https://github.com/niyiyu/PNW-ML}} \cite{ni2023pnw} datasets. 
The BreastMNIST dataset is a medical dataset in the form of the three-dimensional tensor  ($\mathbb{R}^{780\times224\times224}$), the PosteriorEmission dataset is a remote sensing dataset in the form of the four-dimensional tensor ($\mathbb{R}^{240\times40\times12\times8}$).  
The microPNW dataset is a micro version of the Pacific Northwest Curated Seismic dataset, each sample in the microPNW dataset is transformed into the form of channels $\times$ time frames $\times$ frequency bins, thus the microPNW dataset is a four-dimensional tensor ($\mathbb{R}^{100\times150\times3\times50}$), and we also normalize the value of each feature in the microPNW dataset to the range of 0 to 1.

For parameter settings, we set the initial parameters as 
$t_{2}=1.3, ~ \beta^{1}_{i}=0.2,~ \alpha^{k}_{i}=\min(1.03\beta^{k}_{i},\alpha_{max}), ~ \alpha_{max}=\beta_{max}=0.9999, ~ \gamma^{k}_{i}=1.01,~\rho_{1}=10^{-5}$, $T=10^{-3}$, and the initial matrix is also generated by the uniform distribution. Similarly, the number of non-zero elements in each matrix cannot exceed $30\%$ of the total number of elements. We also define Obj as the objective function value, relative error (Rel) as Rel$=\frac{\|\mathcal{X}-{\llbracket  \{ A_{i}\}_{i=1}^N  \rrbracket }\|_{F}}{\|\mathcal{X}\|_{F}}$. 
$Ranking$ is also defined as the number of times that the corresponding method obtains the lowest relative error among all methods across all independent runs.  We also adopt these metrics to quantitatively describe the numerical performance of all methods in this experiment.

We set $t_{max}=40$ in this experiment. 
Table \ref{Tabten} and Table \ref{Tabten2} report the average results over 20 independent runs with $R$ varying among $\{50,60,70\}$ on these datasets. Figure \ref{figten1}, Figure \ref{figten2} and Figure \ref{figten3} also record the evolution of average objective function value over 20 independent runs with respect to time on these datasets.  
From Table \ref{Tabten}, Table \ref{Tabten2}, Figure \ref{figten1}, Figure \ref{figten2} and Figure \ref{figten3}, it is easy to get the same conclusion that     
\begin{itemize}

    \item [1)] \textit{IBPL outperforms all other methods except IBPL$^+$ and IBPL$^+$-TP in all cases}: 
     This shows that using two different extrapolation points and the independence of extrapolation parameters of these extrapolation points can effectively enhance the numerical performance.   

    \item [2)] \textit{IBPL$^+$ outperforms all other methods except IBPL$^+$-TP in all cases}: 
    This once again demonstrates that when using two different extrapolation points with independent extrapolation parameters, incorporating an adaptive momentum update strategy for the extrapolation parameters into the proposed method can indeed further accelerate convergence and improve numerical performance. 

   \item [3)] \textit{IBPL$^+$-TP still remarkably outperforms all other methods in all cases}: 
    This not only reaffirms the benefits of using two different extrapolation points with independent extrapolation parameters and the adaptive momentum update strategy for the extrapolation parameters, but also shows that introducing a two-phase adaptive momentum update strategy, which explicitly separates the rapid descent phase from the steady refinement phase to update the extrapolation parameters effectively, provides an additional boost in convergence speed and numerical robustness, making IBPL$^+$-TP the most effective method overall.

\end{itemize}

\begin{figure*}[!ht]
    \centering 
    \subfloat[ \centering $\cos \theta^k_{\min}$ by the methods on lp\_ship12l(left) and BreastMNIST(right) datasets. ] 
    {
    \includegraphics[width=0.45\linewidth]{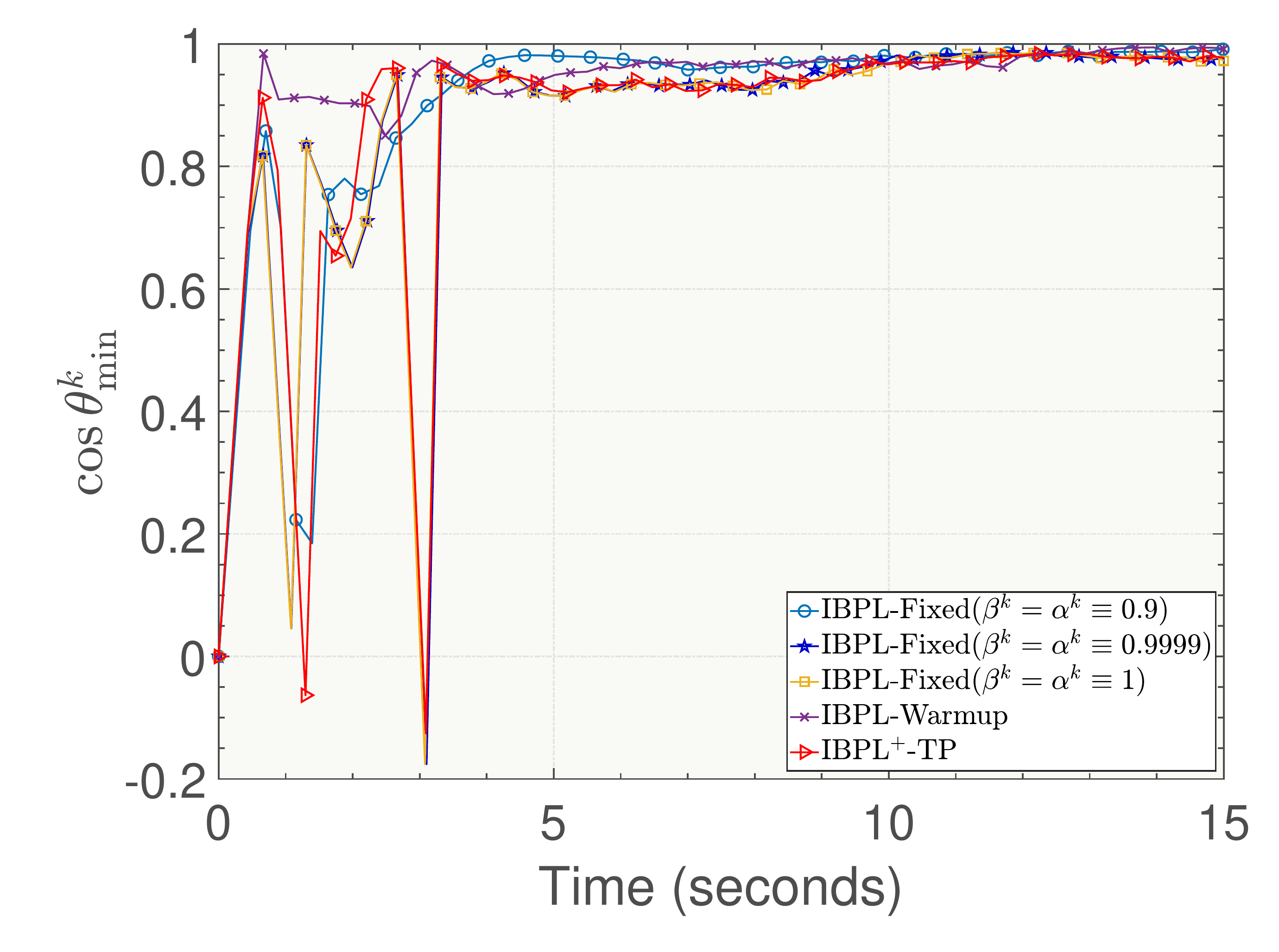}
    \includegraphics[width=0.45\linewidth]{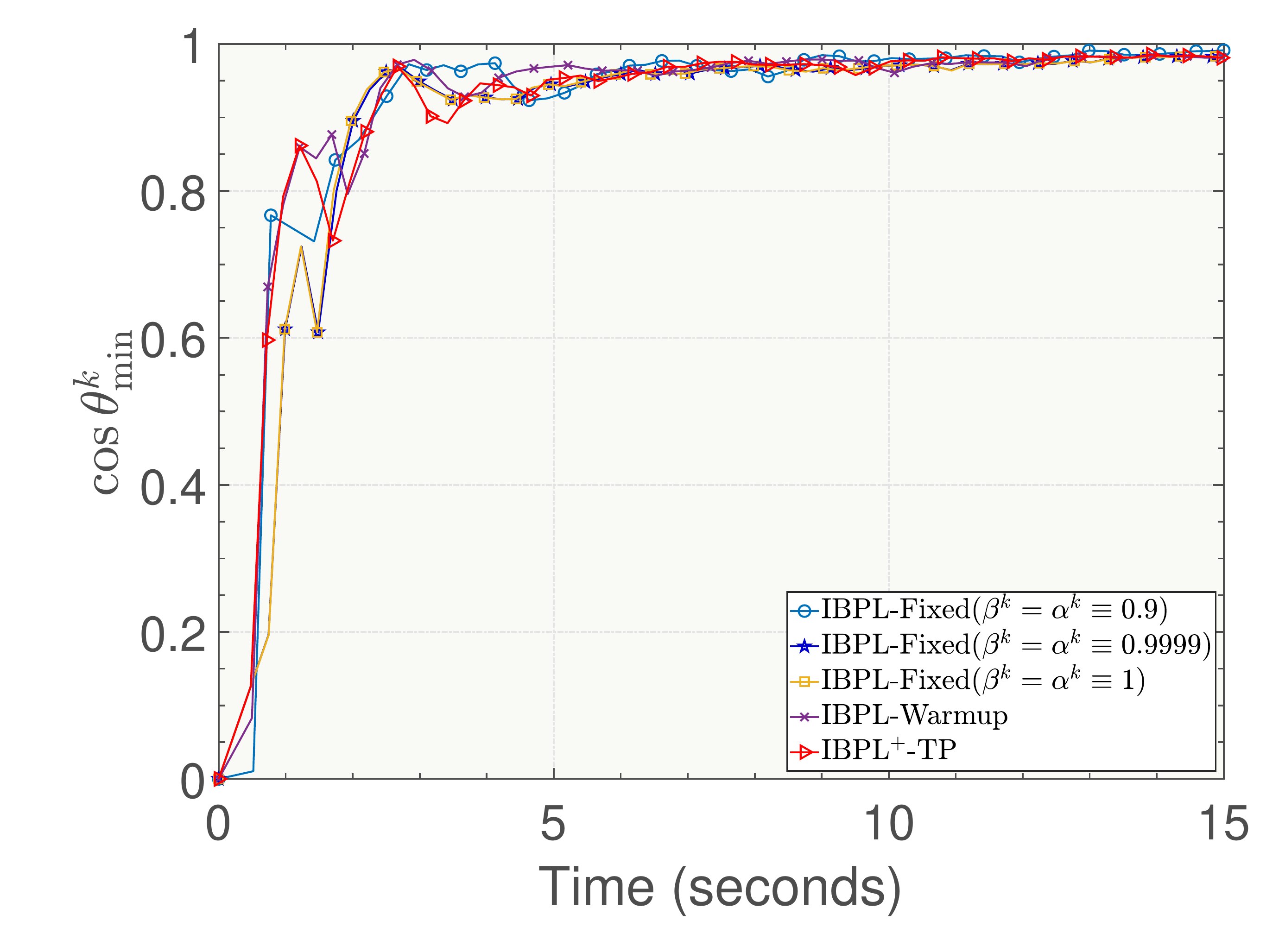}
    }

    \subfloat[ \centering  Objective function value by the methods on lp\_ship12l(left) and BreastMNIST(right) datasets.  A partially enlarged view of several curves with the fastest descending speed and the best effect is also provided.]
    {
    \includegraphics[width=0.45\linewidth]{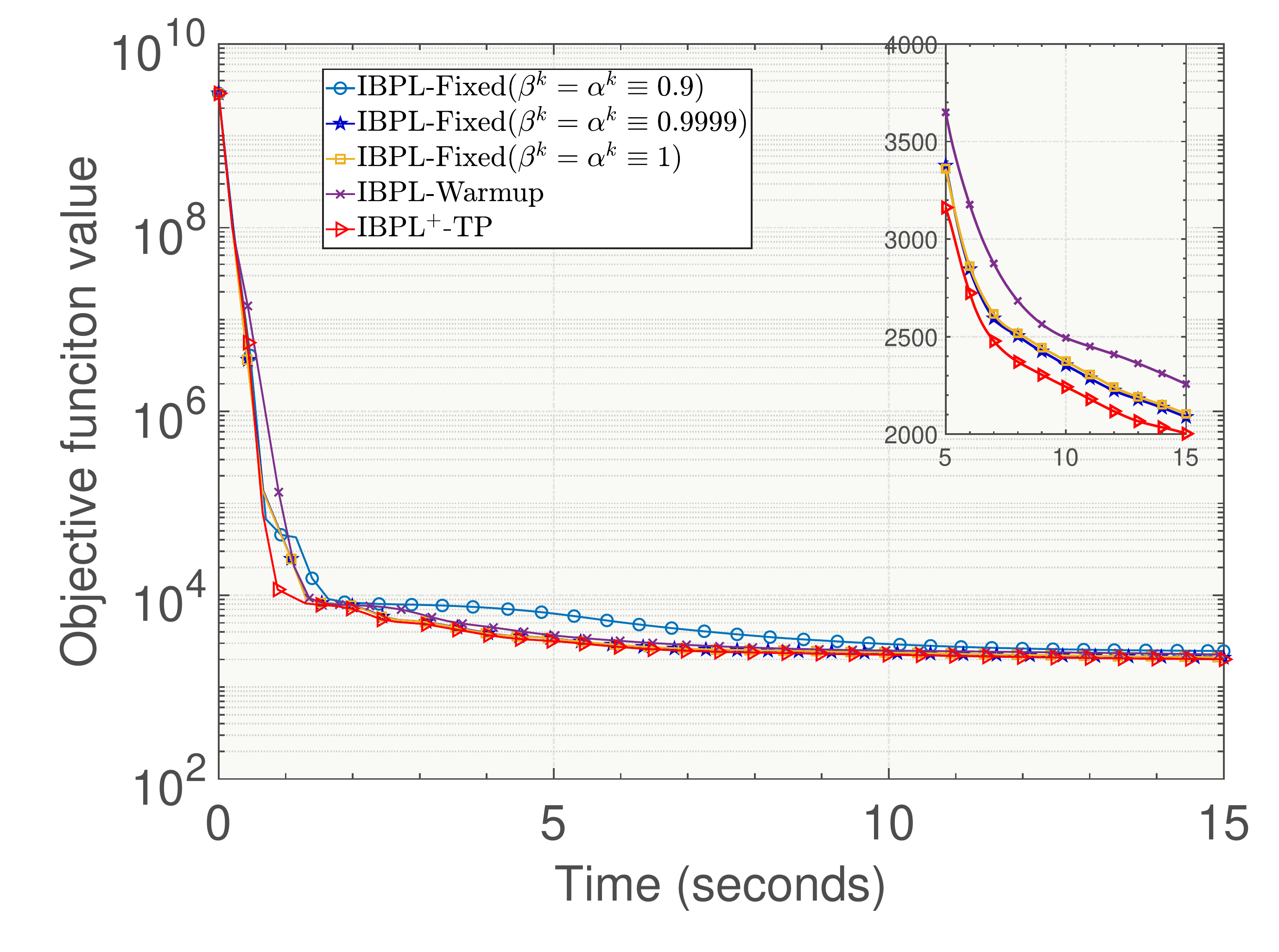}
    \includegraphics[width=0.45\linewidth]{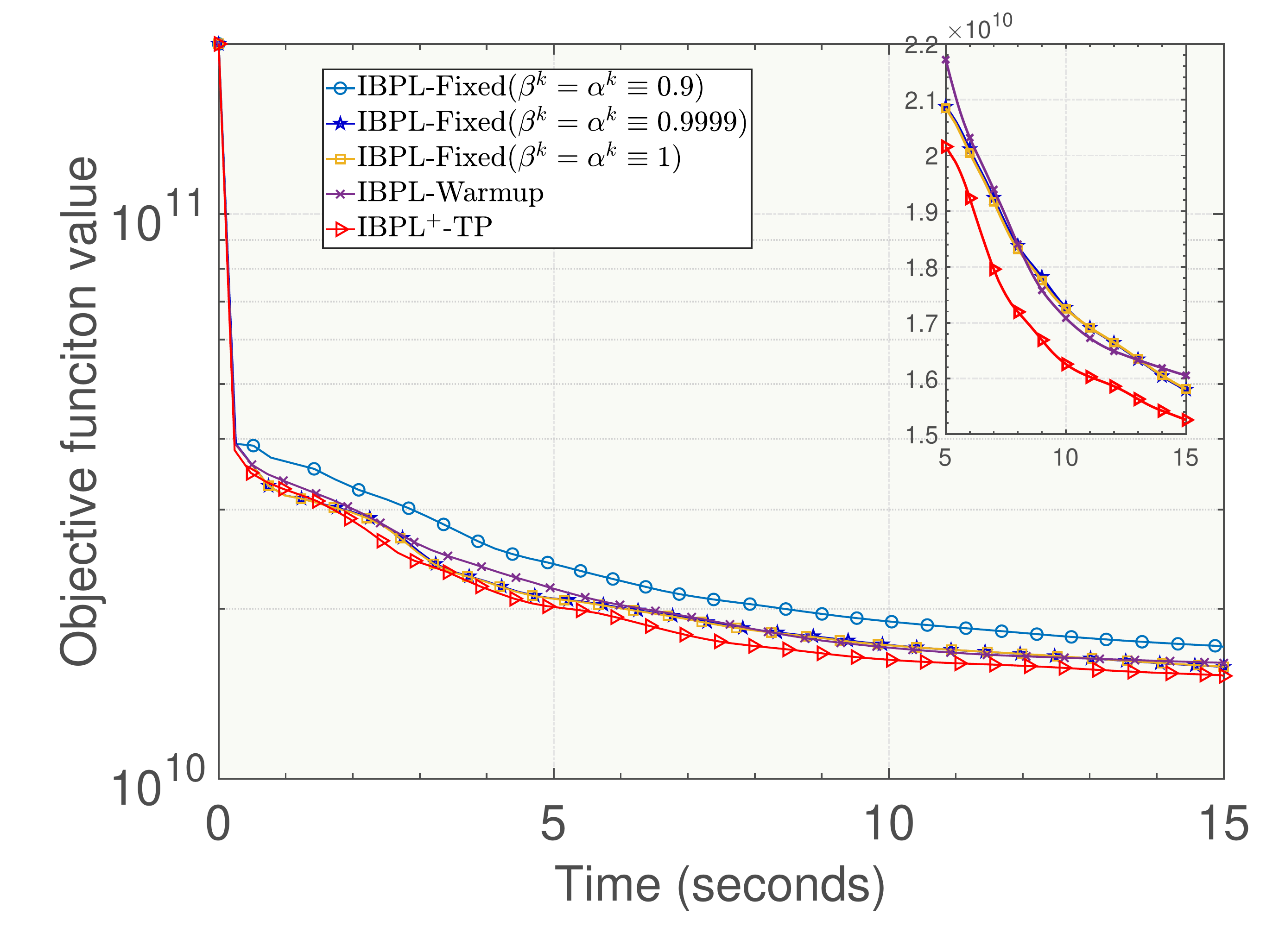}
    }

    \subfloat[\centering  $\|x_{(N)}^{k+1}-x_{(N)}^k\|_F$ by the methods on lp\_ship12l(left) and BreastMNIST(right) datasets.   ]
    {
    \includegraphics[width=0.45\linewidth]{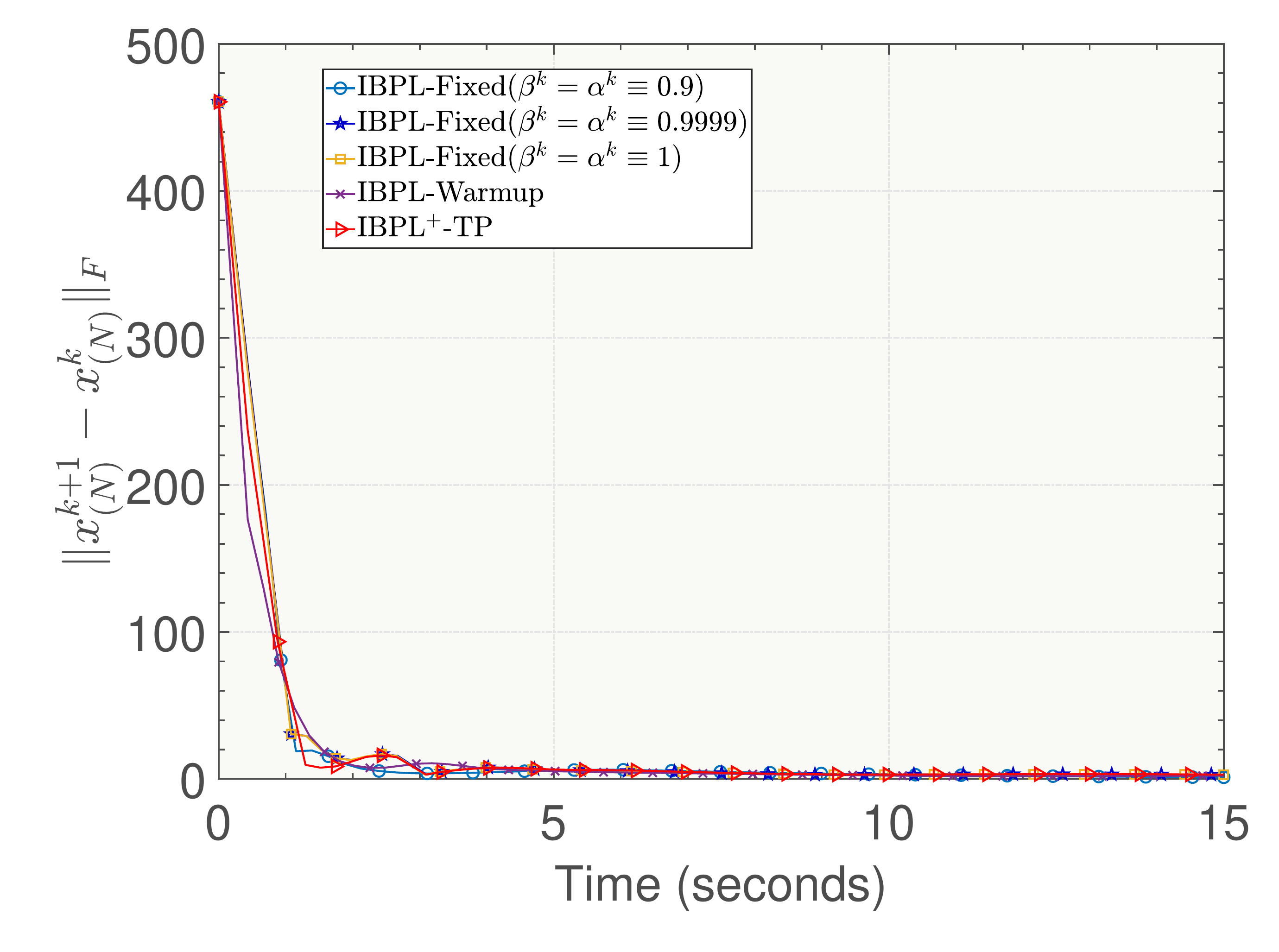}
    \includegraphics[width=0.45\linewidth]{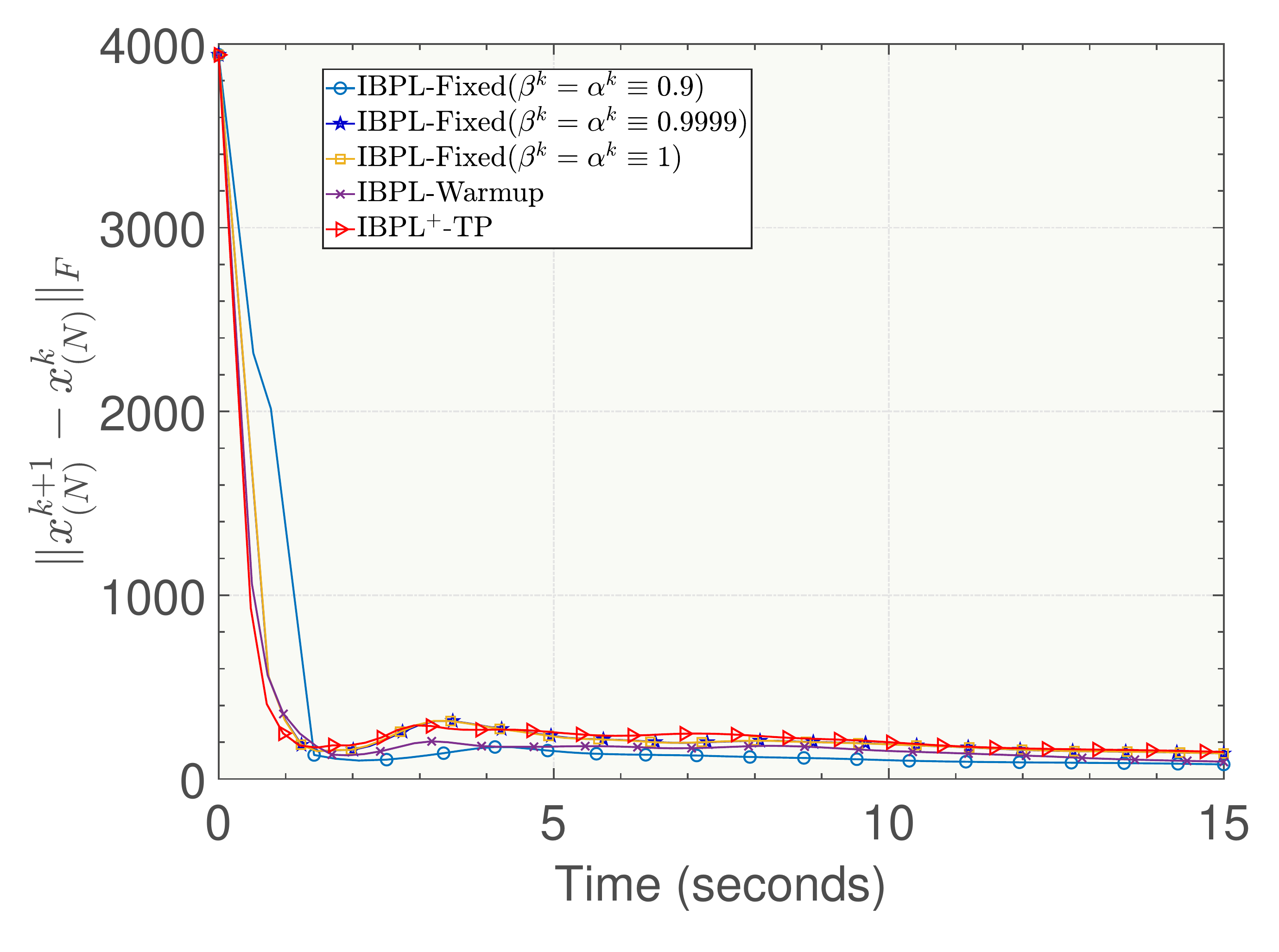}
    }  
    
    \caption{\centering Numerical stability analysis of the proposed method on the lp\_ship12l and BreastMNIST datasets. }
    
    \label{fignum}
\end{figure*}

\subsection{Numerical stability and warm-up analysis}
Next, we investigate whether a warm-up phase is needed to maintain numerical stability of the proposed method and analyze the impact of the extrapolation parameters on numerical stability, especially when the extrapolation parameters take the extreme values that are close to or even exceed the theoretical bound required for global convergence. 

Specifically, first, to quantitatively characterize the iterative behavior of the $\{x_{(N)}^k\}_{k\in \mathbb{N}}$, we define cosine similarity as $\cos \theta_i^k=\frac{\langle x_{i}^{k+1}-x_{i}^{k},x_{i}^{k}-x_{i}^{k-1}\rangle}{\|x_{i}^{k+1}-x_{i}^{k}\|_F\|x_{i}^{k}-x_{i}^{k-1}\|_F}$, which can directly measure the directional consistency of the iterative trajectory of the $i$-th variable block. The values of $\cos \theta_i^k$ close to $-1$ indicate that the iterative trajectory of the $i$-th variable block exhibits a sharp directional change (zig-zagging), and the values of $\cos \theta_i^k$ near $0$ indicate orthogonal steps, corresponding to frequent direction changes. Conversely, the values of $\cos \theta_i^k$ close to $1$ indicate that the iterative trajectory of the $i$-th variable block is moving in a consistent direction (smooth progress). Building on this, to intuitively assess the overall stability of the iterative trajectory, we introduce the worst-case cosine similarity metric $\cos \theta^k_{\min} = \min (\{ \cos \theta_i^{k}\}_{i=1}^N)$, which captures the worst-case (i.e., least consistent) update direction among the variables. 


Second, we propose two transformation forms of the IBPL$^+$-TP to investigate whether a warm-up phase is needed and analyze the impact of the extreme extrapolation parameter values on the numerical stability, as follows. 

\begin{itemize} 
    \item [1)]IBPL-Fixed: it is a version of IBPL$^+$-TP with fixed extrapolation parameter values and $T<0$, which means that the extrapolation parameters $\alpha^{k}_{i}$ and $\beta^{k}_{i}$ are set to constant values, and these constant values are even allowed to exceed the theoretical bound required for global convergence.   Specifically, for IBPL-Fixed in this experiment, we vary the extrapolation parameter values among $\{0.9, 0.9999,1\}$ to systematically analyze their impact on numerical stability.  
 
    \item [2)]IBPL-Warmup: it is a version of IBPL$^+$-TP with a warm-up phase for the extrapolation parameters, which means that the extrapolation parameters $\alpha^{k}_{i}$ and $\beta^{k}_{i}$ are set as $t_{1}=1,~t_{k+1}=\frac{1+\sqrt{1+4t_{k}^{2}}}{2}, ~ \alpha^{k}_{i}=\beta^{k}_{i}=\frac{t_{k}-1}{t_{k+1}}$.  

\end{itemize}

Then, we evaluate the numerical performance of these two transformation forms on both the lp\_ship12l dataset for the $\ell_0$-SNMF problem and the BreastMNIST dataset for the $\ell_0$-SNCP problem. We set $r=300$ for $\ell_0$-SNMF and $R=50$ for $\ell_0$-SNCP as well as $t_{max}=15$ in this experiment, all other parameters remain unchanged. 
Figure \ref{fignum} illustrates the $\cos \theta^k_{\min}$, objective function value and $\|x_{(N)}^{k+1}-x_{(N)}^k\|_F$ with respect to time on the lp\_ship12l and the BreastMNIST datasets. From Figure \ref{fignum}, we have the following observations.

\begin{itemize}
    \item [1)] For the IBPL-Fixed with the large extrapolation parameter values (i.e., $0.9999, 1$) and IBPL$^+$-TP, $\cos \theta^k_{\min}$ exhibits one or two obvious sharp oscillations on the lp\_ship12l dataset, while the $\cos \theta^k_{\min}$ of IBPL-Warmup remains stable without significant oscillations on the lp\_ship12l dataset. However, on the lp\_ship12l dataset, the convergence speed of both IBPL-Fixed with the large extrapolation parameter values and IBPL$^+$-TP is faster than that of the IBPL-Warmup. Besides, on the BreastMNIST dataset, $\cos \theta^k_{\min}$ of all methods does not oscillate, and the convergence speed of the IBPL$^+$-TP is still faster than that of the IBPL-Warmup, but the convergence speed of the IBPL-Fixed with the large extrapolation parameters is comparable to that of IBPL-Warmup.

    \item [2)] Furthermore, after approximately 5 seconds, the $\cos \theta^k_{\min}$  of all methods stabilizes and converges to 1 on both datasets, and $\|x_{(N)}^{k+1}-x_{(N)}^k\|_F$ of all methods also stabilizes and converges to 0 on both datasets. 

    \item [3)] Throughout the optimization, the objective function value decreases monotonically for all methods on both datasets, irrespective of the extrapolation parameter values (even when the extrapolation parameter values exceed the theoretical bound required for global convergence). This empirical finding aligns with Theorem \ref{t1}, that is, the monotonic convergence of Algorithms  \ref{IBPL$^+$} and \ref{IBPLtp} holds independently of the extrapolation parameters. 
\end{itemize} 

From the above observations,  we conclude that taking larger or extreme extrapolation parameter values, even those values exceeding theoretical limits, may risk transient instability of the early iterative trajectory but does not impair the convergence speed, while warm-up ensures stability of the iterative trajectory at the potential cost of convergence speed. Furthermore, although IBPL$^+$-TP occasionally exhibits oscillations in the early iterative trajectory, the convergence speed of IBPL$^+$-TP is consistently faster than that of IBPL-Warmup, which further confirms the effectiveness of our proposed method. 
Therefore, warm-up is not strictly necessary for our proposed method, and the proposed method is self‑stabilizing.  
Moreover, these observations experimentally validate the correctness of Theorem \ref{t1}, i.e., the monotonic convergence of Algorithm \ref{IBPL$^+$} and Algorithm \ref{IBPLtp} indeed holds independently of the extrapolation parameters.

\begin{figure*}[!ht]
\centering
    \includegraphics[width=0.47\linewidth]{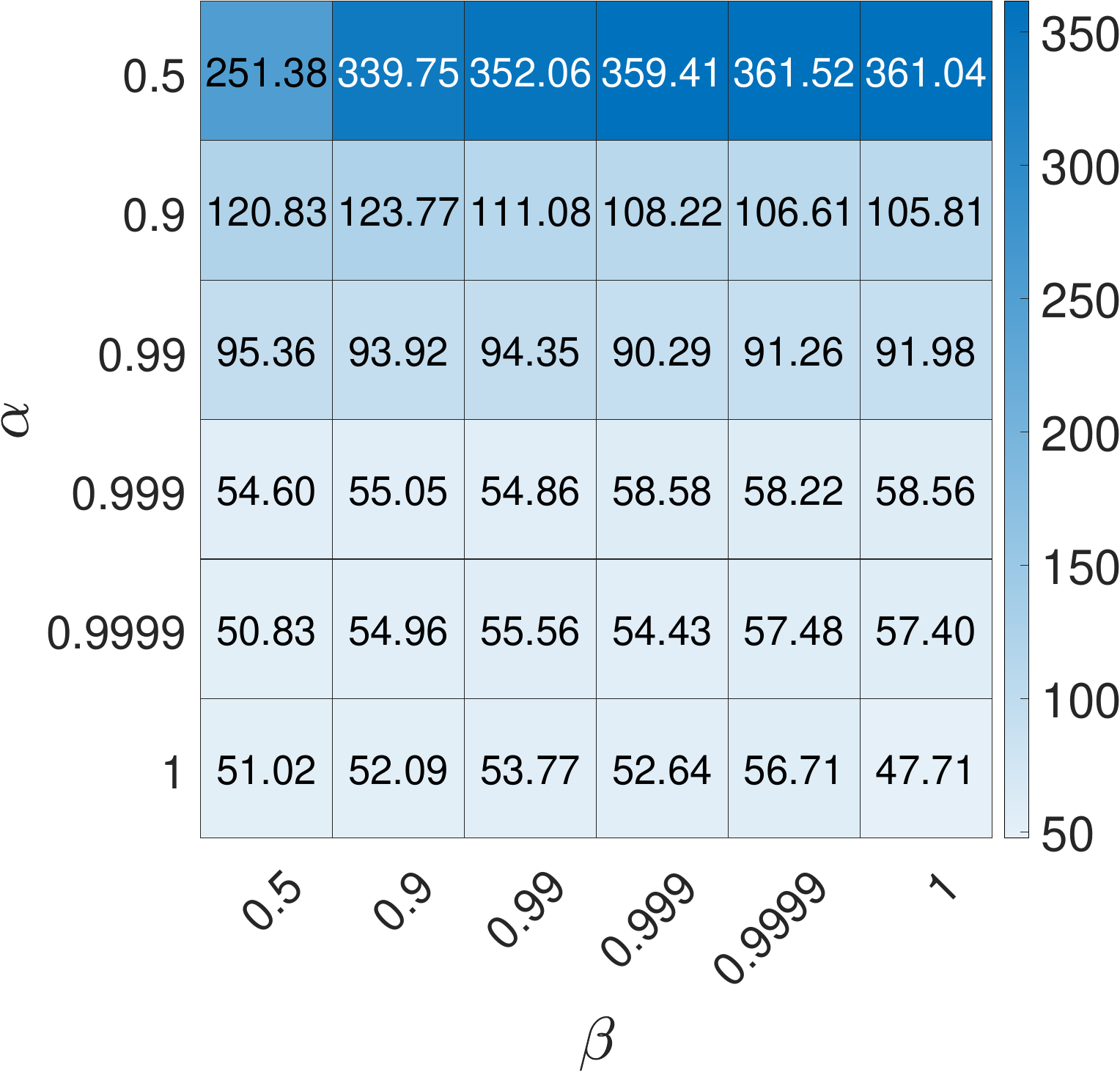} \hspace{4mm}
    \includegraphics[width=0.47\linewidth]{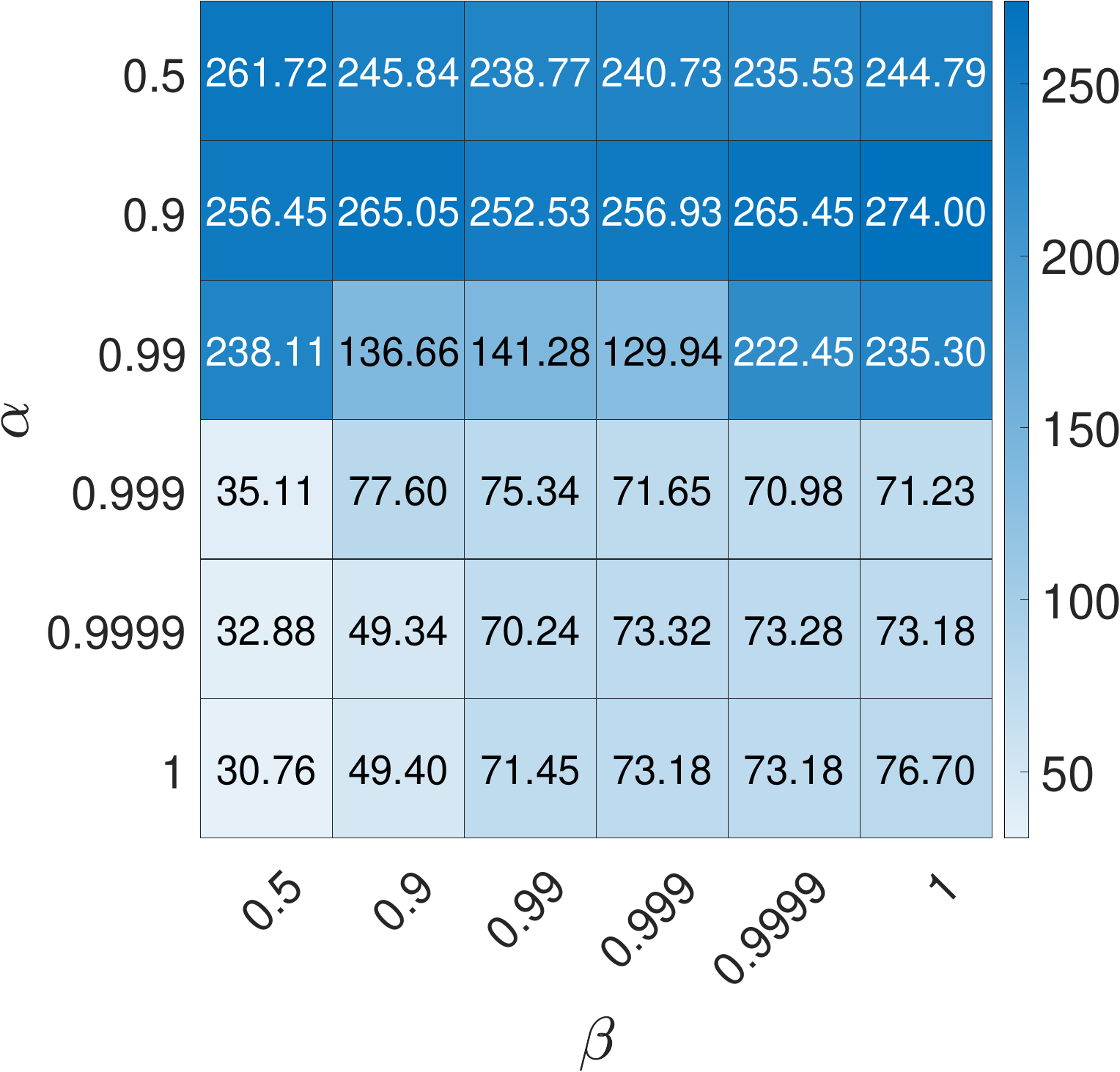}
    \caption{{Average convergence time with respect to the extrapolation parameters $\alpha$ and $\beta$ on the lp\_ship12l (left) and BreastMNIST (right) datasets, respectively. }  
    } \label{figsensitivity}
\end{figure*}

\subsection{Parameter sensitivity analysis}
Next, we evaluate the convergence performance of the proposed method under various extrapolation parameter settings. Specifically, we consider a discrete set of values for the extrapolation parameters $\alpha$ and $\beta$, ranging from conservative to extreme: 
$\{0.5, 0.9, 0.99, 0.999, 0.9999, 1.0\}$. 
For each combination $(\alpha, \beta)$ in this $6\times 6$ grid, we run the IBPL-Fixed algorithm (defined in the previous section) on the lp\_ship12l dataset for $\ell_0$-SNMF problem and on the BreastMNIST dataset for $\ell_0$-SNCP problem, and record the time required to reach the stopping criterion $\frac{|J(x^{k+1})-J(x^k)|}{|J(x^0)|} < 10^{-6}$. 
We set $r=300$ for $\ell_0$-SNMF and $R=50$ for $\ell_0$-SNCP in this experiment, and all other parameters remain unchanged. The average convergence time results over 10 independent runs are summarized in the heatmaps of Figure \ref{figsensitivity}, where the color hue of the corresponding entry indicates the average convergence time when the stopping criterion is met.

From Figure \ref{figsensitivity}, we have the following observations. 

\begin{itemize}
    \item [1)] For the lp\_ship12l dataset, the convergence performance of the proposed method generally shows a positive correlation with $\alpha$. 
    However, the convergence performance is relatively insensitive to  $\beta$ when $\alpha\geq 0.9$.  
    Notably, the convergence performance achieves the best results when $\alpha = \beta = 1$.

    \item [2)] For the BreastMNIST dataset, the convergence performance of the proposed method also generally shows a positive correlation with $\alpha$. 
    However, the convergence performance is relatively sensitive to $\beta$ when $\alpha\geq 0.99$ and $\beta$ is small.    
    Notably, the convergence performance achieves the best results when $\alpha =1$ and $\beta = 0.5$. 
\end{itemize}

\begin{figure*}[!ht]
\centering
    \includegraphics[width=0.47\linewidth]{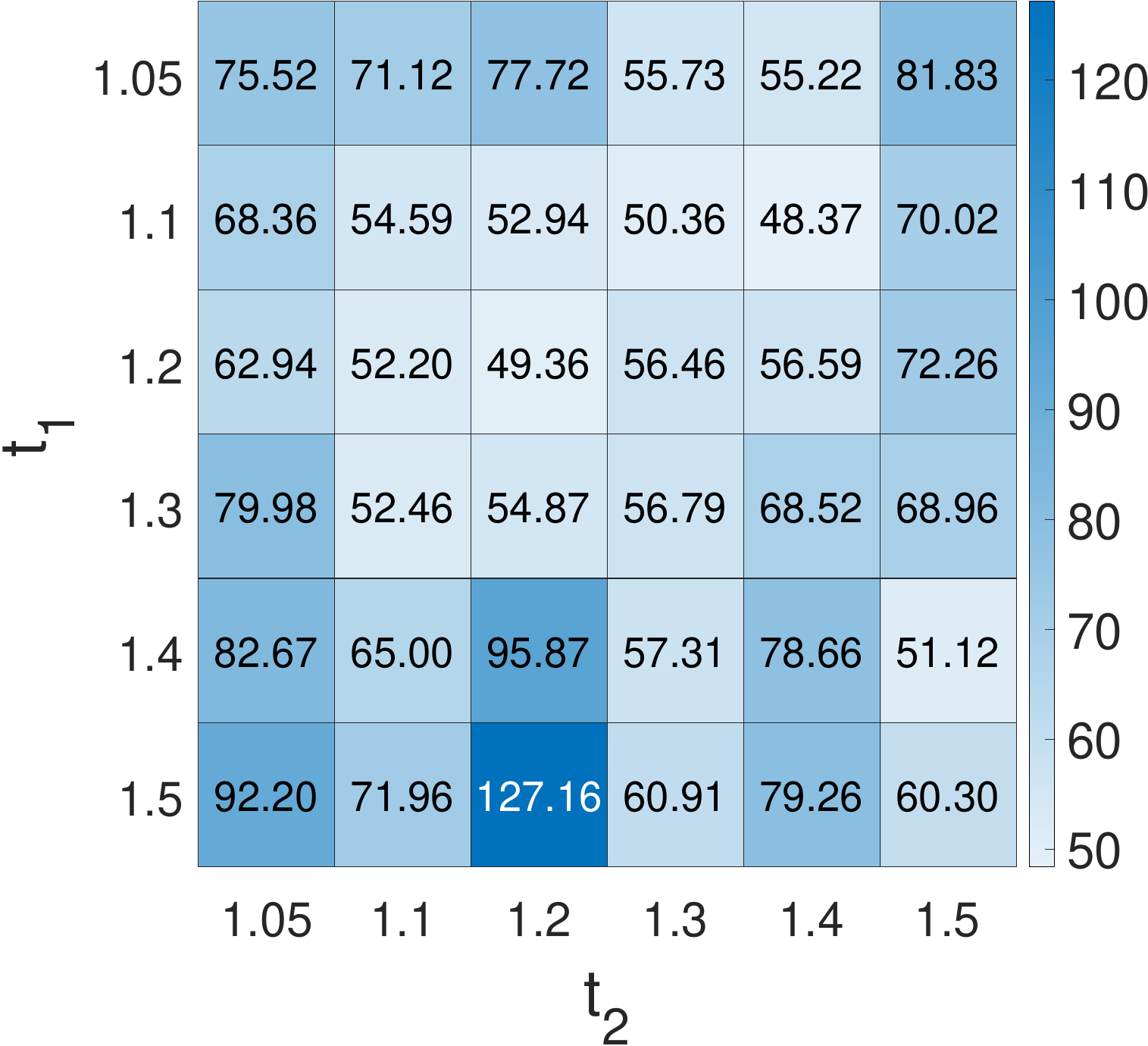} \hspace{4mm}
    \includegraphics[width=0.47\linewidth]{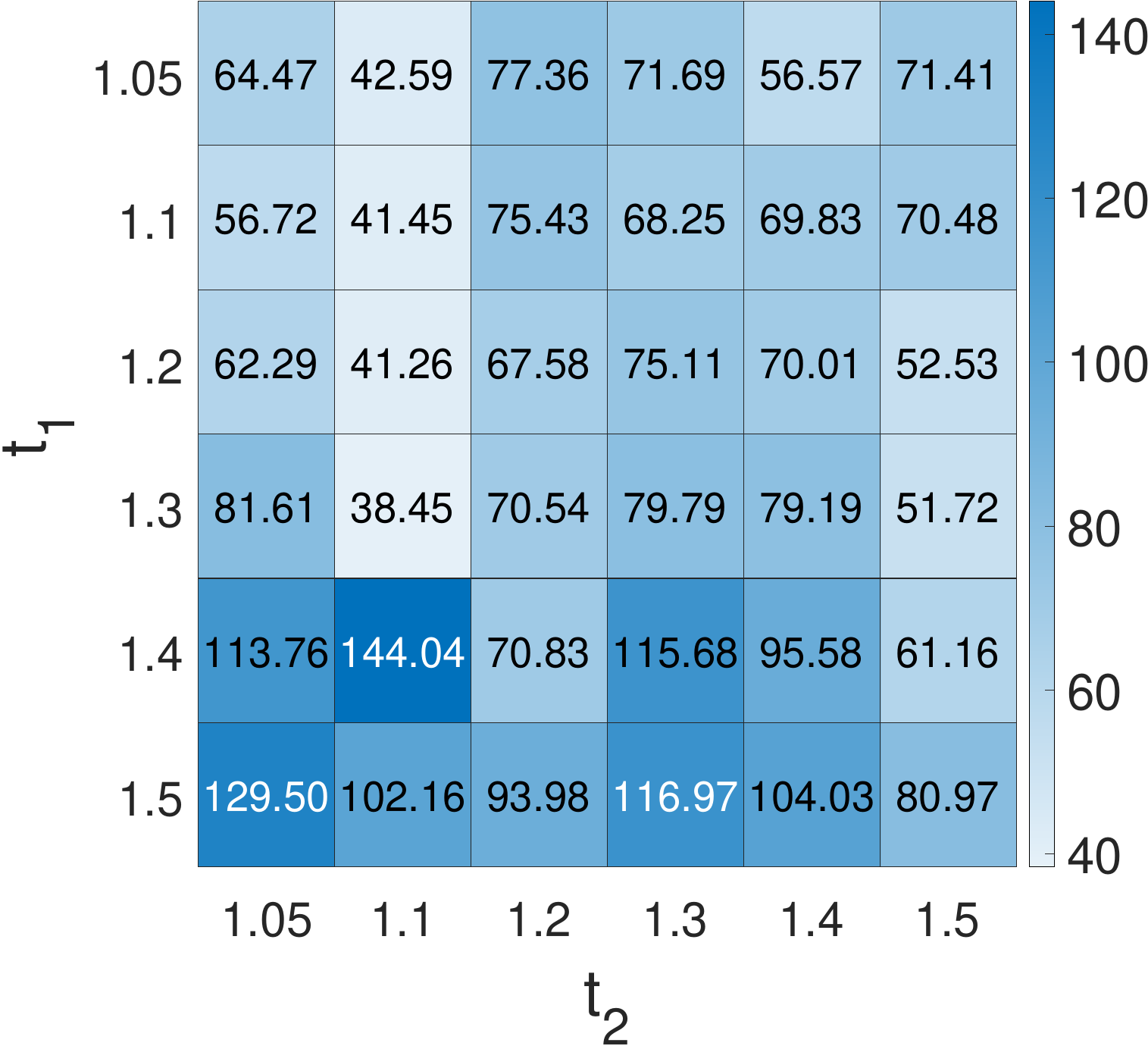}
    \caption{ Average convergence time with respect to the parameters $t_{1}$ and $t_{2}$ on the lp\_ship12l (left) and BreastMNIST (right) datasets, respectively. }  
     \label{figsensitivity1}
\end{figure*}

Additionally, we also analyze the impact of the parameters $t_{1}$ and $t_{2}$ in our proposed method on convergence performance. Specifically, we consider a discrete set of values for the parameters $t_1$ and $t_2$: 
$\{1.05, 1.1, 1.2, 1.3, 1.4, 1.5\}$. 
For each combination $(t_1, t_2)$ in this $6\times 6$ grid, we run the IBPL$^+$-TP algorithm on the lp\_ship12l dataset for $\ell_0$-SNMF problem and on the BreastMNIST dataset for $\ell_0$-SNCP problem, and record the time required to reach the stopping criterion $\frac{|J(x^{k+1})-J(x^k)|}{|J(x^0)|} < 10^{-6}$. All other parameters remain unchanged. 
The average convergence time results over 10 independent runs are summarized in the heatmaps of Figure \ref{figsensitivity1}, where the color hue of the corresponding entry also indicates the average convergence time when the stopping criterion is met.

From Figure \ref{figsensitivity1}, we have the following observations. 

\begin{itemize}
    \item [1)] For the lp\_ship12l dataset, the convergence performance of the proposed method generally initially improves and then declines when $t_{1}$ increases. 
    However, when $t_1\geq 1.4$ or $t_2$ is large or small, the convergence performance is relatively sensitive to $t_2$.  
    Notably, the proposed method exhibits robust and favorable convergence performance when $t_1$ and $t_2$ are selected within $[1.1, 1.3]$, with relatively small variation in convergence speed across this range. Outside this region, convergence performance becomes more sensitive to parameter choices, indicating potential numerical instability. Therefore, selecting $t_1$ and $t_2$ within this stable region is recommended for practical use, which also justifies our experimental settings.

    \item [2)] For the BreastMNIST dataset, the convergence performance of the proposed method also generally initially improves and then declines when $t_{1}$ increases. 
    However, when $t_1\geq 1.4$ or $t_2$ is large or small, the convergence performance is also relatively sensitive to $t_2$.     
      Notably, the proposed method exhibits robust and favorable convergence performance when $t_1$ is selected within $[1.1, 1.3]$ and $t_2$ is selected within $[1.2, 1.4]$, with relatively small variation in convergence speed across this region. Outside this region, convergence performance becomes more sensitive to parameter choices, indicating potential numerical instability. Therefore, selecting $t_1$ and $t_2$ within this stable region is recommended for practical use, which also justifies our experimental settings. 
\end{itemize}

\section{Conclusion}
\label{conclu}

In this paper, we proposed a novel method named the inertial block proximal linearized method with two-phase adaptive momentum (IBPL$^+$-TP) for solving a class of multiblock nonsmooth nonconvex optimization problems. 
The proposed method allows the use of two different extrapolation points and can eliminate the constraint relationship between the extrapolation parameters of these two extrapolation points and other parameters, while also employing a two-phase adaptive momentum strategy that explicitly separates the rapid descent phase from the steady refinement phase to effectively update the extrapolation parameters, thereby improving the flexibility and enhancing the numerical performance. Furthermore, while maintaining the above advantages, we proved that our method ensures the convergence of this class of problems, and we also established both the global convergence and the convergence rate of our method. 
We applied our method to solve the nonconvex and nonsmooth $\ell_0$-SNMF and $\ell_0$-SNCP problems. 
The experimental results on these two problems demonstrated that our proposed method clearly outperforms some state-of-the-art methods. 
Additionally, we investigated both the numerical stability and parameter sensitivity of the proposed method. 
For future work, we plan to propose new mathematical tools to explore how to further relax the basic assumptions of the proposed method to improve the generalization ability and practical utility.

\bibliographystyle{TRR}
\bibliography{refer}

\begin{acks}
The work was supported in part by the CEA Youth Key Project on Earthquake Information (No. CEAITNS202607), the Spark Program of Earthquake Sciences of China Earthquake Administration (XH25033YB), and Earthquake Science Technology Innovation Team Project of Yunnan Province (CXTD202507). 
\end{acks}

\section*{Declaration of conflicting interests}
The author declared no potential conflicts of interest with respect to the research, authorship, and/or publication of this paper.

\section*{Data availability statement}
All data used in this paper are publicly available data, and the URLs for obtaining them are provided in the footnotes on pages 14 and 17, respectively. 

\section*{Author contribution statement}
The author confirms sole responsibility for the following: study conception and algorithm framework design, data collection, convergence analysis, experimental design and analysis, interpretation of experimental results, and manuscript preparation.

\end{document}